\documentclass{amsart}

\usepackage{amssymb,mathtools,hyperref,url,enumitem,upgreek,tikz-cd,centernot}
\usepackage[scr=boondox]{mathalfa}
\hypersetup{colorlinks,linkcolor={blue},citecolor={blue},urlcolor={blue}}
\usepackage[linesnumbered,ruled,noend]{algorithm2e}
\usepackage[style=alphabetic,maxbibnames=99]{biblatex}

\renewbibmacro{in:}{}
\DeclareFieldFormat*{title}{#1}
\setlist[enumerate]{left=2ex}

\newcommand{\commentR}[1]{\tcp*[r]{\parbox[t]{4.5cm}{\raggedright\normalfont{#1}}}}
\newcommand{\commentF}[1]{\tcp*[f]{\parbox[t]{4.5cm}{\raggedright\normalfont{#1}}}}
\SetKwComment{tcp}{}{}
\let\oldnl\nl
\newcommand{\nonl}{\renewcommand{\nl}{\let\nl\oldnl}}
\makeatother

\usepackage{etoolbox} 
\makeatletter
\patchcmd{\algocf@makecaption@ruled}{\hsize}{\textwidth}{}{} 
\patchcmd{\@algocf@start}{-1.5em}{0em}{}{}
\makeatother

\SetAlCapHSkip{0pt} 
\newcommand{\legendre}[2]{\smash{\genfrac{(}{)}{}{}{#1}{#2}}}

\newtheorem{innercustomthm}{Theorem}
\newenvironment{customthm}[1]
  {\renewcommand\theinnercustomthm{#1}\innercustomthm}
  {\endinnercustomthm}

\newtheorem{theorem}{Theorem}[section]
\newtheorem{lemma}[theorem]{Lemma}
\newtheorem{proposition}[theorem]{Proposition}
\newtheorem{corollary}[theorem]{Corollary}
\newtheorem{conjecture}[theorem]{Conjecture}

\theoremstyle{definition}
\newtheorem{definition}[theorem]{Definition}
\newtheorem{notation}[theorem]{Notation}

\numberwithin{equation}{section}

\allowdisplaybreaks

\newcommand{\cM}{\mathscr{M}}
\newcommand{\yM}{\mathbf{y}_{\hspace{-0.4ex}\mathscr{M}}}
\newcommand{\yR}{\mathbf{y}_{\mathbb{R}}}
\newcommand{\cB}{\mathscr{B}}
\newcommand{\cO}{\mathscr{O}}
\newcommand{\cP}{\mathscr{P}}
\newcommand{\scp}{\mathscr{p}}
\newcommand{\bF}{\mathbb{F}}
\newcommand{\bQ}{\mathbb{Q}}
\newcommand{\oQ}{\overline{\mathbb{Q}}}
\newcommand{\oZ}{\overline{\mathbb{Z}}}
\newcommand{\oF}{\overline{\mathbb{F}}}
\newcommand{\bR}{\mathbb{R}}
\newcommand{\bC}{\mathbb{C}}
\newcommand{\bZ}{\mathbb{Z}}
\newcommand{\bx}{\mathbf{x}}
\newcommand{\by}{\mathbf{y}}
\newcommand{\bp}{\mathbf{p}}
\newcommand{\bq}{\mathbf{q}}
\renewcommand{\bf}{\mathbf{f}}
\newcommand{\fp}{\mathfrak{p}}
\newcommand{\bt}{\mathbf{t}}
\newcommand{\be}{\mathbf{e}}
\newcommand{\cl}{\textup{cl}}
\newcommand{\tr}{\operatorname{tr}}

\begin{document}

\title[McCullough--Wanderley conjectures]{Markoff triples and Nielsen equivalence in $\text{SL}_2(\bF_p)$}

\thanks{This research is supported by NSF grant 2336000.}

\author{Daniel E. Martin}
\address{Clemson University, O-110 Martin Hall, 220 Parkway Drive, Clemson, SC}
\email{dem6@clemson.edu}



\date{\today}

\begin{abstract}In 2013, Darryl McCullough and Marcus Wanderley made a series of conjectures that describe the Nielsen equivalence classes and $T_2$-equivalence classes of pairs of generators for $\text{SL}_2(\bF_q)$ and the Markoff equivalence classes of triples in $\bF_q^3$ that solve $x^2+y^2+z^2=xyz+\kappa$ for some $\kappa\in\bF_q$. (The case $\kappa=0$ was originally conjectured by Baragar in 1991.) We prove that one of the McCullough--Wanderley conjectures, the ``$Q$-Classification Conjecture" on Markoff triples, implies the others. Then we prove that the $Q$-Classification Conjecture holds if $q=p$ is a prime such that $24{,}504{,}480$ does not divide $p^2-1$. More generally, for any integer $d$, we reduce the $Q$-Classification Conjecture for all primes $p\not\equiv \pm 1\,\text{mod}\,d$ to checking whether a roughly $2d\times 2d$ matrix with entries in $\bQ[\kappa]$ is invertible. We (and SageMath) perform this invertibility check for all prime powers $d$ up to $17$, hence the modulus $24{,}504{,}480=2\hspace{0.3ex}\text{lcm}(1,2,\dots,17)$.
\end{abstract}

\maketitle

\section{Introduction}

\subsection{Statement of Results} The primary result of this paper is the following:

\begin{theorem}\label{thm:main1}Let $p$ be prime such that $24{,}504{,}480\centernot\mid p^2-1$. For any $\kappa\in\bF_p\backslash\{4\}$, there is a single orbit of solutions in $\bF_p^3$ to $x^2+y^2+z^2=xyz+\kappa$ under the group generated~by \[(x,y,z)\mapsto (yz-x,y,z),\;(x,y,z)\mapsto (x,xz-y,z)\;\text{and}\;(x,y,z)\mapsto (x,y,xy-z),\] other than the following exceptional orbits up to coordinate permutation, provided the elements exist in $\bF_p$: \begin{enumerate}\item the orbit of $(\sqrt{\kappa},0,0)$, \item the orbit of $(1,1,1)$ when $\kappa=2$, \item the orbit of $(1,0,\frac{1}{2}(1\pm \sqrt{5}))$ when $\kappa=\frac{1}{2}(5\pm\sqrt{5})$, \item or the orbits of $(1,0,\sqrt{2})$ and $(\frac{1}{2}(1+\sqrt{5}),1,1)$ when $\kappa=3$.\end{enumerate}\end{theorem}

The modulus $24{,}504{,}480=2\hspace{0.3ex}\text{lcm}(1,2,\dots,17)$ can be increased with further computation. For any positive integer $d$, we reduce the theorem above for all primes $p\not\equiv\pm 1\,\text{mod}\,d$ to an explicit matrix rank calculation. Performing this calculation for all prime powers $d$ up to $17$ proves the theorem above. This resolves the``$Q$-Classification Conjecture" of McCullough and Wanderley \cite{wanderley} for all primes that avoid the forbidden  set of $2^7$ congruence classes modulo $5\cdot 7\cdot 9\cdot 11\cdot 13\cdot 16\cdot 17$. (Sage code to verify the conjecture for a prime power input $d$ can be found on the author's website: \href{https://dem6.people.clemson.edu/}{\url{dem6.people.clemson.edu}}.)

Bourgain, Gamburd, and Sarnak have obtained an asymptotic version of Theorem~\ref{thm:main1}. In \cite{sarnak} they prove that for any $\varepsilon>0$, if $x$ is sufficiently large then the density of primes $p<x$ for which Theorem~\ref{thm:main1} fails when $\kappa=0$ is at most $x^\varepsilon$. As noted in \cite{bourgain} and \cite{sarnak}, and as detailed in a forthcoming paper \cite{BGS2}, their argument generalizes to arbitrary $\kappa$. Similar questions have also been addressed for variants and generalizations of the Markoff surface in \cite{baragar2,baragar3} and \cite{fuchs}. 

Setting $\kappa=0$ yields the classical Markoff equation $x^2+y^2+z^2=xyz$ (sometimes with $3xyz$ in place of $xyz$). In this case, the transitivity asserted by Theorem \ref{thm:main1} was first conjectured to hold for all primes by Baragar in 1991 \cite{baragar}. Baragar's conjecture was resolved for all but finitely many primes (specifically $p > 10^{393}$ \cite{eddy}) by a theorem of Chen \cite{chen}, building on the work of Bourgain, Gamburd, and Sarnak \cite{BGS,bourgain,sarnak}. An alternative proof of Chen's theorem has been provided by the author in \cite{martin}, and it has been generalized to other Markoff-type equations by de~Courcy-Ireland, Litman, and Mizuno in \cite{mizuno}.

Investigation into the classical Markoff equation ($\kappa=0$) was historically motivated by Diophantine approximation. More precisely, the entries in solutions from $\bZ^3$, called \textit{Markoff numbers}, control the Markoff and Lagrange spectra \cite{markoff}. This phenomenon does not occur for nonzero $\kappa$---the more general Markoff equation has surprising connections to other disciplines. Over $\bC$, the group action in Theorem~\ref{thm:main1} mirrors that of the monodromy group of Painlev{\'e} VI equations on $\bC^2$ \cite{boalch} \cite{planat}. And over $\bF_p$ (and its extensions), the same group action mirrors that of Nielsen moves on $\text{SL}_2(\bF_p)$ generating pairs. Indeed, the small orbits indicated in (1--4) have been computed from both the Painlev{\'e} VI \cite{dubrovin} \cite{lisovyy} and $\text{SL}_2(\bF_p)$ \cite{mccullough} \cite{wanderley} perspectives without reference to one another.

Our motivation for considering arbitrary $\kappa$ is to determine Nielsen classes of $\text{SL}_2(\bF_p)$, the subject of the McCullough--Wanderley conjectures. The secondary result of this paper is the following, proved as a consequence of Theorem \ref{thm:main1}.

\begin{theorem}\label{thm:main2}Let $p$ be prime with $12{,}252{,}240\centernot\mid p^2-1$. For any $\kappa\in\bF_p\backslash\{4\}$, there is a single orbit of generating pairs $(A,B)$ for $\textup{SL}_2(\bF_p)$ with $\tr ABA^{-1}B^{-1}=\kappa-2$ under the group generated by \[(A,B)\mapsto (A,AB),\;(A,B)\mapsto (B,A)\;\text{and}\;(A,B)\mapsto (A^{-1},B),\] except when $\kappa=0$ and $p\equiv 1\,\textup{mod}\,4$, in which case there are two orbits.\end{theorem}

This resolves the ``Classification Conjecture," the ``Trace Conjecture," and the ``$T$-Classification Conjecture" of McCullough and Wanderley \cite{wanderley} for the same set of congruence classes as Theorem \ref{thm:main1}.

Our two theorems split the paper into two components. The remainder of the introduction exposits the relation between Theorems \ref{thm:main1} and Theorem \ref{thm:main2}. Then in Section \ref{sec:2} we prove that the former implies the latter. In fact, Theorem \ref{thm:equiv_cons} proves more generally that the McCullough--Wanderley conjectures are equivalent over $\bF_q$ for any prime power $q$. The perspective is centered around Nielsen classes through the proof of Theorem \ref{thm:equiv_cons} in Section \ref{sec:2}. Then the perspective shifts to Markoff triples and does not return. Section \ref{sec:3} outlines the proof strategy for Theorem \ref{thm:main1}, and Sections \ref{sec:4}--\ref{sec:8} carry out the strategy.

\subsection{Background} For an integer $n\geq 2$ and a group $G$, two $n$-tuples in $G$ are called \textit{Nielsen equivalent} if one can be obtained from the other via a sequence of \textit{elementary Nielsen moves}: \begin{align}\label{eq:n1}(g_1,...,g_i,...,g_n)&\mapsto  (g_1,...,g_j^{\pm 1}g_i,...,g_n),\\ \label{eq:n2}(g_1,...,g_i,...,g_n)&\mapsto (g_1,...,g_ig_j^{\pm 1},...,g_n),\\ \label{eq:n3}(g_1,...,g_i,...,g_j,...,g_n)&\mapsto  (g_1,...,g_j,...,g_i,...,g_n),\\ \label{eq:n4}\text{or }(g_1,...,g_i,...,g_n)&\mapsto (g_1,...,g_i^{-1},...,g_n)\end{align} for distinct $i$ and $j$. Nielsen moves provide a kind of Euclidean algorithm for the combinatorial group theorist. A survey on their history and utility can be found in \cite{fine}. See also \cite{havlena} and \cite{day} for more recent references and applications, and see \cite[Chapter 9]{crowell} and \cite{lustig} for specialized applications to knot theory and K-theory. Nielsen moves have also become important in computational group theory over the last few decades. Group-generating $n$-tuples are the vertices and elementary Nielsen moves are the edges of the \textit{extended product replacement graph}, on which a random walk is known as the \textit{product replacement algorithm} for generating random group elements \cite{pak}.

Nielsen moves naturally partition the set of $n$-tuples in $G$. The resulting \textit{Nielsen equivalence classes} have been completely determined in a handful of cases (refer to \cite{pak} for a list). In most of those cases, $n$ strictly exceeds $d(G)$, the minimum size of a generating set for $G$. The extent of our knowledge when $n=d(G)$ is as follows: there is a unique Nielsen class when $G$ is the fundamental group of a closed surface \cite{louder} \cite{zieschang}, and Nielsen classes in abelian groups are completely determined \cite{diaconis} \cite{oancea}. The only other full account of Nielsen classes when $n=d(G)$ for an infinite family of groups is conjectural, due to McCullough and Wanderley \cite{wanderley}. Let us describe~it.

Higman observed that if $(g_1,g_2)$ and $(\tilde{g}_1,\tilde{g}_2)$ are Nielsen equivalent in $G$, then the \textit{extended conjugacy classes} of the commutators, $\cl_G([g_1,g_2])\cup\cl_G([g_2,g_1])$ and $\cl_G([\tilde{g}_1,\tilde{g}_2])\cup\cl_G([\tilde{g}_2,\tilde{g}_1])$, are equal. This union is called the \textit{Higman invariant} of a Nielsen class. The ``Classification Conjecture" asserts that this is a complete invariant in the case $G=\text{SL}_2(\bF_q)$:

\begin{conjecture}[``Classification Conjecture" \cite{wanderley}]\label{con:classification}Nielsen classes of generating pairs in $\textup{SL}_2(\bF_q)$ are uniquely determined by the Higman invariant.\end{conjecture}

When $G=\text{SL}_2(\bF_q)$, all matrices in $\cl_G([g_1,g_2])\cup\cl_G([g_2,g_1])$ have the same trace, so we also have a \textit{trace invariant}. An equivalent version of the Classification Conjecture in terms of the trace invariant can be found in \cite{wanderley}. (The trace invariant is used in the phrasing of Theorem \ref{thm:main2}.)

McCullough and Wanderley also consider the coarser notion of equivalence determined by so-called $T_n$\emph{-systems}. Their ``$T$-Classification Conjecture" asserts that $T_2$-systems in $\text{SL}_2(\bF_q)$ are uniquely determined by the trace invariant. They prove that the Classification Conjecture implies the $T$-Classification Conjecture \cite{wanderley}.

Finally, McCullough and Wanderley consider the triple in $\bF_q^3$ associated to a pair of matrices $A,B\in\text{SL}_2(\bF_q)$: \begin{equation}\label{eq:tracemap}(A,B)\mapsto (\text{tr}\,A,\text{tr}\,B,\text{tr}\,AB).\end{equation} A triple in $\bF_q^3$ is called \textit{essential} if it is the image of a pair $(A,B)$ that generates $\text{SL}_2(\bF_q)$. To see the relation to Nielsen moves, let us recall a few facts relevant to the trace map above. First, Macbeath showed that the preimage of any triple in $\bF_q^3$ is a single nonempty $\text{SL}_2(\bF_{q^2})$-simultaneous conjugacy class of $\text{SL}_2(\bF_q)$ pairs \cite{macbeath}. Second, we have Fricke's trace identity, \[(\tr A)^2+(\tr B)^2+(\tr AB)^2=(\tr A)(\tr B)(\tr AB)+\tr\,[A,B] + 2.\] Third, we have already remarked that $\tr\,[A,B]$ is constant on Nielsen classes. When combined, these facts tell us that Nielsen moves permute solutions in $\bF_q^3$ to the Markoff equation \begin{equation}\label{eq:markoff}x^2+y^2+z^2=xyz+\kappa\end{equation} for some fixed $\kappa\in\bF_q$. We call such solutions \textit{Markoff triples (with respect to $\kappa$)}. Comparing the Markoff equation to Fricke's trace identity, let us highlight the correspondence \[\kappa=\tr\,[A,B]+2,\] to be used when translating between $\text{SL}_2(\bF_q)$ pairs and $\bF_q$ triples.

It is straightforward to work out the action on triples corresponding to elementary Nielsen moves: \begin{align}\label{eq:Gamma_gens}\nonumber((A,B)\mapsto (A,AB))&\;\rightsquigarrow\; ((x,y,z)\mapsto(x,z,xz-y)),\\ \nonumber((A,B)\mapsto (B,A))&\;\rightsquigarrow\; ((x,y,z)\mapsto(y,x,z)), \\ \text{and }((A,B)\mapsto (A^{-1},B))&\;\rightsquigarrow\; ((x,y,z)\mapsto(x,y,xy-z)).\end{align} All Nielsen moves are generated by these three. The last map, $(x,y,z)\mapsto(x,y,xy-z)$, is called a \textit{Vieta involution}. The first- and second-coordinate Vieta involutions are defined similarly and can be obtained from the appropriate compositions of the three maps above. In light of this, two triples are called \textit{Markoff equivalent} if one can be obtained from the other by a combination of Vieta involutions and coordinate permutations---those maps induced by Nielsen moves.

\begin{conjecture}[``$Q$-Classification Conjecture" \cite{wanderley}]\label{con:Q}The Markoff class of an essential triple $(x,y,z)\in\bF_q^3$ is uniquely determined by $\kappa\coloneqq x^2+y^2+z^2-xyz$.\end{conjecture}  

As indicated in Theorem \ref{thm:main1}, the Vieta involutions alone produce the same Markoff classes. There is no need for coordinate permutations.

McCullough and Wanderley verified their conjectures computationally for $q\leq 101$, and they proved them for all $q$ such that $q-1$ is prime and $q+1$ is thrice a prime (which may or may not constitute an infinite set) \cite{wanderley}. Our approach to these conjectures takes the Markoff perspective. We begin in Section \ref{sec:2} by proving the following:

\begin{theorem}\label{thm:main3}The $Q$-Classification Conjecture implies the Classification and $T$-Classification Conjectures.\end{theorem}

This allows us to determine Nielsen classes in $\text{SL}_2(\bF_q)$ without real consideration for its group structure. Note that Theorem \ref{thm:main3} is already known for certain special cases of $q$ and $\kappa$ \cite{wanderley} \cite{Campos}, to be described shortly.

After Section \ref{sec:2} we focus on Markoff triples and the proof of Theorem \ref{thm:main1}, which is essentially the $Q$-Classification Conjecture. Again, see Section \ref{sec:3} for an overview of the proof.

\section{Relations among the McCullough--Wanderley conjectures}\label{sec:2}

We begin by relating the trace and Higman invariants.

\begin{proposition}[Proposition 5.2 and Lemma 5.3 in \cite{wanderley}]\label{prop:wanderley}Among $\textup{SL}_2(\bF_q)$ generating pairs, there is a unique Higman invariant (i.e. extended conjugacy class of the commutator) of any given trace in $\bF_q\backslash\{\pm 2\}$. There is no Higman invariant of trace $2$, and there are either one or two Higman invariants of trace $-2$ depending on whether $q\equiv 3\,\textup{mod}\,4$ or $q\equiv 1\,\textup{mod}\,4$, respectively.\end{proposition}

McCullough and Wanderley deduce from this that the Classification Conjecture implies the $T$- and $Q$-Classification Conjectures (Corollary 5.7 and Proposition 8.5 in \cite{wanderley}). The converse implications are more challenging and only partially known under additional hypotheses. Specifically, it is proved in \cite{wanderley} that the $Q$-Classification Conjecture implies the Classification Conjecture when $q$ is even or when $4-\kappa$ is not a square in $\bF_q$, and Campos-Vargas proved the implication for all $\kappa\in\bF_p\backslash\{4\}$ when $q=p$ is a prime congruent to $3\,\text{mod}\,4$ \cite{Campos}. Our goal in this section is to prove the implication for arbitrary prime powers $q$ and arbitrary $\kappa\in\bF_q\backslash\{4\}$. To do so, we employ a similar proof strategy to that of McCullough and Wanderley. Let us describe it.

Let $(A_1,B_1)$ and $(A_2,B_2)$ be generating pairs for $\text{SL}_2(\bF_q)$ with the same Higman invariant. If the $Q$-Classification Conjecture holds, there is some sequence of Vieta involutions and coordinate permutations that transforms $(\tr A_2,\tr B_2,\,\tr A_2B_2)$ into $(\tr A_1,\tr B_1,\,\tr A_1B_1)$. This corresponds to a sequence of Nielsen moves applied to $(A_2,B_2)$, but it may not end at $(A_1,B_1)$. We only know that the endpoint $(\tilde{A}_1,\tilde{B}_1)$ satisfies $(\tr A_1,\tr B_1,\,\tr A_1B_1)=(\tr \tilde{A}_1,\tr \tilde{B}_1,\tr \tilde{A}_1\tilde{B}_1)$. We would like to say that this equality of triples implies $(A_1,B_1)$ is Nielsen equivalent to $(\tilde{A}_1,\tilde{B}_1)$ and thus to $(A_2,B_2)$---that would prove the Classification Conjecture. So, naturally, we ask what can be said about two matrix pairs that correspond to the same triple. The answer comes primarily from the work of MacBeath \cite{macbeath}. (Note that matrix pairs are called \textit{conjugate} if they are simultaneously conjugate.)

\begin{lemma}[Lemma 2.1(ii) in \cite{wanderley}]\label{lem:wanderley}Two pairs of generators for $\textup{SL}_2(\bF_q)$ are Nielsen equivalent if they are $\textup{SL}_2(\bF_q)$-conjugate.\end{lemma}

\begin{theorem}[Theorem 3 in \cite{macbeath}]\label{thm:macbeath}Let $(\alpha,\beta,\gamma)\in\bF_q^3$ solve the Markoff equation for some $\kappa\neq 4$. If $q$ is odd, there are exactly two $\textup{SL}_2(\bF_q)$-conjugacy classes of matrix pairs with trace $(\alpha,\beta,\gamma)$.\end{theorem}

Returning to our setup in $\text{SL}_2(\bF_q)$, it follows immediately that $(A_1,B_1)$ and $(\tilde{A}_1,\tilde{B}_1)$ must be Nielsen equivalent in the special case $q\equiv 1\,\text{mod}\,4$ and $\kappa=0$ (or, equivalently,  $\tr A_1^{-1}B_1^{-1}A_1B_1=-2$). Indeed, the two Higman invariants of trace $-2$ identified in Proposition \ref{prop:wanderley} must be the Higman invariants of the two conjugacy classes of matrix pairs identified in Theorem \ref{thm:macbeath}. Since $(A_1,B_1)$ and $(\tilde{A}_1,\tilde{B}_1)$ have the same Higman invariant by hypothesis and both correspond to the same Markoff triple, they must lie in the same $\text{SL}_2(\bF_q)$-conjugacy class by Theorem \ref{thm:macbeath}. Thus $(A_1,B_1)$ and $(\tilde{A}_1,\tilde{B}_1)$ are Nielsen equivalent by Lemma \ref{lem:wanderley}.

Outside of the special case $\kappa=0$ and $q\equiv1\,\text{mod}\,4$, we are not so lucky---there is only one Higman invariant of trace $\kappa-2$, but there are two conjugacy classes of matrix pairs corresponding to the Markoff triple $(\tr A_1,\tr B_1,\,\tr A_1B_1)$. It is entirely possible that $(A_1,B_1)$ and $(\tilde{A}_1,\tilde{B}_1)$ lie in different conjugacy classes, rendering Lemma \ref{lem:wanderley} inapplicable. Our strategy here is to show that the Nielsen class of $(A_1,B_1)$ contains two distinct conjugacy classes corresponding to each (or any) Markoff triple, so it must contain $(\tilde{A}_1,\tilde{B}_1)$ and thus $(A_2,B_2)$. 

At this point, our proof diverges from that of McCullough and Wanderley. In Lemma 10.2 of \cite{wanderley}, it is shown that $(A,B)$ and $(A^{-1},B^{-1})$ are not conjugate if $2-\tr A^{-1}B^{-1}AB$ is not a square in $\bF_q$. Since $(A,B)$ and $(A^{-1},B^{-1})$ correspond to the same Markoff triple and are evidently Nielsen equivalent, that completes McCullough and Wanderley's proof: the $Q$-Classification Conjecture implies the Classification Conjecture when $4-\kappa$ is not a square. In our proof, we use instead the Nielsen equivalent pairs $(A,B)$ and $(B,A)$. It turns out that when $\kappa\in\bF_q\backslash\{0,4\}$ or when $\kappa=0$ and $q\equiv 3\,\text{mod}\,4$, there exist generators $A$ and $B$ for $\text{SL}_2(\bF_q)$ such that $(A,B)$ and $(B,A)$ are not conjugate yet correspond to the same Markoff triple. That is the only missing ingredient to a complete proof.

We can construct such generators $A$ and $B$ from a specific kind of Markoff triple. It takes the form $(\alpha,\alpha,\gamma)$ with $\gamma$ satisfying the properties below.

\begin{lemma}\label{lem:gamma}Let $q > 353$, and let $p=\textup{char}(\bF_q)$. If $\kappa\in\bF_q\backslash\{0,4\}$ or if $\kappa=0$ and $q\equiv 3\,\textup{mod}\,4$, there exists $\gamma\in\bF_q\backslash\{0\}$ such that $\bF_p(\gamma)=\bF_q$, and neither $2-\gamma$, $\kappa-\gamma^2$, $\kappa-8+4\gamma-\gamma^2$, nor $-\kappa+\kappa \gamma-\gamma^2$ is a square in $\bF_q$.\end{lemma}

\begin{proof}When $\kappa$ is not 0 or 4, none of the four polynomials in $\gamma$ share any roots or have repeated roots. Consequently, the Weil bound on multiplicative character sums guarantees the existence of our desired $\gamma$ for sufficiently large $q$. Since we wish to prove the lemma for all $q > 353$, let us check what ``sufficiently large" means.

Let $\chi$ be the quadratic character on $\bF_q^\times$ and set $\chi(0)=0$. The form of the Weil bound we need is Theorem 11.23 in \cite{iwaniec}: $|\sum_{\smash{\bF_q}}\chi(f(x))|\leq (\deg(f)-1)\sqrt{q}$ provided $f(x)\in\bF_q[x]$ is not a square in $\overline{\bF}_q[x]$ (or, more generally, not an $n^\text{th}$ power if $\chi$ has order $n$). By expanding the products below and applying the Weil bound to each resulting sum, we get
\begin{align*}\frac{1}{16}\sum_{\gamma\in\bF_q}\!\Big((1-\chi(2-\gamma))(1-\chi(\kappa-\gamma^2))\hspace{-4cm}&\\[-0.4cm]&\cdot(1-\chi(\kappa-8+4\gamma-\gamma^2))(1-\chi(-\kappa+\kappa \gamma-\gamma^2))\Big)\geq \frac{1}{16}(q-56\sqrt{q}).\end{align*} The argument in the sum above is 1 when neither $2-\gamma$, $\kappa-\gamma^2$, $\kappa-8+4\gamma-\gamma^2$, nor $-\kappa+\kappa \gamma-\gamma^2$ is a square and 0 otherwise, except that roots of the four polynomials are counted as $\frac{1}{2}$ or $0$. As per the lemma statement, we wish to avoid these seven roots as well as $0$ and elements from a proper subfield of $\bF_q$. A crude over-count (provided $q$ is at least 11) of the number of elements to be avoided is $\frac{3}{2}\sqrt{q}$, which is less than $\frac{1}{16}(q-56\sqrt{q})$ when $q > 6400$. Thus when $\kappa$ is not 0 or 4 and $q > 6400$, the desired $\gamma$ exists.

When $\kappa=0$ and $q\equiv 3\,\text{mod}\,4$, the requirement that $\kappa-\gamma^2$ and $-\kappa+\kappa\gamma-\gamma^2$ are not squares holds automatically. This makes the resulting bound on $q$ much less than $6400$. We omit details.

For prime powers smaller than $6400$, the lemma can be verified by direct computation. Since the number of acceptable $\gamma$ in $\bF_q$ is roughly $\frac{1}{16}q$, a brute-force search is quick.\end{proof}

We will need to know that the Markoff triple $(\alpha,\alpha,\gamma)$ obtained from Lemma~\ref{lem:gamma} is essential. For this we have McCullough and Wanderley's description of nonessential triples in Section 11 of \cite{wanderley}. The theorem below provides a summary. Parts (1--4b) also appear as the ``Main Theorem" in McCullough's unpublished manuscript \cite{mccullough}, where additional proof details are provided. A full account is also provided in \cite{Campos}. 

\begin{theorem}[Section 11 in \cite{wanderley}]\label{thm:nonessential}Let $\varphi=\frac{1}{2}(1+\sqrt{5})$ and $\overline{\varphi}=\frac{1}{2}(1-\sqrt{5})$. If $\kappa\in\bF_q\backslash\{4\}$, then a Markoff triple with respect to $\kappa$ is essential if and only if it is not among the following exceptions up to permuting or negating coordinates:\begin{enumerate}\item $(\sqrt{\kappa},0,0)$, \item $(1,1,0)$ or $(1,1,1)$ when $\kappa=2$, \item[(3a)] $(\varphi,\varphi,\varphi)$, $(\varphi,\varphi,1)$, or $(\varphi,0,1)$ when $\kappa=2+\varphi$, \item[(3b)] $(\overline{\varphi},\overline{\varphi},\overline{\varphi})$, $(\overline{\varphi},\overline{\varphi},1)$, or $(\overline{\varphi},0,1)$ when $\kappa=2+\overline{\varphi}$,\item[(4a)] $(\sqrt{2},0,1)$ or $(\sqrt{2},\sqrt{2},1)$ when $\kappa=3$, \item[(4b)] $(\varphi,\overline{\varphi},0)$, $(\varphi,\overline{\varphi},-1)$, $(\varphi,1,1)$, or $(\overline{\varphi},1,1)$, when $\kappa=3$\item[(5)]or $(\alpha,\beta,\gamma)$ with $\bF_p(\alpha^2,\beta^2,\gamma^2,\kappa)\neq\bF_q$, where $p=\textup{char}(\bF_q)$.\end{enumerate}\end{theorem}

Except for (5), each category above includes all triples in a single Markoff class up to permuting and negating coordinates. Of course, category (5) can account for many different Markoff classes if $\bF_q$ has proper subfields. Note that the numbering above matches Theorem~\ref{thm:main1}.

We now have all the pieces needed to equate the McCullough--Wanderley conjectures.

\begin{customthm}{\ref{thm:main3}}\label{thm:equiv_cons}The $Q$-Classification Conjecture implies the Classification and $T$-Classification Conjectures.\end{customthm}

\begin{proof}Assume the $Q$-Classification Conjecture. Let $\kappa\in\bF_q\backslash\{4\}$, and assume $q\equiv 3\,\text{mod}\,4$ if $\kappa=0$ (otherwise we are done by Lemma \ref{lem:wanderley}, Theorem \ref{thm:macbeath}, and Proposition \ref{prop:wanderley}). Assume that $\gamma\in\bF_q$ from Lemma \ref{lem:gamma} exists (and fix one). Then there exists $\alpha\in\bF_q$ with \[\alpha^2=\frac{\kappa-\gamma^2}{2-\gamma}\] because the right-side expression is a square by choice of $\gamma$. Thus $(\alpha,\alpha,\gamma)$ is a Markoff triple for the given $\kappa$. We claim it is an essential triple. Indeed, it cannot fall into category (5) of Theorem \ref{thm:nonessential} because $\bF_p(\alpha^2,\gamma^2,\kappa)\supseteq \bF_p(\gamma)=\bF_q$, again by choice of $\gamma$. From categories (1--4b), the only nonessential triples of the form $(\alpha,\alpha,\gamma)$ for which it is possible that neither $2-\gamma$ nor $\kappa-\gamma^2$ is a square are $\pm(1,1,0)$. As we insisted that $\gamma\neq 0$ in Lemma \ref{lem:gamma}, $(\alpha,\alpha,\gamma)$ must be essential.

By the $Q$-Classification Conjecture, any $\text{SL}_2(\bF_q)$ generating pair with commutator trace $\kappa-2$ is Nielsen equivalent to a pair with trace $(\alpha,\alpha,\gamma)$. So to prove our theorem it suffices to show that all pairs with trace $(\alpha,\alpha,\gamma)$ are Nielsen equivalent. By Lemma \ref{lem:wanderley} and Theorem \ref{thm:macbeath}, this follows if we can find two Nielsen equivalent pairs of trace $(\alpha,\alpha,\gamma)$ that are not conjugate.

Observe that \[\alpha^2-4=\frac{\kappa-8+4\gamma-\gamma^2}{2-\gamma}\] is also a square by choice of $\gamma$, so there exists $\zeta\in\bF_q$ with \begin{equation}\label{eq:x}\zeta+\zeta^{-1}=\alpha.\end{equation} Next, the choice of $\gamma$ allows us to fix $\eta\in\bF_q$ such that \[\eta^2=\frac{-\kappa+\kappa\gamma-\gamma^2}{\kappa-8+4\gamma-\gamma^2}=\frac{\alpha^2-\kappa}{\alpha^2-4}.\] It does not matter which of the two choices of $\zeta$ or $\eta$ we use. Finally, observe that \[\frac{\alpha^2}{\eta^2}-4=\frac{(\gamma^2-4\gamma+\kappa)^2}{(2-\gamma)(-\kappa+\kappa\gamma-\gamma^2)}\] is also a square, so there exists $\vartheta\in\bF_q^\times$ with $\eta(\vartheta+\vartheta^{-1})=\alpha$. Now here, the choice between the two possibilities for $\vartheta$ does matter. Any value of $\vartheta$ makes $(\zeta+\zeta^{-1},\eta(\vartheta+\vartheta^{-1}),\eta(\zeta\vartheta+\zeta^{-1}\vartheta^{-1}))$ a Markoff triple, so one of the two values that solves $\eta(\vartheta+\vartheta^{-1})=\alpha$ must make it $(\alpha,\alpha,\gamma)$, while the other (replacing $\vartheta$ with $\vartheta^{-1}$) makes it $(\alpha,\alpha,\alpha^2-\gamma)$. We pick the one for which \begin{equation}\label{eq:y}\eta(\vartheta+\vartheta^{-1})=\alpha\hspace{\parindent}\text{and}\hspace{\parindent}\eta(\zeta\vartheta+\zeta^{-1}\vartheta^{-1})=\gamma.\end{equation} 

The matrix pair \[A=\begin{bmatrix}\zeta & 0 \\ 0 & \zeta^{-1}\end{bmatrix},\;B=\begin{bmatrix}\eta\vartheta & (\zeta^{-1}-\zeta)^{-1} \\ (\zeta^{-1}-\zeta)(\eta^2-1) & \eta\vartheta^{-1}\end{bmatrix}\in\text{SL}_2(\bF_q)\] satisfies $(\tr  A, \tr  B, \tr  AB)=(\alpha,\alpha,\gamma)$. The Nielsen equivalent pair $(B,A)$ also has trace $(\alpha,\alpha,\gamma)$, and we claim it is not $\text{SL}_2(\bF_q)$-conjugate to $(A,B)$. To see this, we compute \begin{align*}\begin{bmatrix}\zeta^{-1}-\eta \vartheta & (\zeta-\zeta^{-1})^{-1} \\ \eta \vartheta-\zeta & (\zeta^{-1}-\zeta)^{-1}\end{bmatrix}(B,A)\begin{bmatrix}\zeta^{-1}-\eta \vartheta & (\zeta-\zeta^{-1})^{-1} \\ \eta \vartheta-\zeta & (\zeta^{-1}-\zeta)^{-1}\end{bmatrix}^{\!-1}\hspace{-2cm}&\\&=\left(A,\begin{bmatrix}\eta\vartheta & \eta\vartheta-\zeta^{-1} \\ \zeta-\eta\vartheta & \eta\vartheta^{-1}\end{bmatrix}\right).\end{align*} Call the last matrix $C$. The equation above shows that $(B,A)$ and $(A,C)$ are $\text{SL}_2(\bF_q)$-conjugate, so the claim follows if $(A,B)$ and $(A,C)$ are not conjugate. The centralizer of $A$ is the subgroup of diagonal matrices. If $D\in\text{SL}_2(\bF_q)$ is diagonal, the top-right entry of $DBD^{-1}$ is a square multiple of $(\zeta^{-1}-\zeta)^{-1}$. Thus $DBD^{-1}$ could equal $C$ only if $(\zeta^{-1}-\zeta)(\eta\vartheta-\zeta^{-1})$ is a square in $\bF_q$. But \begin{flalign*}&&(\zeta^{-1}-\zeta)(\eta\vartheta-\zeta^{-1})&=\zeta^{-1}\eta(\vartheta+\vartheta^{-1})+1-\zeta^{-2}-\eta(\zeta \vartheta+\zeta^{-1}\vartheta^{-1})&&\\ && &=\zeta^{-1}\alpha+1-\zeta^{-2}-\gamma && \text{by (\ref{eq:y})}\\&& &=\zeta^{-1}(\zeta+\zeta^{-1})+1-\zeta^{-2}-\gamma && \text{by (\ref{eq:x})}\\&& &=2-\gamma, && \end{flalign*} which is not a square as per Lemma \ref{lem:gamma}. Thus $(A,B)$ and $(B,A)$ cannot be $\text{SL}_2(\bF_q)$-conjugate. This completes the proof when $\gamma$ from Lemma \ref{lem:gamma} exists.

As described below, the theorem is proved by direct computation when $q$ and $\kappa$ are such that $\gamma$ from Lemma \ref{lem:gamma} does not exist.\end{proof}

McCullough and Wanderley have already verified their conjectures for $q \leq 101$. For $q > 101$, there are exactly 37 pairs $q,\kappa$ for which $\gamma$ from Lemma \ref{lem:gamma} does not exist, and $q\leq 181$ in all but two of those pairs (namely $q=353$ and $\kappa=36$ or $181$). Over such small fields, verifying Theorem \ref{thm:equiv_cons} computationally is feasible: pick any generating pair $(A,B)$ for $\textup{SL}_2(\bF_q)$ with $\tr\,[A,B] = \kappa-2$, and build the Nielsen class of $(A,B)$ until two non-conjugate pairs are found that correspond to the same Markoff triple. For example, in $\bF_{353}$ with $\kappa=181$, we find \[\left(\begin{bmatrix}296 & 0 \\ 0 & 161\end{bmatrix},\begin{bmatrix}182 & 183 \\ 74 & 216\end{bmatrix}\right)\hspace{\parindent}\text{and}\hspace{\parindent}\left(\begin{bmatrix}296 & 0 \\ 0 & 161\end{bmatrix},\begin{bmatrix}182 & 201 \\ 315 & 216\end{bmatrix}\right).\] Both pairs correspond to the essential Markoff triple $(104,45,45)$. They are not simultaneously conjugate because, as in the last proof, the top-right entries 183 and 201 from the second matrices are not square multiples of one another. To see that these pairs are Nielsen equivalent, let $\sigma_z\circ\tau_x$, $\tau_z$, and $\sigma_z$ denote the three Nielsen moves in (\ref{eq:Gamma_gens}), respectively. Then the composition $\sigma_z\circ\tau_x\circ(\sigma_z\circ\tau_z)^2\circ(\sigma_z\circ\tau_x)^2\circ\sigma_z$ maps the first matrix pair above to the second.

As a final remark, the map $(A,B)\mapsto(B,A)$ used in the proof of Theorem \ref{thm:equiv_cons} is not a \textit{special} (determinant 1) Nielsen move; it comes from the nontrivial coset in $\textup{Aut}(F_2)/\textup{SAut}(F_2)$, where $F_2$ is the free group on two letters and $\textup{SAut}(F_2)$ is the kernel of $\textup{Aut}(F_2)\xrightarrow{\smash{\raisebox{-0.3ex}{\(\scriptstyle\text{ab}\)}}}\textup{GL}_2(\bZ)\xrightarrow{\smash{\raisebox{-0.3ex}{\(\scriptstyle\text{det}\)}}}\{\pm1\}$. So for the computational group theorist, our results on the $Q$-Classification Conjecture do not fully determine connected components of the product replacement graph, but rather the \textit{extended} product replacement graph. In the special case that $4-\kappa$ is not a square, however, there is already McCullough and Wanderley's proof of Theorem \ref{thm:equiv_cons}, which uses the \textit{special} Nielsen move $(A,B)\mapsto(A^{-1},B^{-1})$.

\section{Overview of the remaining sections}
\label{sec:3}

\subsection{The main definitions}\label{ss:3.1} We introduce all but one of the objects central to the proof of Theorem \ref{thm:main1} in advance. The definition that we skip for now is not as succinct as those below.

\begin{notation}\label{not:Gamma}Let $\sigma_x$, $\sigma_y$, and $\sigma_z$ denote the three Vieta involutions as in Theorem~\ref{thm:main1}, let $\tau_x$, $\tau_y$, and $\tau_z$ denote the three coordinate transpositions indexed by their fixed coordinate, and let $\Gamma$ denote the group generated by the $\sigma_i$ and $\tau_j$ along with the double sign change $(x,y,z)\mapsto(x,-y,-z)$. Let $\Gamma_{\!x}$ denote the stabilizer of the first coordinate (generated by $\tau_x$, $\sigma_y$, and the double sign change).\end{notation}

Throughout the paper, $R$ is an integral domain and $F$ is its field of fractions.  We use $\overline{F}$ to denote the algebraic closure of $F$ and $\overline{R}$ to denote the integral closure of $R$ in $\overline{F}$. There are reminders of this notation throughout.

The definitions below, just like the term \textit{Markoff triple}, depend on the value of $\kappa$ that determines the Markoff equation. Since we so rarely have occasion to consider two distinct values of $\kappa$ at once (it only happens in the proof of Proposition \ref{prop:reduction}), the subscript $\kappa$ is suppressed in notation.

\begin{notation}For a fixed $\kappa\in R$, let $\cM(R)\subseteq R^3$ denote the set of Markoff triples.\end{notation}

\begin{notation}\label{not:P(R)}For a fixed $\kappa\in R$, let $\cP(R)$ denote those $f\in \overline{R}[x]$ such that \[\sum_{\mathclap{\bt\in\cO}} f(x)=0\] for any finite $\Gamma$-invariant subset $\cO\subseteq\cM(R)$. Note that ``$x$" in the summation is shorthand for $x(\bt)$, where $x(\alpha,\beta,\gamma)=\alpha$.\end{notation}

The inclusion of $(x,y,z)\mapsto(-x,-y,z)$ in $\Gamma$ guarantees $x^{2n+1}\in\cP(R)$ for any $n\geq 0$. It is the even degree polynomials in $\cP(R)$ that are hard to come by.

\begin{definition}A \emph{first-coordinate orbit} is a set of the form $\Gamma_{\!x}\cdot\bt$ for some $\bt\in\cM(R)$. We use $\cO_x$ to denote a generic first-coordinate orbit and $\cO_{\alpha}$ for some $\alpha\in R$ to denote a generic first-coordinate orbit in which $x=\alpha$.\end{definition}

\begin{definition}Let $\alpha\in R$, and let $\zeta\in\overline{R}$ solve $\zeta+\zeta^{-1}=\alpha$. The \textit{(rotation) order} of $\alpha$, denoted $\text{ord}(\alpha)$, is the smallest positive even integer $n$ such that $\zeta^n=1$. If no such $n$ exists, we say $\alpha$ has infinite order.\end{definition}

Insisting that $\alpha$ have even order is not standard in the literature, nor is it of theoretical importance in our work. It does, however, lead to cleaner propositions.

\begin{notation}\label{not:P_x(R,d)}For $\kappa\in R$ and an integer $d\geq 2$, let $\cP_{\!x}(R,d)$ denote those $f\in \overline{R}[x,y,z]$ such that $\sum_{\cO_\alpha}\!f(\bt)=0$ whenever $\cO_\alpha$ is finite and $2d\centernot\mid \text{ord}(\alpha)$. Also let $\cP_{\!x}(R,\infty)$ denote those $f\in \overline{R}[x,y,z]$ such that $\sum_{\cO_\alpha}\!f(\bt)=0$ for all but finitely many $\alpha\in R$.\end{notation}

\subsection{Proof strategy} Let us turn to $R=\bF_q$ and Theorem \ref{thm:main1}. Several small $\Gamma$-invariant subsets of $\cM(\bF_q)$ are identified for certain $\kappa$ in Theorem~\ref{thm:nonessential}. The $Q$-Classification Conjecture predicts that $\Gamma$ acts transitively on all remaining Markoff triples. In particular, if $q=p$ is prime (to avoid case (5) of Theorem~\ref{thm:nonessential}) and $\kappa$ is neither $2$ (to avoid case (2)), $2+\varphi$ (to avoid case (3a)), $2+\overline{\varphi}$ (to avoid case (3b)), $3$ (to avoid case (4)), nor $4$ (generally forbidden), then the $Q$-Classification Conjecture predicts that $\Gamma$ should act transitively on $\cM(\bF_p)$ if $\legendre{\kappa}{p}=-1$, and $\cM(\bF_p)$ should break into exactly two $\Gamma$-orbits if  $\legendre{\kappa}{p}=1$, namely $\Gamma\cdot(\sqrt{\kappa},0,0)$ (from case (1)) and everything else.

To limit the number of $\Gamma$-invariant subsets of $\cM(\bF_p)$, we plan to build up the rank of $\cP(\bF_p)$ from Notation \ref{not:P(R)}. To see why this works, suppose $\cO$ is $\Gamma$-invariant, and for $\alpha\in\bF_p$ let $c_\cO(\alpha)$ count the number of triples in $\cO$ with first coordinate $\alpha$. For any $f=f(x)\in\bF_p[x]$, \[\sum_{\mathclap{\bt\in\cO}}f(x)=\sum_{\alpha\in\bF_p}c_\cO(\alpha)f(\alpha).\] Provided $f$ is an even polynomial, this shows that $f\in\cP(\bF_p)$ if and only if the vectors \[\begin{bmatrix}f(0)\\ f(1)\\\vdots\\ f(p-1)\end{bmatrix}\;\;\text{and}\;\;\begin{bmatrix}c_\cO(0)\\ c_\cO(1)\\\vdots\\ c_\cO(p-1)\end{bmatrix}\] are orthogonal for every $\Gamma$-invariant subset $\cO$ of $\cM(\bF_p)$. So the more polynomials we produce in $\cP(\bF_p)$, the smaller its orthogonal complement, which means there are fewer possible $\Gamma$-invariant subsets. This is made precise in Theorem \ref{thm:P^perp}, which rephrases Theorem \ref{thm:main1} and the $Q$-Classification Conjecture in terms of the expected orthogonal complement of $\cP(\bF_p)$ .

This leads us to the last fundamental definition that is missing from Section~\ref{ss:3.1}. There is a polynomial reduction algorithm, call it $\Phi$ (the subscript $\kappa$ is suppressed again), that is useful for producing elements of $\cP(\bF_p)$, and more generally $\cP(R)$. The algorithm takes as input a multivariate polynomial $f=f(x,y,z)$ and outputs a univariate polynomial $\Phi(f)=\Phi(f)(x)$ satisfying $\sum_\cO f(\bt)=\sum_\cO\Phi(f)(x)$ for any $\Gamma$-invariant set $\cO\subseteq\cM(R)$. Our strategy is to apply $\Phi$ to multivariate polynomials for which it is easy to check that $\sum_\cO f(\bt)$ vanishes when $R=\bF_q$. For example, if $f=(x^q-x)yz^3$ then $\Phi(f)$ is in $\cP(\bF_q)$ because $f$ is identically 0 on $\bF_q^3$.

Let us define the output of $\Phi$ for monic monomial inputs from $R[x,y,z]$, where $R$ is some integral domain. Arbitrary inputs are then handled by extending linearly. Consider the input $x^\ell y^mz^n$ with $\ell,m,n>0$. For any $\cO\subseteq\cM(R)$, we have \[\sum_{\mathclap{\bt\in\cO}}x^\ell y^m z^n=\sum_{\bt\in\cO}x^{\ell-1}y^{m-1}z^{n-1}(x^2+y^2+z^2-\kappa)\] by virtue of $\bt$ being a Markoff triple (we need not even assume $\cO$ is $\Gamma$-invariant here). The total degree of $x^{\ell-1}y^{m-1}z^{n-1}(x^2+y^2+z^2-\kappa)$ is one less than that of $x^\ell y^mz^n$. We denote a reduction of this form as $\rho$, so \[\rho(x^\ell y^mz^n)=x^{\ell-1}y^{m-1}z^{n-1}(x^2+y^2+z^2-\kappa).\] Now consider a monomial with only two variables, say $y^mz^n$ with $m,n > 0$. If $\cO\subset\cM(R)$ is closed under the first-coordinate Vieta involution then \begin{flalign}\label{eq:sigma_preserve}\nonumber&&\sum_{\mathclap{\bt\in\cO}}y^mz^n &= \sum_{\mathclap{\bt\in\cO}}(y^mz^n-xy^{m-1}z^{n-1}z+xy^{m-1}z^{n-1})\\\nonumber && &=\sum_{\mathclap{\bt\in\cO}}y^{m-1}z^{n-1}(yz-x)+\sum_{\mathclap{\bt\in\cO}}xy^{m-1}z^{n-1}\\\nonumber&& &=\sum_{\mathclap{\bt\in\cO}}xy^{m-1}z^{n-1}+\sum_{\mathclap{\bt\in\cO}}xy^{m-1}z^{n-1}&&\text{by Vieta involution}\\&& &=\sum_{\mathclap{\bt\in\cO}}2xy^{m-1}z^{n-1}.\end{flalign} Again, the total degree of $2xy^{m-1}z^{n-1}$ is one less than the degree of the input. We abuse notation and express a step of this form as \[\sigma_x(y^mz^n)=2xy^{m-1}z^{n-1}.\] Of course $\sigma_y$ and $\sigma_z$ are defined analogously. Once all multivariate terms of $f$ have been eliminated, what remains can be expressed in $x$ alone using the coordinate permutations $\tau_x,\tau_y,\tau_z\in\Gamma$. Again we abuse notation and write \[\tau_y(z^n)=x^n.\] The final result is the desired univariate polynomial $\Phi(f)$ that satisfies \begin{equation}\label{eq:Phi_invariance}\sum_{\mathclap{\bt\in\cO}}f(\bt)=\sum_{\mathclap{\bt\in\cO}}\Phi(f)(x)\text{ whenever }\Gamma\cdot\cO=\cO.\end{equation}

Our reduction algorithm is deterministic because for every monomial there is a unique prescribed operation. Note that the choice to combine ``like terms" at any stage does not affect the output because we defined $\Phi$ on monomials and extended linearly. Also note that variable permutations need not be reserved for the final stage. For example, the result is the same whether we apply $\sigma_x$, $\tau_z\circ\sigma_y$, and $\tau_y\circ \sigma_z$ to the three terms of $yz+xz+xy$, respectively, or whether we permute variables in order to combine the three monomials first then apply $\sigma_x$ to $3yz$. The general principle is below. It matches the relation between $\tau_i$ and $\sigma_j$ viewed as elements of~$\Gamma$.

\begin{proposition}\label{prop:permute}For any $i,j\in\{x,y,z\}$, $\tau_i\circ\sigma_j = \sigma_{\tau_i(j)}\circ\tau_i$ on any applicable bivariate monomial, and $\tau_i\circ\rho = \rho\circ\tau_i$ on any trivariate monomial.\end{proposition}

\begin{proof}Without loss of generality, let $j=x$. Then \begin{align*}(\tau_i\circ\sigma_x)(y^mz^n)&=\tau_i(2xy^{m-1}z^{n-1})\\&=2\tau_i(x)\tau_i(y)^{m-1}\tau_i(z)^{n-1}\\&=\sigma_{\tau_i(x)}(\tau_i(y)^m\tau_i(z)^n)\\&=(\sigma_{\tau_i(x)}\circ\tau_i)(y^mz^n).\end{align*} The second claim is verified in similar fashion.\end{proof}

We have already seen a few polynomials in $\oF_q[x,y,z]$ that evidently sum to $0$ over $\Gamma$-invariant subsets of $\cM(\bF_q)$. All odd polynomials work due to the double sign change in $\Gamma$, as do the even polynomials $(x^{q+1}-x^2)y^{2n}$ (whose $\Phi$ reduction is also even) since $x^{q+1}-x^2$ is identically 0 on $\bF_q$. Unfortunately, finding a formula for $\Phi((x^{q+1}-x^2)y^{2n})$ appears to be a serious challenge, and without one we cannot determine the rank of the reduced polynomials as $n$ ranges. For $f\in\bF_q[x,y,z]$, it is straightforward to find the coefficients of the largest powers of $x$ in $\Phi(f)$. Indeed, Section \ref{sec:5} is devoted to proving such a formula (Theorem~\ref{thm:gen_form}; see also (\ref{eq:spec_form}) and Theorem~\ref{thm:spec_form}). But the author has no formula for the coefficients of smaller powers. This makes it difficult to prove that the span of $\Phi((x^{q+1}-x^2)y^{2n})$ as $n$ ranges includes polynomials of small degree, specifically anything of degree less than $\frac{1}{2}(q+1)$. 

In avoidance of this obstacle, we turn to a family of polynomials with much more variation in their degrees. The smallest example from this family is $f(x,y,z)=y^4- y^2z^2+\tfrac{1}{2}x^2y^2$. Before seeing its significance, here is its reduction: \begin{align*}y^4-y^2z^2+\tfrac{1}{2}x^2y^2&\xmapsto{\tau_y}y^4-x^2y^2+\tfrac{1}{2}x^2y^2=y^4-\tfrac{1}{2}x^2y^2\\&\xmapsto{\sigma_z}y^4-xyz\\&\xmapsto{\rho}y^4-(x^2+y^2+z^2)\\&\xmapsto{\tau_y,\tau_z}x^4-3x^2.\end{align*} To see the utility of this polynomial, consider the following reduction of $xf$: \begin{align}\label{eq:phi_x}\nonumber xf=xy^4-xy^2z^2+\tfrac{1}{2}x^3y^2&\xmapsto{\sigma_z} 2y^3z-xy^2z^2+x^2yz\\\nonumber &\xmapsto{\rho}2y^3z-yz(x^2+y^2+z^2)+x^2yz\\&\xmapsto{\tau_x}2y^3z - x^2yz - y^3z - y^3z+x^2yz=0.\end{align} What makes this last reduction special is that it only uses operations that preserve the first coordinate. In other words, it avoids $\sigma_x$, $\tau_y$, and $\tau_z$. Thus if $\cO_\alpha$ is some first-coordinate orbit, then \begin{equation}\label{eq:first_phi_x}\alpha\sum_{\mathclap{\bt\in\cO_{\alpha}}}f(\bt)=\sum_{\mathclap{\bt\in\cO_{\alpha}}}x f(\bt)=\sum_{\bt\in\cO_{\alpha}}0=0.\end{equation} So if $\alpha\neq 0$, the left-side sum above must vanish. But then if $0$ \textit{never} appears as a coordinate in $\cM(\bF_q)$ (which happens when $\kappa=0$ and $q\equiv 3\,\text{mod}\,4$), the sum of $f$ over every first-coordinate orbit must vanish. Now, any $\Gamma$-invariant set $\cO\subseteq\cM(\bF_q)$ can be viewed as a disjoint union of first-coordinate orbits, so this would imply $0=\sum_{\cO}f(\bt)=\sum_{\cO}\Phi(f)(x)$. But then $\Phi(f)=x^4-3x^2$ must lie in $\cP(\bF_q)$!...at least when $\kappa=0$ and $q\equiv 3\,\text{mod}\,4$.

In pursuit of similar polynomials, we make extensive use of the partial reduction algorithm in (\ref{eq:phi_x}), which is restricted to the operations $\rho$, $\sigma_y$, $\sigma_z$, and $\tau_x$. We call this reduction algorithm $\Phi_x$. Since $\Phi_x$ cannot reduce a monomial of the form $y^mz^n$, the output of $\Phi_x$ need not be univariate. Instead, it is some element of $R[x]+R[y,z]$ satisfying the analog of (\ref{eq:Phi_invariance}): \begin{equation}\label{eq:Phi_x_invariance}\sum_{\mathclap{\bt\in\cO_x}}f(\bt)=\sum_{\mathclap{\bt\in\cO_x}}\Phi_x(f)(\bt)\text{ whenever }\Gamma_{\!x}\cdot\cO_x=\cO_x.\end{equation} To make the output of $\Phi_x$ well-defined, we adopt the convention that every monomial appearing in $\Phi_x(f)$ has degree in $y$ at least that of $z$.

If $\text{char}(R)\neq 2$, then the rotation order of $0$ is $4$ in $R$ since $0$ is the sum of a primitive fourth root of unity and its inverse. Recalling Notation \ref{not:P_x(R,d)}, we conclude that $y^4- y^2z^2+\frac{1}{2}x^2y^2\in\cP_{\!x}(R,2)$. We generalize this example in Section \ref{sec:6}: for integers $d$ and $n$ satisfying $d\,|\,n$ and $\text{char}(R)\centernot\mid d$, we find polynomials of degree $2n$ in $\cP_{\!x}(R,d)$, one for each $\alpha\in R\backslash\{\pm 2\}$ with $2d\,|\,\text{ord}(\alpha)$ and $\text{ord}(\alpha)\,|\,2n$. These polynomials turn out to be eigenvectors of eigenvalue $\alpha$ with respect to a certain linear map, just as (\ref{eq:first_phi_x}) suggests $y^4-y^2z^2+\tfrac{1}{2}x^2y^2$ is an eigenvector of eigenvalue $0$. Now, the example with $\alpha=0$ and $d=2$ was not widely applicable---only in the special case $\kappa=0$ and $q\equiv 3\,\text{mod}\,4$ does $\cM(\bF_q)$ not possess a triples with $0$ as an entry. But for larger $d$, building up $\cP_{\!x}(R,d)$, particularly when $R=\oZ$, is more useful. Let's consider how the case $d=4$ is relevant to all prime powers $q\not\equiv\pm 1\,\text{mod}\,8$. If $\fp\subset\oZ$ is a prime over $\text{char}(\bF_q)$ and $f\in\cP_{\!x}(\oZ,4)$ we can reduce $f\,\text{mod}\,\fp$ to produce a polynomial in $\cP_{\!x}(\bF_q,4)$ (Proposition \ref{prop:reduction}). By Notation \ref{not:P_x(R,d)}, the only first-coordinate orbits $\cO_\alpha\subset \cM(\bF_q)$ on which $\sum_{\cO_\alpha}f(\bt)$ is not guaranteed to vanish are those where $8\,|\,\text{ord}(\alpha)$. But since $q\not\equiv\pm1\,\text{mod}\,8$ and all elements of $\bF_q$ have rotation order dividing $q\pm 1$ (see Corollary \ref{cor:total_count}), no such orbits exist. In particular, $\Phi(f)\in\cP(\bF_q)$.

For $q=p$ a prime not congruent to $\pm 1\,\text{mod}\,2d$, we are able to prove that the $\Phi$ reductions of $\cP_{\!x}(\oZ,d)\,\text{mod}\,\fp$ generate all but perhaps the smallest degree polynomials that are expected to be in $\cP(\bF_p)$. This is Corollary \ref{cor:found_em}, the culmination of all the work in Sections \ref{sec:5} and \ref{sec:6}. ``Filling out" the rest of $\cP(\bF_p)$ so that it matches what Theorem \ref{thm:P^perp} predicts is a computational task that we approach from two angles. First, in Section \ref{sec:7} we add polynomials of the form $\Phi((x^{p+1}-x^2)y^{2n})$ into $\cP(\bF_p)$. We only need a few of them, namely $n=1,2,3,4$ for generic $\kappa$, and it is not too hard to find a complete formula for the $\Phi$ reductions when $n$ is small. Second, in Section \ref{sec:8}, we use computer assistance to compute the reductions of the smallest polynomials in $\cP_{\!x}(\oZ,d)$ and check that they span the expected space (see Algorithm \ref{alg}). Working over $\oZ$ and projecting onto residue fields is the only way to fill out $\cP(\bF_p)$ for all $p\not\equiv\pm 1\,\text{mod}\,2d$ simultaneously. That is the purpose of considering $\cP_{\!x}(\oZ,d)$.

In summary, we proceed as follows:
\begin{enumerate}[labelindent=4.2em]
    \item[Section \ref{sec:4}:] Determine what $\cP(\bF_p)$ must equal for Theorem \ref{thm:main1} to hold (Theorem~\ref{thm:P^perp}).
    \item[Section \ref{sec:5}:] Compute generic formulas for the top coefficients of $\Phi(f)$ (Theorems \ref{thm:gen_form} and \ref{thm:spec_form}).
    \item[Section \ref{sec:6}:] Find ``eigenvectors" generalizing (\ref{eq:phi_x}) (Theorem \ref{thm:eigenvector}). Specialize Section \ref{sec:5}'s formula for $\Phi(f)$ to these eigenvectors (Theorem \ref{thm:eigen_form}) to prove that $\cP_{\!x}(\oZ,d)$ contains polynomials of any sufficiently large degree (Corollary \ref{cor:found_em}).
    \item[Section \ref{sec:7}:] Fill in the gaps (which are small; 2-, 3-, or 4-dimensional depending on $\kappa$) between what $\cP_{\!x}(\oZ,d)\,\text{mod}\,\fp$ is expected to be and what Section~\ref{sec:4} says $\cP(\bF_p)$ must be (Theorem \ref{thm:local}).
    \item[Section \ref{sec:8}:]With help from SageMath, prove that $\cP_{\!x}(\oZ,d)$ contains the small degree polynomials that are missing from Section \ref{sec:6} but expected in Section \ref{sec:7}.
\end{enumerate}

Theoretical work essentially stops after the proof of Lemma \ref{lem:yvecs}, at which point the paper becomes largely computational.

\section{The conjectured space $\cP(\bF_p)$}\label{sec:4}

The three preliminary results below are well-known. We provide proofs, albeit terse, for the sake of completeness.

\begin{proposition}\label{prop:orbits}Let $\cO_\alpha\subseteq\cM(R)$ be some first-coordinate orbit. If $\alpha^2\neq 4$, let $\zeta,\eta \in\overline{F}$ solve $\zeta+\zeta^{-1}=\alpha$ and $(\alpha^2-4)\eta^2=(\alpha^2-\kappa)$. There exists $\vartheta\in\overline{F}$ such that \[\cO_\alpha = \begin{cases}(\alpha,\eta(\zeta^n\vartheta+\zeta^{-n}\vartheta^{-1},\eta(\zeta^{n+1}\vartheta+\zeta^{-n-1}\vartheta^{-1})) & \alpha^2\neq 4,\kappa \\ (\alpha,\zeta^n\vartheta,\zeta^{n+1}\vartheta) & \alpha^2=\kappa \\ (\alpha,\vartheta+n\sqrt{\kappa-4},\vartheta+(n+1)\sqrt{\kappa-4}) & \alpha^2=4\end{cases}\] for $n\in \bZ$, up to permuting or negating the second and third coordinates.\end{proposition}

\begin{proof}A quick check shows that a triple in any of the three given forms solves the Markoff equation. Conversely, any Markoff triple can be rewritten in one these three forms by solving for $\zeta$, $\eta$, and $\vartheta$. 

Next, it may be computed directly that $\tau_x\circ \sigma_y$ increases $n$ by exactly one in all three cases and $\sigma_y\circ\tau_x$ decreases $n$ by exactly one on all three cases. Since $\Gamma_{\!x}$ is generated by $\tau_x$, $\sigma_y$, and $(x,y,z)\mapsto(x,-y,-z)$, we see that $\cO_{\alpha}$ consists of the indicated triples up to permuting or negating the second and third coordinates.\end{proof}

\begin{corollary}\label{cor:total_count}Let $q$ be an odd prime power, and let $\chi$ be the quadratic character on $\bF_q^\times$. If $\alpha\in\bF_q\backslash\{\pm 2,\pm\sqrt{\kappa}\}$, then there are $q-\chi(\alpha^2-4)$ triples in $\cM(\bF_q)$ with first coordinate $\alpha$, $\textup{ord}(\alpha)$ divides $q-\chi(\alpha^2-4)$, and every first-coordinate orbit $\cO_{\alpha}$ has size $\textup{ord}(\alpha)$ or $2\,\textup{ord}(\alpha)$.\end{corollary}

\begin{proof}Define $\chi(0)=0$. By solving the Markoff equation for $z$, the number of triples with first two coordinates $\alpha$ and $\beta$ is seen to be $1+\chi(\beta^2(\alpha^2-4)-4\alpha^2+4\kappa)$. Thus the number of triples with first coordinate $\alpha$ is a well-known sum: \[\sum_{\beta\in\bF_q}(1+\chi(\beta^2(\alpha^2-4)-4\alpha^2+4\kappa))=q-\chi(\alpha^2-4)\] as claimed.

Next consider some $\cO_{\alpha}$ with $\zeta+\zeta^{-1}=\alpha$ so that $\text{ord}(\alpha)$ is the multiplcative order of $\zeta$ if $-1\in\langle\zeta\rangle$ and twice the multiplicative order of $\zeta$ otherwise. We have $\chi(\alpha^2-4)=-1$ if and only if $\zeta\not\in\bF_q$, in which case the image of $\zeta$ under the nontrivial element of $\text{Gal}(\bF_{q^2}/\bF_q)$ is $\zeta^{-1}$. That is, $\zeta$ is in the kernel of the norm map $\bF_{q^2}^\times\to\bF_q^\times$. This kernel has size $q+1$, so $\text{ord}(\alpha)$ divides $q+1=q-\chi(\alpha^2-4)$. The other possibility is $\zeta\in\bF_q$, in which case $\text{ord}(\alpha)$ divides the size of $\bF_q^{\times}$, which is $q-1=q-\chi(\alpha^2-4)$. Finally, note that $\text{ord}(\alpha)$ counts the number of elements in Proposition \ref{prop:orbits} as $n$ varies \textit{and} as the sign of the second and third coordinates change (because $\text{ord}(\alpha)$ is twice the multiplcative order of $\zeta$ when $-1\not\in\langle\zeta\rangle$). After transposing the second and third coordinates, this count can either double or stay the same.\end{proof}

\begin{corollary}\label{cor:max_orders}Let $q$ be an odd prime power. If $\alpha\in\bF_q\backslash\{\pm\sqrt{\kappa}\}$ has rotation order $q+1$ or $q-1$, then there is a unique first-coordinate orbit $\cO_{\alpha}$ in $\cM(\bF_q)$.\end{corollary}

\begin{proof}Let $(\alpha,\beta_1,\gamma_1),(\alpha,\beta_2,\gamma_2)\in\cM(\bF_q)$, and let $\zeta,\eta\in\bF_{q^2}$ solve $\zeta+\zeta^{-1}=\alpha$ and $(\alpha^2-4)\eta^2=(\alpha^2-\kappa)$. Assume that $\text{ord}(\alpha)=q\pm 1$ so that $\langle\zeta,-1\rangle$ is either $\bF_q^\times$ or the kernel of the norm $\bF_{q^2}^\times\to\bF_q^\times$.

Observe that \[\vartheta_i^{\pm 1}\coloneqq\frac{1}{2\sqrt{\alpha^2-\kappa}}\left(\sqrt{\alpha^2-4}\beta_i\pm (2\gamma_i-\alpha\beta_i)\right)\] satisfies $\eta(\vartheta_i+\vartheta_i^{-1})=\beta_i$. If $\sqrt{\alpha^2-4}\in\bF_q$, we see directly from the formula for $\vartheta_i^{\pm1}$ that $\vartheta_1\vartheta_2^{-1}\in\bF_q^\times$. But $\sqrt{\alpha^2-4}\in\bF_q$ also implies $\zeta\in\bF_q^\times$. Thus $\vartheta_1$ and $\vartheta_2$ differ by a power of $\pm\zeta$, which shows $(\alpha,\beta_2,\gamma_2)\in \Gamma_{\!x}\cdot (\alpha,\beta_1,\gamma_1)$ by Proposition \ref{prop:orbits}. If $\sqrt{\alpha^2-4}\not\in\bF_q$, then $\vartheta_1,\vartheta_2\not\in\bF_q$. Using the formulas for $\vartheta_1$ and $\vartheta_2$, we can determine their images under the nontrivial automorphism of $\bF_{q^2}/\bF_q$: either $\vartheta^q_1=\vartheta_1^{-1}$ and $\vartheta^q_2=\vartheta_2^{-1}$ if $\sqrt{\alpha^2-\kappa}\not\in\bF_q$ or $\vartheta^q_1=-\vartheta_1^{-1}$ and $\vartheta^q_2=-\vartheta_2^{-1}$ if $\sqrt{\alpha^2-\kappa}\in\bF_q$. Either way, $\vartheta_1\vartheta_2^{-1}$ belongs to the kernel of the norm $\bF_{q^2}^\times\to\bF_q^\times$, which is generated by $\pm\zeta$. Again this gives $(\alpha,\beta_2,\gamma_2)\in \Gamma_{\!x}\cdot (\alpha,\beta_1,\gamma_1)$ by Proposition \ref{prop:orbits}.\end{proof}

\begin{notation}\label{not:P^perp(F_q)}Index the coordinates of $\bF_q^q$ by $0,\dots,q-1$, and let $\be_i$ be the $i^\text{th}$ standard basis row vector. Let $\cP^{\perp}\hspace{-0.3ex}(\bF_q)$ denote the orthogonal complement of the image of $\cP(\bF_q)$ under the composition \begin{align*}\bF_q[x]\to\bF_q[x]/(x^q-x)&\xrightarrow{\sim}\bF_q^q\\ x^i+(x^q-x)&\,\mapsto \be_i.\end{align*}\end{notation}

\begin{notation}\label{not:y_M}Let \[\bx=\sum_{i=0}^{q}x^i\be_i^T,\] and for $\alpha\in\bF_q$ let $\bx(\alpha)$ denote the vector $\bx$ with $x$ replaced by $\alpha$.\end{notation}

Row vectors are used for polynomials and column vectors for inputs so that ``$f(x)$" can be written as a matrix product in the form $\bf(\bx)$.

\begin{lemma}\label{lem:P^perp}For $\cO\subseteq\cM(\bF_q)$ and $\alpha\in\bF_q$, let $c_{\cO}(\alpha)$ count the number of triples in $\cO$ with first coordinate $\alpha$. Then  \[\cP^\perp\hspace{-0.3ex}(\bF_q)=\textup{span}\left\{\sum_{\alpha\in\bF_q}\!c_{\cO}(\alpha)\bx(\alpha)\;\Bigg|\; \Gamma\cdot\cO=\cO\right\}.\]\end{lemma}

\begin{proof}The proof is essentially an unwrapping of definitions. Let $\bf\in\bF_q^q$ correspond to some $f\in\bF_q[x]$ as per Notation \ref{not:P^perp(F_q)}. By definition, $f\in\cP(\bF_q)$ if and only if $\sum_{\cO}f(x)=0$ whenever $\Gamma\cdot\cO=\cO$. We have defined each $c_{\cO}(\alpha)$ so that \[\sum_{\bt\in\cO}f(x)=\sum_{\alpha\in\bF_q}\!c_{\cO}(\alpha)f(\alpha)=\bf\Bigg(\sum_{\alpha\in\bF_q}\!c_{\cO}(\alpha)\bx(\alpha)\Bigg).\] In particular, $\cP^\perp\hspace{-0.3ex}(\bF_q)$ contains the span in the lemma statement, and the orthogonal complement of the span is contained in $\cP(\bF_q)$. \end{proof}

Combining Lemma \ref{lem:P^perp} with Theorem \ref{thm:nonessential} tells us what we should expect $\cP^\perp\hspace{-0.3ex}(\bF_q)$ to equal provided the $Q$-Classification Conjecture holds. Indeed, to obtain spanning vectors we need only compute the coefficients $c_{\cO}(\alpha)$ for the various orbits listed in Theorem \ref{thm:nonessential} and plug them into Lemma \ref{lem:P^perp}'s formula. The vectors corresponding to cases (1--4b) are given below. As before, $\varphi\coloneqq\frac{1}{2}(1+\sqrt{5})$ and $\overline{\varphi}\coloneqq\frac{1}{2}(1-\sqrt{5})$. \begin{align}\label{eq:y_vecs}\nonumber\by_\kappa&\coloneqq 2\bx(0)+\tfrac{1}{2}(\bx(\sqrt{\kappa})+\bx(-\sqrt{\kappa}))\\ \nonumber \by_1&\coloneqq \bx(0)+\tfrac{3}{2}(\bx(1)+\bx(-1))\\ \nonumber\by_{\varphi}&\coloneqq 2\bx(0)+\tfrac{3}{2}(\bx(1)+\bx(-1))+\tfrac{5}{2}(\bx(\varphi)+5\bx(-\varphi)) \\ \nonumber \by_{\overline{\varphi}}&\coloneqq 2\bx(0)+\tfrac{3}{2}(\bx(1)+\bx(-1))+\tfrac{5}{2}(\bx(\overline{\varphi})+\bx(-\overline{\varphi}))\\ \nonumber \by_2&\coloneqq 2\bx(0)+\tfrac{3}{2}(\bx(1)+\bx(-1))+2(\bx(\sqrt{2})+\bx(-\sqrt{2}))\\ \by_5&\coloneqq 2\bx(0)+3(\bx(1)+\bx(-1))+\tfrac{5}{2}(\bx(\varphi)+\bx(-\varphi))+\tfrac{5}{2}(\bx(\overline{\varphi})+\bx(-\overline{\varphi})).\end{align} As an example, $\by_2$ comes from the orbit in case (4a) of Theorem~\ref{thm:nonessential}, in which $0$ occurs eight times as a first coordinate, $1$ and $-1$ occur six times each, and $\sqrt{2}$ and $\sqrt{2}$ occur eight times each. (The choice to scale by $\frac{1}{4}$ is made for the sake of computational convenience in Section \ref{sec:7}.)

Computing coefficients for case (5) of Theorem \ref{thm:nonessential} is more challenging, but we restrict to the case that $q=p$ is prime so the computation is unnecessary. This leaves only the orbit of essential triples. Since this orbit is the complement of the nonessential triples in $\cM$, we can obtain its vector by subtracting those above (depending on $\kappa$) from $\sum_{\bF_p}c_{\cM}(\alpha)\bx(\alpha)$. Hence we define \begin{flalign}\label{eq:yM}\nonumber&& \yM&\coloneqq\sum_{\alpha\in\bF_p}\!c_{\cM}(\alpha)\bx(\alpha) && \\ \nonumber && &=\sum_{\alpha\in\bF_p}(p-\chi(\alpha^2-4))\sum_{i=0}^{p-1}\alpha^i\be_i^T&& \text{by Corollary \ref{cor:total_count}}\\ \nonumber && &=-\sum_{i=0}^{p-1}\Bigg(\sum_{\alpha\in\bF_p}\!(1+\chi(\alpha^2-4))\alpha^i\Bigg)\be_i^T && \text{since }{\textstyle\sum_{\bF_p}\!\alpha^i=0}\\ \nonumber && &=-\sum_{i=0}^{p-1}\Bigg(\sum_{\zeta\in\bF_p^{\smash{\times}}}\!(\zeta+\zeta^{-1})^i\Bigg)\be_i^T && \text{by setting }\alpha=\zeta+\zeta^{-1}\\ && &=2\be_{p-1}+\sum_{i=0}^{\frac{p-1}{2}}\binom{2i}{i}\be_{2i}^T && \text{since }{\textstyle\sum_{\bF_p^\times}\!\zeta^i=0}\text{ if }p-1\centernot\mid i.\end{flalign}

\begin{theorem}\label{thm:P^perp}Let $p$ be prime, let $\kappa\in\bF_p\backslash\{4\}$, and for $\alpha\in\bF_p$ let $\delta_\alpha=1$ if $\alpha$ is a square and $0$ otherwise. Theorem \ref{thm:main1} (or the $Q$-Classification Conjecture) holds if and only if $\cP^\perp\hspace{-0.3ex}(\bF_p)$ is the span of\begin{enumerate}\item $\yM$ and $\delta_\kappa\by_{\kappa}$ when $\kappa\neq 2,3,2+\varphi,2+\overline{\varphi}$, \item $\yM$, $\delta_\kappa\by_{\kappa}$, and $\by_1$ when $\kappa=2$, 
\item[(3a)] $\yM$, $\delta_\kappa\by_{\kappa}$, and $\by_{\varphi}$ when $\kappa=2+\varphi$, \item[(3b)] $\yM$, $\delta_\kappa\by_{\kappa}$, and $\by_{\overline{\varphi}}$ when $\kappa=2+\overline{\varphi}$, or \item[(4)] $\yM$, $\delta_\kappa\by_\kappa$, $\delta_2\by_2$, and $\delta_5\by_5$ when $\kappa=3$.\end{enumerate}\end{theorem}

\begin{proof}Assume the $Q$-Classification Conjecture holds (which is slightly weaker than assuming Theorem \ref{thm:main1} holds, because the group in Theorem \ref{thm:main1} is generated by Vieta involutions alone). Then every $\Gamma$-orbit in $\cM(\bF_p)$ is either one of the nonessential subsets explicitly listed in cases (1--4b) of Theorem~\ref{thm:nonessential} or else the complement (in $\cM(\bF_p)$) of their union. The vectors that Lemma \ref{lem:P^perp} associates to these orbits have been listed in cases (1--4) of the present theorem. Note that vectors corresponding to an orbit that only exists for certain $p$, namely $\by_\kappa$, $\by_2$, and $\by_5$, have been scale by a $\delta$-value accordingly.

Conversely, assume that $\cP^\perp\hspace{-0.3ex}(\bF_p)$ is the span indicated in one of the cases (1--4), depending on $\kappa$. Let $\cO$ be a nonempty $\Gamma$-orbit in $\cM(\bF_p)$. Assuming $p\geq 13$, we may fix a primitive root $\zeta$ for $\bF_{\smash{p}}^\times$ such that $\tilde{\alpha}\coloneqq\zeta+\zeta^{-1}$ is neither $0$, $\pm1$, $\pm\sqrt{2}$, $\pm\sqrt{\kappa}$, $\pm\varphi$, nor $\pm\overline{\varphi}$. By Corollary \ref{cor:max_orders}, $\cO$ either contains every triple with first coordinate $\tilde{\alpha}$ or none of them. In other words, either $c_\cO(\tilde{\alpha})=c_{\cM}(\tilde{\alpha})$ or $c_\cO(\tilde{\alpha})=0$. Suppose the latter occurs. Let us consider how $\sum_{\smash{\bF_p}}\!c_{\cO}(\alpha)\bx(\alpha)$, which belongs to $\cP^\perp\hspace{-0.3ex}(\bF_p)$ by Lemma \ref{lem:P^perp}, could possibly be expressed as a linear combination of the spanning vectors in any of the cases (1--4). The vectors $\bx(\alpha)$ for $\alpha\in\bF_p$ form a basis for $\bF_p^{\smash{p}}$. With respect to this basis, the coefficient of $\bx(\tilde{\alpha})$ is $0$ in each of the listed $\cP^{\smash{\perp}}\hspace{-0.3ex}(\bF_p)$ spanning vectors except for $\yM$, because we insisted that $\tilde{\alpha}$ is neither $0$, $\pm1$, $\pm\sqrt{2}$, $\pm\sqrt{\kappa}$, $\pm\varphi$, nor $\pm\overline{\varphi}$. Thus the assumption $c_\cO(\tilde{\alpha})=0$ means $\yM$ must not appear when $\sum_{\smash{\bF_p}}\!c_{\cO}(\alpha)\bx(\alpha)$ is written as a combination of spanning vectors. In particular, $\sum_{\smash{\bF_p}}\!c_{\cO}(\alpha)\bx(\alpha)$ is a combination of only $\bx(0)$, $\bx(\pm1)$, $\bx(\pm\sqrt{2})$, $\bx(\pm\sqrt{\kappa})$, $\bx(\pm\varphi)$, and $\bx(\pm\overline{\varphi})$, meaning $c_\cO(\alpha)\equiv 0\,\text{mod}\,p$ for all $\alpha\in \bF_p\backslash\{0,\pm1,\pm\sqrt{2},\pm\sqrt{\kappa},\pm\varphi,\pm\overline{\varphi}\}$. Corollary \ref{cor:total_count} tells us that $0\,\text{mod}\,p$ occurrences implies no occurrences. That is, all triple entries in $\cO$ must belong to $\{0,\pm 1,\pm\sqrt{2},\pm\sqrt{\kappa},\pm\varphi,\pm\overline{\varphi}\}$. But all such orbits have been listed in Theorem \ref{thm:nonessential}, and they are nonessential. This proves that any $\Gamma$-orbit of essential triples must contain the unique first-coordinate orbit of $\tilde{\alpha}$, implying the uniqueness of such an orbit. This almost shows that Theorem~\ref{thm:main1} (and the slightly weaker $Q$-Classification Conjecture) holds---it remains only to note that the unique $\Gamma$-orbit of essential triples remains unique even if we do not include the double sign change or coordinate permutations as generators in $\Gamma$ (compare Notation~\ref{not:Gamma} to Theorem \ref{thm:main1}). The subgroup of $\Gamma$ generated by Vieta involutions is normal. Furthermore, the parameterizations in Proposition~\ref{prop:orbits} show that if $4\,|\,\text{ord}(\alpha)$ (so that $\zeta^{2n}=-1$ for some $n$), then $(\alpha,-\beta,-\gamma)$ can be obtained from $(\alpha,\beta,\gamma)$ by way of Vieta involutions alone. Also, for each of $\tau_x$, $\tau_y$, and $\tau_z$, there is some essential triple that the transposition preserves. Thus if there were multiple orbits of essential triples under the group generated by Vieta involutions, they would not collapse into a single $\Gamma$-orbit.\end{proof}

\section{Properties of $\Phi$ and $\Phi_x$}\label{sec:5}

\subsection{Degree bounds and the canonical form} Let $R$ be an integral domain with $\text{char}(R)\neq 2$.

\begin{proposition}\label{prop:degree}For any integers $\ell,m,n\geq 0$, the degree of $\Phi(x^\ell y^m z^n)$ is at most $\max\{\ell,n,m\}+\min\{\ell,m,n\}$, with equality if $\ell\equiv m\equiv n\,\textup{mod}\,2$ and $\textup{char}(R)=0$. Furthermore $\Phi(x^\ell y^m z^n)$ is even if $\ell\equiv m\equiv n\,\textup{mod}\,2$ and odd otherwise.\end{proposition}

\begin{proof}The proof proceeds by induction on the total degree $\ell+m+n$. The claim is clear in the base case, which is total degree $0$.

Assume without loss of generality that $\ell\geq m\geq n\geq 0$. If $m=n=0$, then $\Phi(x^\ell)=x^\ell$, and the proposition follows without any induction hypothesis needed. 

If $n=0$ but $m > 0$, then our first $\Phi$ reduction step is \begin{equation}\label{eq:sigma_degree}\sigma_z(x^\ell y^m)=2x^{\ell-1}y^{m-1}z.\end{equation} Comparing the sum of maximum and minimum exponents from either side above, we have $\ell+0 \geq (\ell-1)+\min\{m-1,1\}$, with equality if $m$ is even. Furthermore, exponents on the left side above (which includes $n=0$) are all congruent mod$\,2$ if and only if exponents on the right are all congruent $\text{mod}\,2$. So the proposition follows by the induction hypothesis. 

Finally, if $n > 0$, then our first $\Phi$ reduction step is \begin{equation}\label{eq:rho_degree}\rho(x^\ell y^m z^n)=x^{\ell-1}y^{m-1}z^{n-1}(x^2+y^2+z^2-\kappa).\end{equation} After expanding the right side, the sum of the maximum and minimum exponents in each monomial is at most $\ell+n$, with equality holding for at least the monomial $x^{\ell+1}y^{m-1}z^{n-1}$. And again, congruence of exponents mod$\,2$ has been preserved. We are done by induction.\end{proof}

\begin{proposition}\label{prop:degree_Phi_x}For any integers $\ell,m,n\geq 0$, the degree in $x$ of $\Phi_x(x^\ell y^m z^n)$ is at most $\ell+\min\{m,n\}$, and the degree in $y$ and $z$ combined is at most $m+n$. Furthermore $\Phi(x^\ell y^m z^n)$ is even if $\ell\equiv m\equiv n\,\textup{mod}\,2$ and odd otherwise.\end{proposition}

\begin{proof}As in the previous proposition, the claim follows from comparing degrees on either side of (\ref{eq:sigma_degree}) or (\ref{eq:rho_degree}) and applying induction.\end{proof}

It is much easier to find formulas for $\Phi(x^\ell y^mz^n)$ when one of the exponents is 0. We can rewrite polynomials in $x$ and $y$ only by solving the Markoff equation for $z$: \begin{equation}\label{eq:z_equals}z=\frac{1}{2}\left(xy+\sqrt{x^2y^2-4(x^2+y^2-\kappa)}\right)\end{equation} for some choice of square root. Thus $z$ and $xy-z$ are conjugate with respect to the square root, meaning $f(x,y,z)+f(x,y,xy-z)$ is a polynomial in $x$ and $y$ alone.

\begin{definition}\label{def:canonical}The \emph{canonical form} of $f\in R[x,y,z]$ is $f^*(x,y)=\frac{1}{2}(f(x,y,z)+f(x,y,xy-z))\in R[\frac{1}{2},x,y]$.\end{definition}

The extended coefficient ring $R[\frac{1}{2}]$ for $f^*$ is due to the $\frac{1}{2}$ in (\ref{eq:z_equals}). 

\begin{proposition}\label{prop:canonical_degrees}The canonical form of $x^\ell y^m z^n$ has degree $\ell+n$ in $x$ and degree $m+n$ in $y$. Furthermore, if $f(x)$ denotes the coefficient of $y^{m+n}$ in the canonical form, then $2^nf(x)\equiv x^{\ell+n}\,\textup{mod}\,(x^2-4)$ in $R[x]$, and the coefficient of $x^{\ell+n}$ in $f(x)$ is $1$ if $n=0$ and $\frac{1}{2}$ if $n \geq 1$.\end{proposition}

\begin{proof}The claim is clear if $n=0$ since $x^\ell y^m$ is its own canonical form, so assume $n\geq 1$. By expanding (\ref{eq:z_equals}) raised to the power $n$, the canonical form of $x^\ell y^m z^n$ is \[\frac{x^\ell y^m}{2^n}\sum_{j=0}^{\lfloor\frac{n}{2}\rfloor}\binom{n}{2j}(xy)^{n-2j}(x^2y^2-4(x^2+y^2-\kappa))^{2j}.\] The leading term in the summand with index $j$ is $\binom{n}{2j}x^ny^n$. The sum of every other $n^\text{th}$ binomial coefficient is $2^{n-1}$ when $n\geq 1$, so the leading coefficient of $f(x)$ (notation from the proposition statement) is $\smash{\frac{2^{n-1}}{2^n}}=\frac{1}{2}$ as claimed. Finally, notice in the factor $x^2y^2-4(x^2+y^2-\kappa)$ that the coefficient of $y^2$ is $x^2-4$. So only the summand with index $i=0$ picks up $y^n$ with a nonzero coefficient modulo $x^2-4$.\end{proof}

\begin{proposition}\label{prop:Phi_x(f^*)}For any $f\in R[x,y,z]$, $\Phi_x(f^*)=\Phi_x(f)$.\end{proposition}

\begin{proof}Both canonicalization and $\Phi_x$ are linear over $R$, so we need only check the claim when $f=x^\ell y^mz^n$. If $n=0$ then $f^*=f$, so $\Phi_x(f^*)=\Phi_x(f)$. If $n=1$ then $f^*=\tfrac{1}{2}x^{\ell+1}y^{m+1}$, and since $\sigma_z(\frac{1}{2}x^{\ell+1}y^{m+1})=x^\ell y^m z=f$, we see that $\Phi_x(f^*)=\Phi_x(f)$ in this case as well. If $n\geq 2$, let $g=x^\ell y^mz^{n-2}(xyz+\kappa-x^2-y^2)$. Then $g-f$ is identically $0$ on $\cM(R)$ for any $R$, so $g^*=f^*$. But now observe that the image of $g$ under $\rho$ is $x^\ell y^mz^{n-2}(\rho(xyz)+\kappa-x^2-y^2)=x^\ell y^m z^n$, so $\Phi_x(g)=\Phi_x(f)$. By induction on the degree of $z$, we may assume that $\Phi_x(g^*)=\Phi_x(g)$, from which $\Phi_x(f^*)=\Phi_x(f)$ follows.\end{proof}

\subsection{Partial formulas for $\Phi(f)$} Given $f\in R[x,y,z]$, our goal is to find a formula for as many coefficients of $\Phi(f)$ as possible. We achieve this for generic $f$ in Theorem~\ref{thm:gen_form}. Another formula for special $f$ is presented in Theorem~\ref{thm:spec_form}.

We start by observing that (\ref{eq:Phi_invariance}) and (\ref{eq:Phi_x_invariance}), which state that $\Phi$ and $\Phi_x$ preserve certain sums, are special cases (namely the counting measure) of the following:

\begin{proposition}\label{prop:integral}Let $\cO\subset\cM(R)$ be $\Gamma$-invariant, respectively $\Gamma_{\!x}$-invariant, and let $dA$ be a $\Gamma$-invariant, respectively $\Gamma_{\!x}$-invariant, measure on $\cO$.  Then \[\int_{\cO}\!\!fdA=\int_{\cO}\!\!\Phi(f)dA, \text{ respectively }\int_{\cO}\!\!fdA=\int_{\cO}\!\!\Phi_x(f)dA,\] for any $f\in R[x,y,z]$\end{proposition}

\begin{proof}The only reduction steps for which it is unclear whether the integral is preserved are $\sigma_x$, $\sigma_y$, and $\sigma_z$. The proof in those cases is identical to (\ref{eq:sigma_preserve}) with integrals in place of sums.\end{proof}

\begin{notation}\label{not:M^circ}For a fixed $\kappa\in (0,4)$, let $\cM^\circ$ denote the compact connected component $\cM(\bR)$ with outward orientation. For $\alpha\in(-\sqrt{\kappa},\sqrt{\kappa})$, let $\cO_{\alpha}^\circ\coloneqq\{(x,y,z)\in\cM^\circ\,|\,x=\alpha\}$ with counterclockwise orientation from the perspective of the positive $x$-axis. We write $\cO_x^\circ$ to indicate a generic loop.\end{notation}

The 2-form \[dA\coloneqq\frac{dx\wedge dy}{2z-xy}=\frac{dy\wedge dz}{2x-yz}=\frac{dz\wedge dx}{2y-xz}\] is integrable and nonnegative on $\cM^\circ$. Furthermore, the measure defined by integrating $dA$ is $\Gamma$-invariant. (The measure is also known to be ergodic for the action of $\Gamma$. See \cite{goldman2}, \cite{cantat}, and especially \cite[Section 5]{goldman}.) There is also a $\Gamma_{\!x}$-invariant line measure on any $\cO_{\alpha}^\circ$ defined by the 1-form \[dA_{\alpha}\coloneqq\frac{dy}{\alpha y-2z}=\frac{dz}{2y-\alpha z}.\] As with $\cO_x^\circ$, we write $dA_x$ for generic $x$. Since the 1-form does not appear alongside the 2-form in the literature, we provide a brief proof.

\begin{proposition}\label{prop:dphi}Each $\cO_x^\circ$ is parameterized by \[(x,y,z)=\left(2\cos\theta,2\sqrt{\frac{x^2-\kappa}{x^2-4}}\cos\phi,2\sqrt{\frac{x^2-\kappa}{x^2-4}}\cos(\theta+\phi)\right)\] for $\phi\in[0,2\pi)$. With respect to this parameterization, $\sqrt{4-x^2}dA_x=d\phi$, which induces a $\Gamma_{\!x}$-invariant measure on $\cO_x^\circ$ via integration.\end{proposition}

\begin{proof}For the parameterization, note that for any $\alpha\in(-\sqrt{\kappa},\sqrt{\kappa})$, $\cO_{\alpha}$ is parameterized by case (1) of Proposition \ref{prop:orbits}. We have simply rewritten $\zeta+\zeta^{-1}$ and $\zeta^n\vartheta+\zeta^{-n}\vartheta^{-1}$ as $2\cos\theta$ and $2\cos\phi$, respectively.

Now that we have the parameterization, the formula for $dA_x$ can be checked by differentiating. We omit details.

Let us check $\Gamma_{\!x}$-invariance. The effect of $\tau_x$, $\sigma_z$, and $(x,y,z)\mapsto(x,-y,-z)$ on $\phi$ are $\phi\mapsto \theta+\phi$, $\phi\mapsto\phi$, and $\phi\mapsto\phi+\pi$, respectively. All are shifts, for which $d\phi$ is invariant. Since those three maps generate $\Gamma_{\!x}$, this shows $d\phi$ (and thus $dA_x$) is $\Gamma_{\!x}$-invariant.

Finally, $d\phi$ (and thus $dA_x$) induces a $\Gamma_{\!x}$-invariant measure because it is nonvanishing. Our choice of orientation for $\cO_x^\circ$ agrees with increasing $\phi$.\end{proof}

The whole point of switching to $\bR$ is to utilize the fundamental theorem of calculus. 

\begin{corollary}\label{cor:integral}For any $m\geq 0$, \[\frac{\sqrt{4-x^2}}{2\pi}\!\int_{\cO_x^\circ}\!\!y^{2m}dA_x=\binom{2m}{m}\!\left(\frac{x^2-\kappa}{x^2-4}\right)^{\!\!m}.\]\end{corollary}

\begin{proof}This is checked by direct computation: replace $\sqrt{4-x^2}dA_x$ with $d\phi$ and $y^{2m}$ with $(\frac{x^2-\kappa}{x^2-4})^m(2\cos\phi)^{2m}$ and antidifferentiate over the interval $[0,2\pi)$ with respect to $\phi$.\end{proof}

To see why this is useful, let $f\in R[x,y]$, and suppose for the sake of simplicity that $\Phi_x(f)$, which generally lies in $\bR[x]+\bR[y,z]$, is a polynomial in $x$ only. Then \begin{flalign*}&&\frac{\sqrt{4-x^2}}{2\pi}\int_{\cO_x^\circ}\!\!f\,dA_x &= \frac{\sqrt{4-x^2}}{2\pi}\int_{\cO_x^\circ}\!\!\Phi_x(f)\,dA_x && \text{by Proposition \ref{prop:integral}}\\&& & = \frac{\Phi_x(f)\sqrt{4-x^2}}{2\pi}\int_{\cO_x^\circ}\!\!\,dA_x && \text{since }\Phi_x(f)\text{ is constant on }\cO_x^\circ\\ && &=\Phi_x(f) && \text{by Corollary \ref{cor:integral}.}\end{flalign*} But it is no more trouble to instead evaluate the original integral directly with Corollary \ref{cor:integral}, the result being some rational function, say $g(x)$. We can then expand the denominator of $g(x)$ with Taylor series, thereby solving for the coefficients of $\Phi_x(f)$. 

It is convenient to expand $g(x)$ with respect to the following basis.

\begin{notation}\label{not:h_n}For $n\geq 0$, let \[b_n(x)=\sum_{j=0}^n\binom{2j}{j}\frac{x^{n-j}}{1-2j}.\]\end{notation}

The coefficients are those in the McClaurin series for $\sqrt{1-4x}$. In other words, \begin{equation}\label{eq:b_n}\left(1-\frac{4}{x}\right)^{\!\frac{1}{2}}\!x^n = b_n(x) + \sum_{j=1}^\infty \binom{2j}{j}\frac{x^{-j}}{1-2j}.\end{equation}

The series expansion of $\sqrt{1-4x}$ is just one from a family of expansions that we need. Recall that for any $\alpha\in\bR$, the coefficient of $x^j$ in the McClaurin series expansion of $(1-4x)^\alpha$ is $(-4)^j\alpha(\alpha-1)\cdots (\alpha-j+1)/j!$. When $\alpha=m-\frac{1}{2}$ for some integer $m$, we may rewrite these coefficients as follows: \begin{equation}\label{eq:mcclaurin}\prod_{i=1}^j\frac{1}{i}\big(4i-4m-2\big)=\begin{dcases}\binom{-2m}{-m}^{\!-1}\!\binom{j-m}{j}\binom{2j-2m}{j-m} & m \leq 0,\\[0.1cm] (-1)^m\binom{2m}{m}\binom{j}{m}^{\!-1}\!\binom{2j-2m}{j-m} & 0\leq m\leq j, \\[0.1cm] (-1)^j\binom{2m}{m}\binom{m}{j}\binom{2m-2j}{m-j}^{\!-1} & m \geq j.\end{dcases}\end{equation}

\begin{theorem}\label{thm:gen_form}Over any integral domain $R$ and for any integers $m,n\geq 0$, \[\Phi(x^{2n}y^{2m})=\sum_{j=0}^n\sum_{i=0}^j\binom{2m+2i}{m+i}\binom{m+i}{m}\binom{m}{j-i}(-\kappa)^{j-i}b_{n-j}(x^2)+r(x^2)\] for some $r(x^2)\in R[x]$ of degree at most $2m$.\end{theorem}

\begin{proof}As per Proposition \ref{prop:degree_Phi_x}, $\Phi_x(x^{2n}y^{2m})$ is of the form $f(x^2) + g(y^2,z^2)$, where $g(y^2,z^2)$ has total degree at most $2m$. Thus \begin{align}\label{eq:phicong}\nonumber\Phi(x^{2n}y^{2m})&=\Phi(\Phi_x(x^{2n}y^{2m}))\\\nonumber&=\Phi(f(x^2)+g(y^2,z^2))\\&= f(x^2)+\Phi(g(y^2,z^2)).\end{align}

Proposition \ref{prop:canonical_degrees} tells us we may write the canonical form of $g$ as $g^*=g_m(x^2)y^{2m}+g_{m-1}(x^2)y^{2m-2}+\cdots+g_0(x^2)$ with $\deg g_j\leq 2m$ for all $j$. Now we restrict to $R=\bR$. Applying Corollary \ref{cor:integral} to each power of $y$ in $g^*$ gives \begin{flalign}\label{eq:int_tricks}\nonumber&& f(x^2)\,+\,&\sum_{j=0}^m\binom{2j}{j}\!\left(\frac{x^2-\kappa}{x^2-4}\right)^{\!\!j}g_j(x^2) = \frac{\sqrt{4-x^2}}{2\pi}\int_{\cO_x^\circ}\!\!(f(x^2)+g^*(x^2,y^2))dA_x \hspace{-6cm}&&\\\nonumber&& &=\frac{\sqrt{4-x^2}}{2\pi}\int_{\cO_x^\circ}\!\!(f(x^2)+g(y^2,z^2))dA_x && \text{as }\cO_x^{\circ}\text{ and }dA_x\text{ are }\sigma_z\text{-invariant}\\\nonumber&& &=\frac{\sqrt{4-x^2}}{2\pi}\int_{\cO_x^\circ}\!\!\Phi_x(x^{2n}y^{2m})dA_x && \\\nonumber&& &=\frac{\sqrt{4-x^2}}{2\pi}\int_{\cO_x^\circ}\!\!x^{2n}y^{2m}dA_x && \text{by Proposition \ref{prop:integral}}\\ \nonumber && &=\binom{2m}{m}\frac{x^{2n}(x^2-\kappa)^m}{(x^2-4)^m} &&\text{by Corollary \ref{cor:integral}.} \\ \nonumber && & =\binom{2m}{m}\frac{x^{2n}(1-\frac{4}{x^2})^{\frac{1}{2}}(1-\frac{\kappa}{x^2})^m}{(1-\frac{4}{x^2})^{m+\frac{1}{2}}} && \\ && &=\binom{2m}{m}\!\left(1-\frac{4}{x^2}\right)^{\!\!\frac{1}{2}}\sum_{i=0}^m\binom{m}{i}\frac{(-\kappa)^{m-i}x^{2n-2m+2i}}{(1-\frac{4}{x^2})^{m+\frac{1}{2}}}.\hspace{-5cm} && \end{flalign} While the entire chain of equalities above only holds when $x \in(-\sqrt{\kappa},\sqrt{\kappa})\backslash\{0\}$ and $\kappa\in(0,4)$, it implies equality of the two rational functions at the start and end of the chain. We continue this chain by applying the McClaurin series for $(1-\frac{4}{x^2})^{-m-\frac{1}{2}}$ to obtain the following expression, convergent for $|x| > 2$: \begin{align*}\binom{2m}{m}\!\left(1-\frac{4}{x^2}\right)^{\!\!\frac{1}{2}}&\sum_{i=0}^m\binom{m}{i}\frac{(-\kappa)^{m-i}x^{2n-2m+2i}}{(1-\frac{4}{x^2})^{m+\frac{1}{2}}}\\&=\left(1-\frac{4}{x^2}\right)^{\!\!\frac{1}{2}}\!\sum_{i=0}^m\binom{m}{i}(-\kappa)^{m-i}\!\!\sum_{j=m-i}^\infty\!\binom{2i+2j}{i+j}\binom{i+j}{m}x^{2n-2j}\end{align*} After expanding the square root above, (\ref{eq:b_n}) tells us that the coefficients of nonnegative powers of $x$ match those of \begin{align*}\sum_{i=0}^m\binom{m}{i}(-\kappa)^{m-i}\!\!\sum_{j=m-i}^n&\!\binom{2i+2j}{i+j}\binom{i+j}{m}b_{n-j}(x^2)\\ &= \sum_{j=0}^n\sum_{i=m-j}^m\binom{2i+2j}{i+j}\binom{i+j}{m}\binom{m}{i}(-\kappa)^{m-i}b_{n-j}(x^2)\\&=\sum_{j=0}^n\sum_{i=0}^j\binom{2m+2i}{m+i}\binom{m+i}{m}\binom{m}{j-i}(-\kappa)^{j-i}b_{n-j}(x^2).\end{align*}

Now we connect all the way back to the start of (\ref{eq:int_tricks}). Each rational function $\smash{(\frac{x^2-\kappa}{x^2-4})^jg_j(x^2)}$ can be expressed as a Laurent polynomial in $\frac{1}{x^2}$ in which the largest power of $x$ is $\deg g_j$, which we know to be at most $2m$. So the difference between $f(x^2)$ and the final expression above is some combination of powers $x^{2j}$ with $-\infty < j \leq m$ that converges when $|x| > 2$. Therefore, we have found the first $n-m$ coefficients of $f(x^2)$ (written in terms of $b_{n-j}(x^2)$). According to  (\ref{eq:phicong}), these must also be the first $n-m$ coefficients of $\Phi(x^{2m}y^{2n})$ because $\deg\Phi(g)\leq 2m$ by Proposition \ref{prop:degree}. 

The argument above applies to $R=\bR$. However, only integer coefficients are used throughout the reduction of a monic monomial, and $\bZ$ is the initial object in the category of integral domains. Thus the formula holds for general $R$.\end{proof}

As a small example to test, consider $n=3$ and $m=1$: \begin{align*}\Phi(x^6y^2)&=2x^6 + (8-2\kappa)x^4+(96-8\kappa)x^2-32\kappa\\&=2b_3(x^2) + (12-2\kappa)b_2(x^2) + (124-12\kappa)b_1(x^2)+280-60\kappa.\end{align*} If we ignore the remainder polynomial $r(x^2)$, Theorem \ref{thm:gen_form} provides the following formula: \begin{align*}\binom{2}{1}b_3(x^2)\;+&\left(\!\binom{4}{2}\!\binom{2}{1}-\binom{2}{1}\kappa\right)\!b_2(x^2)\\&+\left(\!\binom{6}{3}\!\binom{3}{1}-\binom{4}{2}\!\binom{2}{1}\kappa\right)\!b_1(x^2)+\left(\!\binom{8}{4}\!\binom{4}{1}-\binom{6}{3}\!\binom{3}{1}\kappa\right)\!b_0(x^2)\end{align*} The coefficients of $b_3(x^2)$ and $b_2(x^2)$ above are correct, but the coefficient of $b_1(x^2)$ is not. The polynomial $r(x^2)$ accounts for this error. Note that the coefficient of $b_0(x^2)$ happens to be correct as well. Interestingly, experimentation suggests that Theorem \ref{thm:gen_form}'s formula for the coefficient of $b_0(x^2)$ is always right. 

The same basic argument can be used to prove the simpler (though, in our case, less useful) formula \begin{equation}\label{eq:spec_form}\Phi(x^{2m}(y^2-\kappa)^{n-m}(y^2-4)^m)=\binom{2m}{m}(x^2-\kappa)^n+r(x^2)\end{equation} for some $r(x^2)\in R[x]$ of degree at most $2m$.

As (\ref{eq:spec_form}) suggests, a $\Phi$ input that is divisible by high powers of $x^2-4$ or $y^2-4$ lends itself to a clean output formula. Theorem \ref{thm:spec_form} below provides such a specialized formula written with respect to the basis $b_0(x^2),b_1(x^2),b_2(x^2),\dots$. It is possible to prove Theorem \ref{thm:spec_form} directly from (\ref{eq:spec_form}), but the proof (at least the one the author found) is far messier and requires many combinatorial identities. The approach taken here only requires the one additional identity in the next lemma.

\begin{notation}Let $n$ be an integer. For a fixed integral domain $R$ let \[\Lambda_n= \{\zeta+\zeta^{-1}\,|\,\zeta\in\overline{R},\,\zeta^{2n}=1\}\hspace{\parindent}\text{and}\hspace{\parindent}\widehat{\Lambda}_n=\Lambda_n\backslash\{\pm2\}.\]\end{notation}

\begin{lemma}\label{lem:special_coef}Let $\ell$, $m$, and $n$ be integers with $\ell\geq 0$, $m\geq 1$, and $n > \ell+m$. The coefficient of $x^{\ell+m}$ in the McClaurin series expansion of $(1-4x)^{m-\frac{1}{2}}$ is \[\frac{1}{n}\sum_{\mathclap{\alpha\in\widehat{\Lambda}_n}}\alpha^{2\ell}(\alpha^2-4)^m.\]\end{lemma}

\begin{proof}Let us rewrite the sum over $\alpha$ as a sum over $\zeta^j+\zeta^{-j}$, where $\zeta$ is a primitive $2n^\text{th}$ root of unity: \[\frac{1}{n}\sum_{\mathclap{\alpha\in\widehat{\Lambda}_n}}\alpha^{2\ell}(\alpha^2-4)^m=\frac{1}{n}\sum_{j=1}^{n-1}\frac{(\zeta^{2j}+1)^{2\ell}(\zeta^{2j}-1)^{2m}}{\zeta^{2j(\ell+m)}}.\] The right-side is the average of the Laurent polynomial $x^{-\ell-m}(x+1)^{2\ell}(x-1)^{2m}$ as $x$ runs over all $n^\text{th}$ roots of unity (including $x=1$ dispite the restriction $1\leq j\leq n-1$; indeed $x^{-\ell-m}(x+1)^{2\ell}(x-1)^{2m}$ vanishes when $x=1$ since $m\geq 1$). But exponents of $x$ in this Laurent polynomial are at most $\ell + m$ in magnitude, which is strictly less than $n$, so this average picks up only the coefficient of $x^0$. In other words, the expression above is the coefficient of $x^{\ell+m}$ in the polynomial $(x+1)^{2\ell}(x-1)^{2m}$, which can be evaluated with Kummer's identity for hypergeometric functions: \begin{flalign*}&& \sum_{i=0}^{\ell+m}(-1)^i\binom{2\ell}{\ell+m-i}\binom{2m}{i}\hspace{-4cm} & && \\ && &=\binom{2m}{\ell+m} {}_2F_1(-\ell-m,-2\ell;m-\ell+1;-1) && \text{(assuming }\ell\leq m\text{)}\\ && &=\binom{2m}{\ell+m}\frac{(-1)^m(2\ell)!(m-\ell)!}{\ell!\,m!} && \text{by Kummer's identity.}\end{flalign*} If $\ell > m$, the roles of $\ell$ and $m$ may be reversed to apply Kummer's identity, arriving at the same final expression either way. This matches the second case of (\ref{eq:mcclaurin}) with $i$ replaced by $\ell+m$.\end{proof}

\begin{theorem}\label{thm:spec_form}Let $f(y^2)\in R[y]$, and let $m\geq 0$ be such that $(y^2-4)^m\mid f(y^2)$. For any positive integers $\tilde{n}\geq n$ with $\textup{char}(R)\centernot\mid\tilde{n}$, \[\Phi(x^{2n}f(y^2))=\frac{1}{\tilde{n}}\sum_{j=0}^n\binom{2j}{j}\sum_{\mathclap{\alpha\in\widehat{\Lambda}_{\tilde{n}}}}\left(\frac{\alpha^2-\kappa}{\alpha^2-4}\right)^{\!\!j}\!f(\alpha^2) b_{n-j}(x^2)+r(x^2)\] for some $r(x^2)\in R[x]$ of degree at most $\max(\deg f(y^2),2n-2m)$.\end{theorem}

\begin{proof}It suffices to consider polynomials of the form $f(y^2)=y^{2\ell}(y^2-4)^m$. 

We start by expanding $x^{2n}y^{2\ell}(y^2-4)^m$ and applying Theorem \ref{thm:gen_form} to each of the resulting monomials. This gives the following expression for the coefficient of $b_{n-j}(x^2)$ in $\Phi(x^{2n}y^{2\ell}(y^2-4)^m)$: \begin{align*}\sum_{k=0}^m(-4)^{m-k}\binom{m}{k}\sum_{i=0}^j\binom{2i+2k+2\ell}{i+k+\ell}\binom{i+k+\ell}{k+\ell}\binom{k+\ell}{j-i}(-\kappa)^{j-i}\hspace{-10cm}&\\&=\sum_{k=0}^m(-4)^{m-k}\binom{m}{k}\sum_{i=0}^j\binom{2i+2k+2\ell}{i+k+\ell}\binom{i+k+\ell}{j}\binom{j}{i}(-\kappa)^{j-i}\\&=\binom{2j}{j}\!\sum_{i=0}^j\binom{j}{i}(-\kappa)^{j-i}\sum_{k=0}^m(-4)^{m-k}\binom{m}{k}\!\left(\!\binom{2j}{j}^{\!\!-1}\!\binom{2i+2k+2\ell}{i+k+\ell}\binom{i+k+\ell}{j}\!\right)\!,\end{align*} which holds provided $n-j > \ell+m$ (otherwise Theorem \ref{thm:gen_form} tells us nothing about the coefficient of $b_{n-j}(x^2)$). 

Now we restrict to $R=\bR$. Recall from the first case in (\ref{eq:mcclaurin}) that the final parenthesized product of binomial coefficients above is the coefficient of $x^{i-j+k+\ell}$ in the McClaurin series of $(1-4x)^{-j-\frac{1}{2}}$. But observe that $\sum_k(-4x)^{m-k}\binom{m}{k}=(1-4x)^m$, so the entire inner sum in the final expression above is the coefficient of $x^{i-j+\ell+m}$ in the McClaurin series of $(1-4x)^{m-j-\frac{1}{2}}$. When $m- j\geq 1$ and $\tilde{n} >i-j+ \ell + m$, we have a formula for this coefficient from Lemma \ref{lem:special_coef}. Picking up where we left off with the final line above, we have shown \begin{align*}\sum_{i=0}^j\binom{j}{i}(-\kappa)^{j-i}\sum_{k=0}^m(-4)^{m-k}\binom{m}{k}\binom{2i+2k+2\ell}{i+k+\ell}\binom{i+k+\ell}{j} \hspace{-6cm}&\\& = \frac{1}{\tilde{n}}\binom{2j}{j}\sum_{i=0}^j\binom{j}{i}(-\kappa)^{j-i}\sum_{\mathclap{\alpha\in\widehat{\Lambda}_{\tilde{n}}}}\alpha^{2i+2\ell}(\alpha^2-4)^{m-j}\\&=\frac{1}{\tilde{n}}\binom{2j}{j}\sum_{\mathclap{\alpha\in\widehat{\Lambda}_{\tilde{n}}}}\alpha^{2\ell}(\alpha^2-4)^{m-j}\sum_{i=0}^j\binom{j}{i}\alpha^{2i}(-\kappa)^{j-i}\\&=\frac{1}{\tilde{n}}\binom{2j}{j}\sum_{\mathclap{\alpha\in\widehat{\Lambda}_{\tilde{n}}}}\alpha^{2\ell}(\alpha^2-4)^{m-j}(\alpha^2-\kappa)^{j}\\&=\frac{1}{\tilde{n}}\binom{2j}{j}\sum_{\mathclap{\alpha\in\widehat{\Lambda}_{\tilde{n}}}}\left(\frac{\alpha^2-\kappa}{\alpha^2-4}\right)^{\!\!j}\!f(\alpha^2).\end{align*} This is the desired expression for the coefficient of $b_{n-j}(x^2)$. For this to hold, recall that the necessary inequalities are $n-j > \ell+m$, $m-j> 0$, and  $\tilde{n} >i-j+ \ell + m$, which (since $\tilde{n}\geq n$) are equivalent to $2n-2j > 2\max( \ell+m,n-m) = \max(\deg y^{2\ell}(y^2-4)^m,2n-2m)$ as in the theorem statement.

Finally, we remove the restriction to $\bR$. Observe that over $\bZ$, the only factor in our formula for $\Phi(x^{2n}f(y^2))$ that is not an element of $\oZ$ is $\frac{1}{\tilde{n}}$ (remember, $(y^2-4)^i$ is assumed to divide $f(y^2)$ when $i\leq m$). Since $\text{char}(R)\centernot\mid \tilde{n}$, this expression is well defined in $\overline{F}$. So because \[
\begin{tikzcd}[sep=small]
\bZ \arrow[d] \arrow[hookrightarrow,r] & \oQ \arrow[d]\\
R \arrow[hookrightarrow,r]& \overline{F}
\end{tikzcd}
\] commutes, if arithmetic in $\oQ$ makes our coefficient of $b_{n-j}$ the correct element of $\bZ$, then arithmetic in $\overline{F}$ makes it the correct element of $R$.\end{proof}

Consider, for example, the reduction of $x^{10}(y^2-4)^2$, and suppose for simplicity that $\kappa=1$: \begin{align*}\Phi(x^{10}(y^2-4)^2)&=6x^{10}-12x^8+6x^6+752x^4+6944x^2-2560\\&=6b_5(x^2)+18b_3(x^2)+812b_2(x^2)+8664b_1(x^2)+16632.\end{align*} If we ignore the remainder polynomial $r(x^2)$, Theorem \ref{thm:spec_form} provides the following formulas for the coefficients of $b_5(x^2)$, $b_4(x^2)$, and $b_3(x^2)$: \[\frac{1}{\tilde{n}}\sum_{\mathclap{\alpha\in\widehat{\Lambda}_{\tilde{n}}}}(\alpha^2-4)^2,\hspace{\parindent}\frac{2}{\tilde{n}}\sum_{\mathclap{\alpha\in\widehat{\Lambda}_{\tilde{n}}}}(\alpha^2-4)(\alpha^2-1),\hspace{\parindent}\text{and}\hspace{\parindent}\frac{6}{\tilde{n}}\sum_{\mathclap{\alpha\in\widehat{\Lambda}_{\tilde{n}}}}(\alpha^2-1)^2\] for any $\tilde{n}\geq 5$. These equal $6$, $0$, and $18(1-\frac{3}{\tilde{n}})$, respectively. In particular, the coefficients of $b_5(x^2)$ and $b_4(x^2)$ match, while the coefficient of $b_3(x^2)$ does not. The polynomial $r(x^2)$ accounts for this error because $\max(\deg(y^2-4)^2,10-4) = 6$.

\section{Eigenvectors for $\alpha\neq\pm 2$}\label{sec:6}

As described in Section \ref{sec:3}, we aim to build rank in $\cP(\bF_q)$ by finding polynomials $f\in\oZ[x,y,z]$ that satisfy $\Phi_x(xf)=\alpha f$ for some $\alpha\in\oZ$ whose image under a projection $\oZ\to\oF_q$ never occurs as an entry in $\cM(\bF_q)$. An example with $\alpha=0$ can be found in (\ref{eq:phi_x}).

Although $\oZ$ is the ring that matters, we continue to work over an arbitrary integral domain $R$ with $\textup{char}(R)\neq 2$.

\subsection{Eigenvector existence}\label{ss:6.1} Let $n\geq 1$, and for each $i=0,...,n$ define the $(i+1)$-element set \[\cB_{i,n}=\{(x^2-\kappa)^{n-i}y^{j}z^{k}\mid j+k=2i,\,j\geq k\}\subset R[x^2,y,z].\] Order the elements of $\cB_{i,n}$ according to decreasing $j$/increasing $k$. Let $\cB_n=\cup_i\cB_{i,n}$ ordered primarily according to decreasing $i$, then according to the order on each $\cB_{i,n}$. For a linear combination $f$ of elements of $\cB_n$, we wish to find another combination from $\cB_n$, call it $g$, satisfying $\Phi_x(xf)=\Phi_x(g)$. From this, the effect of multiplying by $x$ and applying $\Phi_x$ can be represented by a matrix whose eigenvectors (with entries corresponding to polynomial coefficients with respect to $\cB_n$) are easily computed. 

Remark that the sum of degrees in $y$ and $z$ in each monomial is always even. We might also define something like $\widetilde{\cB}_{i,n}$ in which $y$ and $z$ degrees sum to $2i+1$, but there would no ``interaction" between $\cB_n$ and $\widetilde{\cB}_n$ (the meaning of this is made explicit below); we would be doing double the work for nothing---we already know that $\cP(R)$ contains all odd degree polynomials anyway due to the inclusion of $(x,y,z)\mapsto(x,-y,-z)$ in $\Gamma$. However, the lack of consideration for ``$\widetilde{\cB}_n$" forces a restricted hypothesis in many of this section's results. Namely, we only prove properties of those $f\in R[x,y,z]$ for which $f^*$ is even as a polynomial in $y$, as this is true of every monomial in $\cB_n$. 

Let $i > 0$. Consider what happens when we multiply the first element of $\cB_{i,n}$ by $x$ and reduce: \begin{equation}\label{eq:eigen0}x(x^2-\kappa)^{n-i}y^{2i}\xmapsto{\sigma_z} 2(x^2-\kappa)^{n-i}y^{2i-1}z.\end{equation} This is twice the second element of $\cB_{i,n}$. 

Now consider the $(k+1)^\text{th}$ element of $\cB_{i,n}$ for $0 < k < 2i$: \begin{align}\label{eq:eigen1}\nonumber x(x^2-&\kappa)^{n-i}y^jz^k\\\nonumber &\xmapsto{\rho} (x^2-\kappa)^{n-i}(x^2y^{j-1}z^{k-1}+y^{j+1}z^{k-1}+y^{j-1}z^{k+1}-\kappa y^{j-1}z^{k-1})\\&=(x^2-\kappa)^{n-i}y^{j+1}z^{k-1}+(x^2-\kappa)^iy^{j-1}z^{k+1}+(x^2-\kappa)^{n-i+1}y^{j-1}z^{k-1}.\end{align} This is the sum of the $k^\text{th}$ and $(k+2)^\text{th}$ elements of $\cB_{i,n}$ and the $k^\text{th}$ element of $\cB_{i-1,n}$. Finally, consider the last element of $\cB_{i,n}$ (still with $i > 0$): \begin{align}\label{eq:eigen2}\nonumber x(x^2-\kappa)^{n-i}y^iz^i&\xmapsto{\rho} (x^2-\kappa)^{n-i}(x^2y^{i-1}z^{i-1}+y^{i+1}z^{i-1}+y^{i-1}z^{i+1}-\kappa y^{i-1}z^{i-1})\\\nonumber&\xmapsto{\tau_x}(x^2-\kappa)^{n-i}(x^2y^{i-1}z^{i-1}+2y^{i+1}z^{i-1}-\kappa y^{i-1}z^{i-1})\\&= 2(x^2-\kappa)^{n-i}y^{i+1}z^{i-1}+(x^2-\kappa)^{n-i+1}y^{i-1}z^{i-1}.\end{align} This is twice the penultimate element of $\cB_{i,n}$ plus the last element of $\cB_{i-1,n}$.

In light of (\ref{eq:eigen0}), (\ref{eq:eigen1}), and (\ref{eq:eigen2}), we define \begin{equation}\label{eq:A_n}\setlength\arraycolsep{2.5pt}\renewcommand{\arraystretch}{0.9}A_0=\begin{bmatrix}2\end{bmatrix},\hspace{\parindent}A_1=\begin{bmatrix}0 & 2\\ 2 & 0\end{bmatrix},\hspace{0.5\parindent}\text{and}\hspace{0.5\parindent}A_n=\begin{bmatrix} & 1 & & & & \\ 2 & & 1 & & \\ & 1 & & {\smash\ddots} & &\\ & & 1 & & 1 &\\ & & & {\smash\ddots} & & 2\\ & & & & 1 & \end{bmatrix}\text{ for }n\geq 2,\end{equation} where blank entries are 0. These matrices contain the coefficients from those terms in (\ref{eq:eigen0}), (\ref{eq:eigen1}), and (\ref{eq:eigen2}) that come from $\cB_{i,n}$, not $\cB_{i-1,n}$. To account for the coefficients of terms from $\cB_{i-1,n}$, we define the $n\times (n+1)$ matrix \[\setlength\arraycolsep{3pt}B_n=\begin{bmatrix}0 & 1 & & &\\ 0 & & 1 & &\\ {\smash\vdots}& & & {\smash\ddots} & \\ 0& & & & 1\end{bmatrix}\!,\] the $n\times n$ identity matrix with the $0$ column appended to its left side. Finally, for $n\geq 0$ let \begin{equation}\label{eq:M_n}\setlength\arraycolsep{1.5pt}M_n=\begin{bmatrix}A_n & & & &\\ B_n & A_{n-1} & & &\\ & B_{n-1} & {\smash\ddots} & &\\ & & {\smash\ddots} & A_1 &\\ & & & B_1 & A_0\end{bmatrix}\!.\end{equation} This is a square, almost block diagonal matrix over $R$ with $\frac{1}{2}(n^2+3n+2)$ rows and columns.

Note that we have only considered the effect of multiplying by $x$ and applying $\Phi_x$ to elements of $\cB_{i,n}$ when $i > 0$. Regarding $\cB_{0,n}=\{(x^2-\kappa)^n\}$, applying $\Phi_x$ to $x(x^2-\kappa)^n$ does nothing, and this polynomial does not belong to $\cB_n$. In other words, our effort to express $\Phi_x(xf)$ as a linear combination from $\cB_n$ fails if the coefficient of $(x^2-\kappa)^n$ in $f$ is nonzero.

\begin{proposition}\label{prop:Phi_to_mat}Let $f\in R[x,y,z]$ be a linear combination from $\cB_n$ and let $\bf$ be the corresponding column vector of coefficients. Then \[\Phi_x(xf)=\Phi_x(g)+\omega x(x^2-\kappa)^n,\] where $g$ is the polynomial corresponding to $M_n\bf$ and $\omega$ is the last entry of $\bf$.\end{proposition}

\begin{proof}This is the combination of (\ref{eq:eigen0}), (\ref{eq:eigen1}), and (\ref{eq:eigen2}) and the observation that immediately precedes the proposition.\end{proof}

In light of Proposition \ref{prop:Phi_to_mat}, the goal is to find eigenvectors $\bf$ of $M_n$ with final entry $\omega=0$. This provides polynomials with the following property, which generalizes (\ref{eq:first_phi_x}) to nonzero $\alpha$.

\begin{proposition}\label{prop:eigen_works}Suppose $f\in R[x,y,z]$ satisfies $\Phi_x(xf)=\alpha\Phi_x(f)$ for some $\alpha\in R$. Then $\sum_{\cO_{\tilde{\alpha}}}\!f(\bt)=0$ for any finite first-coordinate orbit $\cO_{\tilde{\alpha}}$ with $\tilde{\alpha}\neq\alpha$.\end{proposition}

\begin{proof}If $\tilde{\alpha}\neq\alpha$ then \begin{align*}\sum_{\bt\in\cO_{\tilde{\alpha}}}f(\bt)&=\frac{1}{\tilde{\alpha}-\alpha}\sum_{\bt\in\cO_{\tilde{\alpha}}}(x-\alpha)f(\bt)\\&=\frac{1}{\tilde{\alpha}-\alpha}\sum_{\bt\in\cO_{\tilde{\alpha}}}\Phi_x((x-\alpha)f)(\bt)\\&=\frac{1}{\tilde{\alpha}-\alpha}\sum_{\bt\in\cO_{\tilde{\alpha}}}(\alpha-\alpha)\Phi_x(f)(\bt)=0,\end{align*} as claimed.\end{proof}

\begin{lemma}\label{lem:eigenvector}Suppose $\textup{char}(R)\centernot\mid n$. Each element of $\Lambda_n$ is an eigenvalue of $A_n$. The eigenvector (unique up to scaling) corresponding to $\zeta+\zeta^{-1}$ is \begin{equation}\label{eq:eigenvector}\renewcommand\arraystretch{1.15}\begin{bmatrix}1 \\ \zeta+\zeta^{-1} \\ \zeta^{2}+\zeta^{-2} \\ \vdots\\ \zeta^{n-1} + \zeta^{1-n} \\ \zeta^{n}\end{bmatrix}.\end{equation}\end{lemma}

\begin{proof}This claim is a repeated application of \[(\zeta^{j-1}+\zeta^{1-j})+(\zeta^{j+1}+\zeta^{-j-1})=(\zeta+\zeta^{-1})(\zeta^j+\zeta^{-j})\] as $j$ ranges over the coordinates of the product of $A_n$ and (\ref{eq:eigenvector}).\end{proof}

\begin{definition}Define $u,v$-coordinates on those points $(x,y,z)\in\cM(R)$ with $x^2\neq\kappa,4$ by \[(x,y,z)=\left(u+u^{-1},\sqrt{\frac{\kappa-x^2}{4-x^2}}(v+v^{-1}),\sqrt{\frac{\kappa-x^2}{4-x^2}}(uv+u^{-1}v^{-1})\right).\] Note that $(u,v)$ and $(u^{-1},v^{-1})$ define the same point.\end{definition}

\begin{lemma}\label{lem:uv-sums}Let $\alpha\in R$ with $\textup{ord}(\alpha)\geq 4$. Suppose $f\in R[x,y,z]$ has canonical form $f^*=f_n(x)y^{2n}+\cdots +f_0(x)$. For $i\geq 0$ define \begin{equation}\label{eq:uv-sums}c_i\coloneqq\begin{dcases}0 & \alpha\not\in\Lambda_i,\\ \sum_{j=i}^n\binom{2j}{j-i}\!\left(\frac{\alpha^2-\kappa}{\alpha^2-4}\right)^{\!j}\!f_j(\alpha) & \alpha\in\Lambda_i\backslash\{\pm\sqrt{\kappa}\},\\ f_i(\alpha) & \alpha=\pm\sqrt{\kappa}\in\Lambda_i.\end{dcases}\end{equation} Then \[\frac{1}{|\cO_{\alpha}|}\sum_{\bt\in\cO_{\alpha}}\!f(\bt)=\begin{dcases}c_0+\sum_{i=1}^nc_i(v^{2i}+v^{-2i}) & \alpha^2\neq\kappa\\\sum_{i=0}^nc_iy^{2i} & \alpha^2=\kappa.\end{dcases}\]\end{lemma}

Remark that the values of $y^{2i}$ and $v^{2i}+v^{-2i}$ are constant on $\cO_{\alpha}$ by Proposition~\ref{prop:orbits}. The formula above is not ill-defined.

\begin{proof}Consider what happens when $f^*=f_j(x)y^{2j}$ and $\alpha^2\neq \kappa$. First we switch to $u,v$-coordinates: \begin{align*}\frac{1}{|\cO_{\alpha}|}\sum_{\bt\in\cO_{\alpha}}f(\bt)&=\frac{1}{|\cO_{\alpha}|}\sum_{\bt\in\cO_{\alpha}}f_j(\alpha)y^{2j}\\&=\frac{1}{|\cO_{\alpha}|}\sum_{\bt\in\cO_{\alpha}}f_j(\alpha)\!\left(\frac{\alpha^2-\kappa}{\alpha^2-4}\right)^{\!j}\!(v+v^{-1})^{2j}\\&=\frac{f_j(\alpha)}{|\cO_{\alpha}|}\!\left(\frac{\alpha^2-\kappa}{\alpha^2-4}\right)^{\!j}\sum_{i=-j}^{j}\binom{2j}{j-i}\sum_{\bt\in\cO_{\alpha}}v^{2i}.\end{align*} By the first case in Proposition \ref{prop:orbits}, the $v$-coordinates in $\cO_\alpha$ run over all powers of a $\textup{ord}(\alpha)^\text{th}$ root of unity (times some constant ``$\vartheta$" that depends on exactly which first-coordinate orbit we are in). Thus the inner sum vanishes when $\textup{ord}(\alpha)\centernot\mid 2i$. And when $\textup{ord}(\alpha)\,|\,2i$, the value of $v^{2i}$ is constant on $\cO_\alpha$, which cancels the denominator $|\cO_{\alpha}|$. So the lemma holds with the coefficients defined by \[c_i=\begin{dcases}0 & \alpha\not\in\Lambda_i,\\  \binom{2j}{j-i}\!\left(\frac{\alpha^2-\kappa}{\alpha^2-4}\right)^{\!j}\!f_j(\alpha) & \alpha\not\in\Lambda_i\end{dcases}\] for $i\geq 1$. Summing over $j$ for the more general $f^*=f_n(x)y^{2n}+\cdots +f_0(x)$ completes the proof when $\alpha^2\neq\kappa$. 

The argument when $\alpha^2=\kappa$ is almost identical, just using the second case of Proposition~\ref{prop:orbits} rather than the first case. \end{proof}

\begin{notation}Given $f$ and $\alpha$ as in Lemma \ref{lem:uv-sums}, let $c_{\alpha,i}(f)$ denote the value of $c_i$ as defined in (\ref{eq:uv-sums}).\end{notation}

Note that we are only defining $c_{\alpha,i}(f)$ when $f^*$ is even as a polynomial in $y$.

\begin{proposition}\label{prop:pull_out}Let $f\in R[x,y,z]$, and suppose $f^*$ is even as a polynomial in $y$. Then $c_{\alpha,i}(gf)=g(\alpha)c_{\alpha,i}(f)$ for any $g\in R[x]$, $i\geq 0$, and $\alpha\in \overline{R}$.\end{proposition}

\begin{proof}If $f^*=f_n(x)y^{2n}+\cdots+f_0(x)$ then $(gf)^*=g(x)f_n(x)y^{2n}+\cdots+g(x)f_0(x)$. The claim then follows from the formula for $c_{\alpha,i}(gf)$ in (\ref{eq:uv-sums}).\end{proof}

\begin{proposition}\label{prop:P_equiv}Suppose $\textup{char}(R)=0$. Let $f\in \overline{R}[x,y,z]$ have canonical form $f_n(x)y^{2n}+f_{n-1}(x)y^{2n-2}+\cdots +f_0(x)$. The following are equivalent: \begin{enumerate}\item $f\in \cP_{\!x}(R,\infty)$ (see Notation \ref{not:P_x(R,d)}), \item $c_{\alpha,i}(f)=0$ for all but finitely many pairs $\alpha,i$, \item $c_{\alpha,0}(f)=0$ for all $\alpha\in R$, \item $c_{\alpha,0}(f)=0$ for all but finitely many $\alpha\in R$,\item and $\sum_j\! \binom{2j}{j}(\frac{x^2-\kappa}{x^2-4})^jf_j(x)$ is identically 0.\end{enumerate} \end{proposition}

\begin{proof}First observe that Lemma \ref{lem:uv-sums} expresses $\sum_{\cO_\alpha}\!f(\bt)$ in terms of the free variables $v$ and $y$ in an infinite integral domain. So for some fixed $\alpha$, $\sum_{\cO_\alpha}\!f(\bt)$ can only vanish for all $\cO_{\alpha}$ if $c_{\alpha,i}(f)=0$ for all $i$. In particular, (1) and (2) are equivalent.

Next, (\ref{eq:uv-sums}) defines $c_{\alpha,0}(f)$ as the rational expression in (5) evaluated at $x=\alpha$. Any rational expression with infinitely many roots must be identically $0$, so (3), (4) and (5) are equivalent, and they and are implied by (1). 

To see that (3) implies (2), observe from (\ref{eq:uv-sums}) that when $2i > \deg_y f^*= 2n$, $c_{\alpha,i}(f)=0$ for all $\alpha\in R$. There are only finitely many $i$ with $0 < i \leq n$, and for each one $\widehat{\Lambda}_i$ is finite (crucially using $i > 0$).\end{proof} 

\begin{lemma}\label{lem:norm_form}For any $n\geq 2$ and any $\alpha=\zeta+\zeta^{-1}\in\widehat{\Lambda}_n$, \[\prod_{\mathclap{\tilde{\alpha}\in\Lambda_n\backslash\{\alpha\}}}\,(\alpha - \tilde{\alpha}) = 2n\zeta^n .\]\end{lemma}

\begin{proof}Let $\tilde{\zeta}$ be a primitive $2n^\text{th}$ root of unity, and let $1\leq i\leq n-1$ satisfy $\tilde{\zeta}^i = \zeta$. We have \begin{flalign*}&&\prod_{\mathclap{\tilde{\alpha}\in\Lambda_n\backslash\{\alpha\}}}\,(\alpha - \tilde{\alpha})&=\prod_{\substack{j=0\\ j\neq i}}^{n}((\tilde{\zeta}^i+\tilde{\zeta}^{-i})-(\tilde{\zeta}^j+\tilde{\zeta}^{-j}))&&\\&& &=\prod_{\substack{j=0\\j\neq i}}^n\zeta^{-1}\big(\tilde{\zeta}^{i+j}-1\big)\big(\tilde{\zeta}^{i-j}-1\big) &&\\&& &=\zeta^n(\tilde{\zeta}^i-1)(\tilde{\zeta}^{i+j}-1)\!\!\prod_{\substack{j=1\\j\neq 2i}}^{2n-1}\!(\tilde{\zeta}^j-1)&& \text{as }\tilde{\zeta}^{i\pm j}\text{ are distinct if }j\neq0,n\\&& &=-\zeta^n(\tilde{\zeta}^{2i}-1)\!\prod_{\substack{j=1\\j\neq 2i}}^{2n-1}(\tilde{\zeta}^j-1) && \\ && &=\zeta^n\!\prod_{j=1}^{2n-1}(1-\tilde{\zeta}^j).&&\end{flalign*} The final product is $1+x+\cdots +x^{2n-1}$ evaluated at $x=1$, which is $2n$.\end{proof}

\begin{theorem}\label{thm:eigenvector}Suppose $\textup{char}(R)\centernot\mid n$. Let $\alpha=\zeta+\zeta^{-1}\in\widehat{\Lambda}_n$. After perhaps scaling (\ref{eq:eigenvector}), it extends uniquely to an eigenvector of $M_n$ over $F[\zeta]$ such that the corresponding (with respect to $\cB_n$) polynomial $\scp$ satisfies $\Phi_x(x\scp)=\alpha\Phi_x(\scp)$ and \begin{equation}\label{eq:sum_of_p}\frac{1}{|\cO_\alpha|}\!\sum_{\bt\in\cO_\alpha}\!\scp(\bt) =\begin{dcases}v^{2n}+v^{-2n} & \alpha^2\neq\kappa\\y^{2n} & \alpha^2=\kappa.\end{dcases}\end{equation} Furthermore, if $\scp^* = p_n(x)y^{2n}+\cdots+p_0(x)$, then \begin{equation}\label{eq:lead_coef}p_n(x)=\begin{dcases}\frac{\zeta^n}{2n}\!\left(\frac{\alpha^2-4}{\alpha^2-\kappa}\right)^{\!\!n}\!\!\!\!\!\prod_{\tilde{\alpha}\in\Lambda_n\backslash\{\alpha\}}\!\!\!\!\!(x-\tilde{\alpha}) & \alpha^2\neq\kappa\\ \frac{\zeta^n}{2n}\!\!\prod_{\tilde{\alpha}\in\Lambda_n\backslash\{\alpha\}}\!\!\!\!\!(x-\tilde{\alpha}) & \alpha^2=\kappa.\end{dcases}\end{equation}\end{theorem}

\begin{proof}We proceed by induction on $n$. If $n=2$ (the base case), or more generally if $\zeta$ is a \textit{primitive} $2n^{\text{th}}$ root of unity, then Lemma \ref{lem:eigenvector} tells us $A_n$ is the only matrix along the diagonal of $M_n$ with $\alpha$ as an eigenvalue. In particular, $\alpha$ has algebraic multiplicity one for $M_n$, so linear algebra provides an eigenvector $\bp$ of $M_n$ with the first $n+1$ entries matching (\ref{eq:eigenvector}). (We will scale $\bp$ appropriately so that (\ref{eq:sum_of_p}) and (\ref{eq:lead_coef}) hold at the end of the proof.)

Now consider the possibility that $\zeta$ is a $2\tilde{n}^\text{th}$ root of unity for some positive $\tilde{n}<n$. As an induction hypothesis, assume the present theorem holds in smaller dimensions. Unlike in the base case, linear algebra makes no eigenvector guarantee; we only know that (\ref{eq:eigenvector}) extends to a vector $\bp$ such that \begin{equation}\label{eq:tildep_def}M_n\bp=\alpha\bp+\tilde{\bp}\end{equation} for some $\tilde{\bp}$ in the $\alpha$-generalized eigenspace of $M_n$.

Since (\ref{eq:eigenvector}) is an eigenvector of $A_n$, we see from $\tilde{\bp}=M_n\bp-\alpha\bp$ that the first $n+1$ coefficients of $\tilde{\bp}$ are $0$. These initial zeros may be removed to view $\tilde{\bp}$ as an element of the $\alpha$-generalized eigenspace of $M_{n-1}$, which is a genuine eigenspace by the induction hypothesis. So up to some scalar, say $\lambda$, the first nonzero entries of $\tilde{\bp}$ must match (\ref{eq:eigenvector}) in some dimension $\tilde{n}<n$. Assume by way of contradiction that $\lambda\neq 0$, meaning $\tilde{\bp}$ is an eigenvector. Using the induction hypothesis again, $\tilde{\bp}$ corresponds to some $\lambda\tilde{\scp}$ with respect to $\cB_{\tilde{n}}$, where $\tilde{\scp}$ satisfies (\ref{eq:sum_of_p}) with $n$ replaced by $\tilde{n}$. Comparing definitions of $\cB_n$ and $\cB_{\tilde{n}}$, we see that $\tilde{\bp}$ corresponds to $\lambda(x^2-\kappa)^{n-\tilde{n}}\tilde{\scp}$ with respect to $\cB_n$. 

Now let $\scp$ be the polynomial corresponding to $\bp$ with respect to $\cB_n$ and let $\omega\in\overline{F}$ denote the last entry of $\bp$. Then (\ref{eq:tildep_def}) combines with Proposition \ref{prop:Phi_to_mat}  to give \begin{equation}\label{eq:Phi(p)}\Phi_x(x\scp)=\alpha \Phi_x(\scp) +\omega x(x^2-\kappa)^n + \lambda\Phi_x((x^2-\kappa)^{n-\tilde{n}}\tilde{\scp}).\end{equation} If $\alpha^2\neq \kappa$ then \begin{flalign*}&& &\frac{\lambda\tilde{n}(\alpha^2-\kappa)^n}{(\alpha^2-4)^{\tilde{n}}}(v^{2\tilde{n}}+v^{-2\tilde{n}})=\frac{\lambda(\alpha^2-\kappa)^{n-\tilde{n}}}{|\cO_\alpha|}\!\sum_{\bt\in\cO_\alpha}\!\tilde{\scp}(\bt)&&\text{by (\ref{eq:sum_of_p}) for }\tilde{\scp}\\ && & \hspace{1cm}=\frac{\lambda}{|\cO_\alpha|}\!\sum_{\bt\in\cO_\alpha}\Phi_x((x^2-\kappa)^{n-\tilde{n}}\tilde{\scp})(\bt)&& \\&& & \hspace{1cm}=\frac{1}{|\cO_\alpha|}\!\sum_{\bt\in\cO_\alpha}\!\!\big(\Phi_x((x-\alpha)\scp)(\bt)-\omega x(x^2-\kappa)^n\big)&& \text{by (\ref{eq:Phi(p)})}\\&& &\hspace{1cm}=\omega\alpha(\alpha^2-\kappa)^n&& \text{since }x=\alpha\text{ in }\cO_\alpha.\end{flalign*} The final expression is constant in the free variable $v$ while the initial expression is not. This is a contradiction, so $\lambda$ must be $0$ and $\tilde{\scp}$ must be the zero vector. But now note that setting $\lambda=0$ above proves $\omega\alpha(\alpha^2-\kappa)^n=0$. So if $\alpha^2\neq \kappa$ then $\omega=0$. The definition of $M_n$ is independent of the value of $\kappa$, so $\bp$ is an eigenvector with final entry $\omega=0$ even when $\alpha^2=\kappa$. Thus $\Phi_x(x\scp)=\alpha\Phi_x(\scp)$ as desired. 

Next let $\scp^*=p_n(x)y^{2n}+\cdots +p_0(x)$, and consider the expression for $\sum_{\cO_\alpha}\!\scp(\bt)$ provided by Lemma \ref{lem:uv-sums}. The only powers of $y$ or $v$ that appear are those with exponent divisible by the order of $\zeta$. So by the induction hypothesis, we may adjust $\scp$ by the $\alpha$-eigenvectors from smaller dimensions in order to cancel all but the top term in Lemma \ref{lem:uv-sums}'s formula. That is, we may assume \begin{equation}\label{eq:lambda_sum}\frac{1}{|\cO_{\alpha}|}\!\sum_{\bt\in\cO_{\alpha}}\scp(\bt)=\begin{dcases}p_n(\alpha)\!\left(\frac{\alpha^2-\kappa}{\alpha^2-4}\right)^{\!\!n}\!\!(v^{2n}+v^{-2n}) & \alpha^2\neq\kappa\\p_n(\alpha)y^{2n} & \alpha^2=\kappa.\end{dcases}\end{equation} It remains only to find a formula for $p_n(x)$ and use it to find the right scalar to make (\ref{eq:sum_of_p}) and (\ref{eq:lead_coef}) true. 

By Proposition \ref{prop:canonical_degrees}, the $y$-degree of the canonical form of any monomial from $\cB_{i,n}$ is $2i$. In particular, $p_n(x)$ is completely determined by the coefficients of monomials from $\cB_{n,n}$, which are given in (\ref{eq:eigenvector}). From this, and again using Proposition \ref{prop:canonical_degrees}, we see that $p_n(x)$ has degree $n$ and leading coefficient $\frac{\zeta^{n}}{2}$ due to the term $\zeta^ny^nz^n$ present in $\scp$. 

Next we determine the roots of $p_n(x)$. Suppose $\tilde{\alpha}\in\widehat{\Lambda}\backslash\{\alpha\}$. On the one hand, Proposition \ref{prop:eigen_works} says the sum of $\scp$ over any first-coordinate orbit $\cO_{\tilde{\alpha}}$ vanishes. On the other hand, Lemma \ref{lem:uv-sums} expresses such a sum as a Laurent polynomial in the free variable $v$ (if $\tilde{\alpha}^2\neq\kappa$) or $y$ (if $\tilde{\alpha}^2=\kappa$). Thus every coefficient $c_{\tilde{\alpha},i}(\scp)$ must be $0$. Since $c_{\tilde{\alpha},n}(\scp)$ is a nonzero multiple of $p_n(\tilde{\alpha})$, this forces $p_n(\tilde{\alpha})=0$. This accounts for $|\widehat{\Lambda}\backslash\{\alpha\}|=n-2$ roots. We claim that the two remaining roots of $p_n(x)$ are $2$ and $-2$. Proposition \ref{prop:canonical_degrees} says we may compute $p_n(x)\,\text{mod}\,(x^2-4)$ by replacing each power of $z$ in $y^{2n}+ \alpha y^{2n-2}z^2+\cdots +\zeta^{n}y^nz^n$ (the coefficients from (\ref{eq:eigenvector})) with the same power of $\frac{1}{2}xy$. So \[p_n(x)\equiv 1 +\sum_{j=1}^{n-1}\frac{(\zeta^j+\zeta^{-j})x^j}{2^j}+\frac{\zeta^nx^n}{2^n}\;\text{mod}\,(x^2-4).\] When $x=\pm 2$, the right side above is simply the sum over $-n < j\leq n$ of $(\pm\zeta)^j$, which is $0$ since $\zeta\neq \pm1$. This proves that \begin{equation}\label{eq:p_n(x)}p_n(x) = \frac{\zeta^n}{2}\!\!\!\prod_{\tilde{\alpha}\in\Lambda_n\backslash\{\alpha\}}\!\!\!\!\!(x-\tilde{\alpha}).\end{equation} By Lemma \ref{lem:norm_form}, $p_n(\alpha)=n$. In particular, scaling $\scp$ by $\frac{1}{n}(\frac{\alpha^2-4}{\alpha^2-\kappa})^n$ when $\alpha^2\neq\kappa$ or $\frac{1}{n}$ when $\alpha^2=\kappa$ turns (\ref{eq:lambda_sum}) into (\ref{eq:sum_of_p}) and it turns (\ref{eq:p_n(x)}) into (\ref{eq:lead_coef}).\end{proof}

\begin{notation}For $n\geq 2$ with $\text{char}(R)\centernot\mid n$ and $\alpha\in\widehat{\Lambda}_n$, let $\scp_{\alpha,n}$ denote the polynomial $\scp$ in the statement of Theorem \ref{thm:eigenvector}. Also let $\scp_{\alpha,n}^+=\frac{1}{2}(\scp_{\alpha,n}+p_{-\alpha,n})$.\end{notation}

Note that every other entry in (\ref{eq:eigenvector}) is negated when $\alpha$ is replaced by $-\alpha$. From the structure of the matrices $A_n$ and $B_n$, it is not hard to check that this alternating pattern continues along the coefficients of $\scp_{\alpha,n}$ and $\scp_{-\alpha,n}$. In other words, $\scp_{\alpha,n}^+$ is obtained from $\scp_{\alpha,n}$ by deleting every other monomial, only keeping those with even powers of $y$ and $z$. By Proposition \ref{prop:degree}, we are keeping precisely those monomials that reduce to even polynomials in $x$.

\begin{corollary}For any $n\geq 2$ and $\alpha\in\widehat{\Lambda}_n$, $\scp_{\alpha,n}$ and $\scp_{\alpha,n}^+$ are elements of $\cP_{\!x}(F,\frac{1}{2}\textup{ord}(\alpha))$.\end{corollary}

\begin{proof}This is follows from Proposition \ref{prop:eigen_works} and the definition of $\cP_{\!x}(F,d)$.\end{proof}

\begin{corollary}\label{cor:p_lambda,n_degree}For any $n\geq 2$ and $\alpha\in\widehat{\Lambda}_n$, $\Phi(\scp_{\alpha,n})$ and $\Phi(\scp_{\alpha,n}^+)$ are monic of degree $2n$.\end{corollary}

\begin{proof}We know $\scp_{\alpha,n}$ is a linear combination of elements of $\cB_n\backslash\cB_{0,n}$. By Proposition \ref{prop:degree}, the only element of $\cB_n\backslash\cB_{0,n}$ that has a $\Phi$-reduction of degree at least $2n$ is $y^{2n}$, which appears in $\scp_{\alpha,n}$ with coefficient $1$ (the first entry in (\ref{eq:eigenvector})). Since $\Phi(y^{2n})=x^{2n}$, the claim is proved.\end{proof}

\begin{corollary}\label{cor:eigen_basis}Suppose $\textup{char}(R)=0$. Let $f\in \overline{R}[x,y,z]$, and suppose $f^*$ is even as a polynomial in $y$. Then $f\in\cP_{\!x}(R,\infty)$ if and only if \[\Phi_x(f)=\sum_{i=2}^\infty\sum_{\alpha\in\widehat{\Lambda}_i}\!c_{\alpha,i}(f)\Phi_x(\scp_{\alpha,i}),\] where the $c_{\alpha,i}(f)$ are the unique coefficients for which this equation holds. In this case, $f\in\cP_{\!x}(R,d)$ for some $d\geq 2$ if and only if $c_{\alpha,i}(f)=0$ whenever $2d\centernot\mid \textup{ord}(\alpha)$.\end{corollary}

\begin{proof}First recall that if $2i > \deg_y f^*$  then $c_{\alpha,i}(f)=0$ by (\ref{eq:uv-sums}). This means there are only finitely many pairs $\alpha,i$ with $i\geq 2$ and $\alpha\in\smash{\widehat{\Lambda}_i}$ for which $c_{\alpha,i}(f)$ could possibly be nonzero. Thus if the summation formula holds, $\Phi_x(f)$ is a finite combination of polynomials $\Phi_x(\scp_{\alpha,i})$, and $\sum_{\cO_x}\!\!\Phi_x(\scp_{\alpha,i})(\bt)=\sum_{\cO_x}\!\scp_{\alpha,i}(\bt)$ is only nonzero on a first-coordinate orbit $\cO_x$ if $x=\alpha$ by Proposition \ref{prop:eigen_works}. Thus $\sum_{\cO_x}\!f(\bt)=\sum_{\cO_x}\!\!\Phi_x(f)(\bt)$ can only be nonzero for finitely many values of $x$. That is what it means to belong to $\cP_{\!x}(R,\infty)$.

Conversely, assume $f\in\cP_{\!x}(R,\infty)$. Let $f^*=f_n(x)y^{2n}+\cdots+f_0(x)$, and let $p_{\alpha}(x)$ denote the coefficient of $y^{2n}$ in $\scp_{\alpha,n}^*$ for $\alpha\in\widehat{\Lambda}_n$. From the formula in (\ref{eq:lead_coef}), the only common roots among the $p_{\alpha}$ as $\alpha$ ranges over $\widehat{\Lambda}_n$ are 2 and $-2$, meaning these polynomials generate the ideal $(x^2-4)$ in $\overline{F}[x]$. But by Proposition \ref{prop:P_equiv}, specifically (1) implies (5), $x^2-4$ must divide $f_n(x)$. Thus there exist $g_{\alpha,n}(x)\in \overline{F}[x]$ such that $f^*-\sum_{\widehat{\Lambda}_n}g_{\alpha,n}\scp_{\alpha,n}^*$ has degree $2n-2$ in the variable $y$. And since $\sum_{\widehat{\Lambda}_n}g_{\alpha,n}\scp_{\alpha,n}^*$ belongs to $\cP_{\!x}(F,\infty)$, we can repeat this process for $n-1$, $n-2$,..., $2$. (We must stop after $2$ because there is no such thing as ``$\scp_{\alpha,1}$".) Thus for the right choice of polynomials $g_{\alpha,i}$, we see that \[f^*-
\sum_{i=2}^n\sum_{\alpha\in\widehat{\Lambda}_i}g_{\alpha,i}\scp_{\alpha,i}^*\] has degree at most $2$ in $y$. Since the sum above still lies in $\cP_{\!x}(F,\infty)$, Proposition \ref{prop:P_equiv}, specifically (1) implies (5), says it must take the form $\tilde{f}(x)(x^2-4)y^2-2\tilde{f}(x)(x^2-\kappa)$ for some $\tilde{f}\in\overline{F}[x]$. It is quick to check that the image of such a polynomial under $\Phi_x$ is the zero polynomial. Thus \begin{flalign*}&& 0&=\Phi_x(f^*)-
\sum_{i=2}^n\sum_{\alpha\in\widehat{\Lambda}_i}\Phi_x(g_{\alpha,i}\scp_{\alpha,i}^*) && \text{by linearity of }\Phi\\&& &=\Phi_x(f^*)-\sum_{i=2}^n\sum_{\alpha\in\widehat{\Lambda}_i}\Phi_x((g_{\alpha,i}\scp_{\alpha,i})^*) && \text{since }g_{\alpha,i}\in\overline{F}[x]\\&& &=\Phi_x(f)-\sum_{i=2}^n\sum_{\alpha\in\widehat{\Lambda}_i}\Phi_x(g_{\alpha,i}\scp_{\alpha,i}) && \text{by Proposition \ref{prop:Phi_x(f^*)}}\\&& &=\Phi_x(f)-\sum_{i=2}^n\sum_{\alpha\in\widehat{\Lambda}_i}g_{\alpha,i}(\alpha)\Phi_x(\scp_{\alpha,i}) &&\text{by Proposition \ref{prop:pull_out}.}\end{flalign*} 

Now, as a consequence of Proposition \ref{prop:eigen_works} and (\ref{eq:sum_of_p}), we know that $\sum_{\cO_{\tilde{\alpha}}}\!\scp_{\alpha,i}(\bt) = 0$ for all first-coordinate orbits $\cO_{\tilde{\alpha}}$ if and only if $\tilde{\alpha}\neq\alpha$. Combining this with the equation above gives the following for any fixed $\alpha\neq \pm 2$ of finite order: \[\frac{1}{|\cO_\alpha|}\sum_{\bt\in\cO_\alpha}\Phi_x(f)(\bt)=\frac{1}{|\cO_\alpha|}\sum_{i=2}^ng_{\alpha,i}(\alpha)\Phi_x(\scp_{\alpha,i})(\bt).\] Assuming $\alpha^2\neq \kappa$, Lemma \ref{lem:uv-sums} expresses the left-hand side as \[\sum_{i=2}^nc_{\alpha,i}(f)(v^{2i}+v^{-2i}),\] where $c_{\alpha,0}(f)=0$ has been omitted since $f\in\cP_x(R,\infty)$, as has $c_{\alpha,1}(f)=0$ since $\alpha\not\in\Lambda_1 = \{\pm2\}$. But then the right-hand side can be evaluated using (\ref{eq:sum_of_p}). It equals \[\sum_{i=2}^ng_{\alpha,i}(\alpha)(v^{2i}+v^{-2i}).\] The last two expressions must be equal for any assignment of the free variable $v$ from the infinite domain $R$. This forces $c_{\alpha,i}(f) = g_{\alpha,i}(\alpha)$, thereby proving our formula.

The corollary's final claim regarding membership in $\cP(R,d)$ also follows  from the fact that $\sum_{\cO_{\tilde{\alpha}}}\!\scp_{\alpha,i}(\bt) = 0$ for all first-coordinate orbits $\cO_{\tilde{\alpha}}$ if and only if $\tilde{\alpha}\neq\alpha$. \end{proof}

When $f^*$ is also even as a polynomial in $x$ (so $f^*\in R[x^2,y^2]$) we have $c_{\alpha,i}(f)=c_{-\alpha,i}(f)$. Hence this last corollary allows us to express $\Phi_x(f)$ as a linear combination of $\Phi_x(\scp_{\alpha,i}^+)$.

\subsection{Another partial formula} Interestingly, we will compute more coefficients of $\Phi(\scp_{\alpha,n}^+)$ than we will of $\scp_{\alpha,n}^+$.





\begin{lemma}\label{lem:binom_int}For any positive integers $j$ and $n$, $\frac{2n}{n+j}\binom{n+j}{n-j}$ is an integer.\end{lemma}

\begin{proof}Let $p$ be a prime, and let $v_p(j)$ and $v_p(n)$ denote the valuations of $j$ and $n$ at $p$. Since $\smash{\frac{2n}{n+j}\binom{n+j}{n-j} = \frac{n}{j}\binom{n+j-1}{2j-1}}$, the task is to show that $v_p(\smash{\binom{n+j-1}{2j-1}})\geq v_p(j)-v_p(n)$. The inequality is immediate if $v_p(j)\leq v_p(n)$, so assume otherwise. Kummer's theorem states that $v_p(\smash{\binom{n+j-1}{2j-1}})$ is the number of carries when $n-j$ and $2j-1$ are added in base $p$. The coefficients of $p^0,\dots ,p^{v_p(j)-1}$ when $2j-1$ is written in base $p$ are all $p-1$, and the coefficient of $p^{v_p(n)}$ when $n-j$ is written in base $p$ is nonzero. Thus a carry occurs at the places $p^{v_p(n)},\dots,p^{v_p(j)-1}$, for a total of $v_p(j)-v_p(n)$ carries, as needed.\end{proof}

\begin{lemma}\label{lem:overZ}Suppose $\textup{char}(R)\centernot\mid n$. Let $m$ and $n$ be integers with $n > m \geq 0$. Define \[f_j(x^2)=\begin{dcases}0 & j < m \\[0.1cm] \frac{2n-2m}{n}\binom{n+m}{n-m}(\kappa-x^2)^{n-m} & j = m\\[0.1cm] \frac{2n}{n+j}\binom{n+j}{n-j}(x^2-4)^{j-m}(\kappa-x^2)^{n-j} & j > m,\end{dcases}\] and let $f=\sum_j f_jy^{2j}$. Then $f\in \cP_{\!x}(F,\infty)$ with $c_{\alpha,i}(f)=0$ if $m < i < n$.\end{lemma}

\begin{proof} To show that $f\in \cP_{\!x}(R,\infty)$ we use the equivalent property (5) from Proposition \ref{prop:P_equiv}, which is \begin{equation}\label{eq:f_i_sum}\sum_{j=0}^n\! \binom{2j}{j}\!\left(\frac{x^2-\kappa}{x^2-4}\right)^{\!\!j}f_j(x^2)=0.\end{equation} By substituting in the definition of $f_j(x^2)$ and solving for $f_m(x^2)$, the equation above reduces to the following identity after some straightforward arithmetic: \[\sum_{j=m+1}^n \!(-1)^j\binom{n}{j}\binom{n+j-1}{j}=(-1)^{m+1}\binom{n-1}{m}\binom{n+m}{m}.\] This is easily checked by backward induction on $m$, the base case being $m=n-1$.

Now suppose $m < i < n$ and consider some $\alpha\in\widehat{\Lambda}_i$. By (\ref{eq:uv-sums}), \begin{align*}c_{\alpha,i}(f)&=\sum_{j=i}^n\binom{2j}{j-i}\!\left(\frac{\alpha^2-\kappa}{\alpha^2-4}\right)^{\!j}\!f_j(\alpha^2)\\&=\frac{(\kappa-\alpha^2)^n}{(\alpha^2-4)^m}\sum_{j=i}^n\frac{2n(-1)^j}{n+j}\binom{2j}{j-i}\binom{n+j}{n-j}\\ &=\frac{2n(-1)^i(\kappa-\alpha^2)^n}{(n+i)(\alpha^2-4)^m}\binom{n+i}{n-i}{}_2F_1(i-n,n+i;2i+1;1),\end{align*} which vanishes by Gauss' formula for hypergeometric functions.\end{proof}

\begin{theorem}\label{thm:eigen_form}Suppose $\textup{char}(R)\centernot\mid n$. For any $n\geq 2$ and $\alpha\in\widehat{\Lambda}_n\backslash\{\pm\sqrt{\kappa}\}$, \[\Phi(\scp_{\alpha,n}^+)=\frac{1}{n}\sum_{j=0}^n\binom{2n-j-1}{j}\!\left(\frac{\alpha^2-4}{\alpha^2-\kappa}\right)^{\!\!n-j}\!(-1)^jb_{n-j}(x^2)+r(x^2)\] for some $r(x^2)\in \overline{F}[x]$ of degree at most $2\lfloor\frac{3}{4}n\rfloor$.\end{theorem}

\begin{proof}Define $f_j(x^2)$ as in Lemma \ref{lem:overZ} with $m=\lfloor\frac{3}{4}n\rfloor$, and set \[f(x^2,y^2)=\sum_{j=0}^n f_j(x^2)y^{2j}.\] For any $\ell\geq 0$ we have $c_{\alpha,i}(x^{2\ell} f)=\alpha^{2\ell} c_{\alpha,i}(f)$, which is 0 when $i > n$ or when $m < i < n$ by Lemma \ref{lem:overZ}. If we omit these vanishing coefficients from the expression in Corollary \ref{cor:eigen_basis}, the result is \[\Phi_x(x^{2\ell} f)=\sum_{\alpha\in\widehat{\Lambda}_n}\alpha^{2\ell} c_{\alpha,n}(f)\Phi_x(\scp_{\alpha,n}) + \sum_{i=1}^m\sum_{\alpha\in\widehat{\Lambda}_i}\alpha^{2\ell} c_{\alpha,i}(f)\Phi_x(p_{\alpha,i}).\] By Corollary \ref{cor:p_lambda,n_degree}, if $i\leq m$ then $\deg\Phi(p_{\alpha,i})=2i\leq 2m$. Thus the entire right-side sum above can be absorbed into the remainder polynomial $r(x^2)$ and therefore ignored. Regarding the left-side sum, let us group $\scp_{\alpha,n}$ and $\scp_{-\alpha,n}$ into $2\scp_{\alpha,n}^+$ (which appears as two occurrences of $\scp_{\alpha,n}^+$ in the sum below) using the fact that $c_{\alpha,n}(f)=c_{-\alpha,n}(f)$ because $f$ is even as a polynomial in $x$. The result is \begin{equation}\label{eq:top_coefs}\Phi(x^{2\ell} f)=\sum_{\alpha\in\widehat{\Lambda}_n}\alpha^{2\ell} c_{\alpha,n}(f)\Phi(\scp_{\alpha,n}^+)+r(x^2)\end{equation} for some $r(x^2)\in R[x]$ of degree at most $2m=2\lfloor\frac{3}{4}n\rfloor$. 

Now assume $\ell <\lfloor\frac{n}{2}\rfloor$. Recall from the definition of $f_j(x^2)$ in Lemma \ref{lem:overZ} that $(x^2-4)^{j-m}\,|\,f_j(x^2)$. So for $j > m$, we deduce from Theorem \ref{thm:spec_form} that \[\Phi(x^{2\ell} f_j(x^2)y^{2j})=\frac{1}{n}\sum_{i=0}^j \binom{2i}{i}\sum_{\alpha\in\widehat{\Lambda}_n}\!\left(\frac{\alpha^2-\kappa}{\alpha^2-4}\right)^{\!\!i}\!\alpha^{2\ell}f_j(\alpha^2) b_{j-i}(x^2)+r_j(x^2)\] for some $r_j(x^2)\in R[x]$ of degree at most $\max(\deg x^{2\ell}f_j(x^2),2j-2(j-m))=\max(2(\ell+n-m),2m)$, which is bounded above by $2m$ because $\ell< \lfloor\frac{n}{2}\rfloor$. When $j \leq m$, we do not care about $\Phi(x^{2\ell}f_j(x^2)y^{2j})$ because it has degree at most $\max(\deg x^{2\ell}f_j(x^2),2j)\leq 2m$ by Proposition \ref{prop:degree}. Summing the expression above over $j$ and combining all $r_j(x^2)$ polynomials into a single remainder gives \begin{align*}\Phi(x^{2\ell} f)=\frac{1}{n}\!\sum_{j=m+1}^n\sum_{i=0}^j \binom{2i}{i}\sum_{\alpha\in\widehat{\Lambda}_n}\!\left(\frac{\alpha^2-\kappa}{\alpha^2-4}\right)^{\!\!i}\!\alpha^{2\ell}f_j(\alpha^2) b_{j-i}(x^2)+\tilde{r}(x^2) &&\end{align*} Now substitute the formula in Lemma \ref{lem:overZ} for $f_j(\alpha^2)$. To get the second line below, change variables by replacing $i$ and $j$ with $j-i$ and $n-i$, respectively: \begin{align*}\Phi&(x^{2\ell} f)\\&=\frac{1}{n}\sum_{j=m+1}^n\sum_{i=0}^j \binom{2i}{i}\sum_{\alpha\in\widehat{\Lambda}_n}\frac{2n(-1)^{n-j}\alpha^{2\ell}(\alpha^2-\kappa)^{n-j+i}}{(n+j)(\alpha^2-4)^{m-j+i}}\binom{n+j}{n-j} b_{j-i}(x^2) +\tilde{r}(x^2)\\&=\frac{1}{n}\sum_{j=0}^{n-m}\sum_{\alpha\in\widehat{\Lambda}_n}\frac{(\alpha^2-\kappa)^{j}\alpha^{2\ell}}{(\alpha^2-4)^{j+m-n}}\sum_{i=0}^j \frac{2n(-1)^i}{2n-i}\binom{2j-2i}{j-i}\binom{2n-i}{i} b_{n-j}(x^2) +\tilde{r}(x^2)\end{align*}

The sum over $i$ can be evaluated using Saalsch{\"u}tz's theorem for hypergeometric functions: \begin{align*}\sum_{i=0}^j \frac{2n(-1)^i}{2n-i}&\binom{2j-2i}{j-i}\binom{2n-i}{i}\\&=\frac{2n(-1)^j}{2n-j}\binom{2n-j}{j}{}_3F_2(-j, 2n-j, \tfrac{1}{2};n - j + \tfrac{1}{2}, n-j+1;1)\\&=(-1)^j\binom{2n-j-1}{j}.\end{align*} Thus we obtain the formula

\[\Phi(x^{2\ell}f)=\frac{1}{n}\sum_{j=0}^{n-m}\sum_{\alpha\in\widehat{\Lambda}_n}\frac{(\alpha^2-\kappa)^j\alpha^{2\ell}}{(\alpha^2-4)^{j+m-n}}\binom{2n-j-1}{j}(-1)^j b_{n-j}(x^2) +\tilde{r}(x^2).\]

From this we are able to determine the coefficient of some power of $x$, say $x^{2i}$, in $c_{\alpha,n}(f)\Phi(\scp_{\alpha,n}^+)$ for each $\alpha\in\widehat{\Lambda}_n$. Indeed, combining the equation above with (\ref{eq:top_coefs}) determines the sum of these coefficients over $\alpha\in\widehat{\Lambda}_n$, and we have one such equation for each nonnegative $\ell<\lfloor\frac{n}{2}\rfloor$. There are only $\lfloor\frac{n}{2}\rfloor$ values of $\alpha\in\widehat{\Lambda}_n$ up to a change of sign, and the $\lfloor\frac{n}{2}\rfloor\times \lfloor\frac{n}{2}\rfloor$ Vandermonde matrix with each column consisting of powers of some $\alpha^2$ is invertible. Thus the system of equations we have produced uniquely determines the coefficient of $x^{2i}$ in each $c_{\alpha,n}(f)\Phi(\scp_{\alpha,n}^+)$ provided $i > m$. One possibly solution is evident: \[c_{\alpha,n}(f)\Phi(\scp_{\alpha,n}^+)=\frac{1}{n}\sum_{j=0}^n\frac{(\alpha^2-\kappa)^j}{(\alpha^2-4)^{j+m-n}}\binom{2n-j-1}{j}(-1)^j b_{n-j}(x^2) +r_\alpha(x^2),\] so this must be \textit{the} solution. By (\ref{eq:uv-sums}), $\smash{c_{\alpha,n}(f) = (\frac{\alpha^2-\kappa}{\alpha^2-4})^nf_n(\alpha^2) = \frac{(\alpha^2-\kappa)^n}{(\alpha^2-4)^m}}$. This completes the proof.\end{proof}

\subsection{The space $\cP_{\!x}(R,d)$} Recall that $\cP_{\!x}(R,d)$ consists of those $f\in\overline{R}[x,y,z]$ such that $\sum_{\cO_\alpha}\!f(\bt)=0$ whenever $2d\centernot\mid \text{ord}(\alpha)$. Proposition \ref{prop:eigen_works} \textit{almost} implies that $\scp_{\alpha,n}$ (and thus $\scp_{\alpha,n}^+$) belongs to $\cP_{\!x}(R,\text{ord}(\alpha))$---the problem is that coefficients of $\scp_{\alpha,n}$ are in $\overline{F}$ and generally not in $\overline{R}$ when $n\geq 4$. However, the polynomial ``$f$" from Lemma \ref{lem:overZ} with ``$m$" set to 0 turns out to be a linear combination of the $\scp_{\alpha,n}^+$ (with $n$ fixed), and $f$ has coefficients in $R$ by Lemma \ref{lem:binom_int}. The exact linear combination is straightforward to compute, so Theorem \ref{thm:eigen_form} provides a formula for the top coefficients of $\Phi(f)$. Let us first recall the definition of ``$f$":

\begin{notation}\label{not:f_n}For $n\geq 0$, let \[f_n(x^2,y^2)=\sum_{j=0}^n\frac{2n}{n+j}\binom{n+j}{n-j}(x^2-4)^j(\kappa-x^2)^{n-j}y^{2j}.\]\end{notation}

\begin{corollary}\label{cor:Phi(gf_n)}For any integers $d,n\geq 2$ and any $g(x^2)\in \overline{R}[x]$, $gf_n\in\cP_{\!x}(R,d)$. Furthermore, \[\Phi(gf_n)=\frac{1}{n}\sum_{j=0}^n\sum_{\alpha\in\widehat{\Lambda}_n}\!\!\binom{2n-j-1}{j}g(\alpha^2)(\alpha^2-4)^{n-j}(\kappa-\alpha^2)^jb_{n-j}(x^2)+r(x^2)\] for some $r(x^2)\in \overline{R}[x]$ of degree at most $2\lfloor\frac{3}{4}n\rfloor$.\end{corollary}

\begin{proof}Since $g$ is a function of $x$ only, $c_{\alpha,i}(gf_n)=g(\alpha^2)c_{\alpha,i}(f_n)$ by Proposition \ref{prop:pull_out}, and this equals 0 if $i\neq n$ by Lemma \ref{lem:overZ}. So Corollary \ref{cor:eigen_basis} gives \begin{equation}\label{eq:Phi(gf_n)}\Phi(gf_n)=\sum_{\alpha\in\widehat{\Lambda}_n}c_{\alpha,n}(gf_n)\Phi(\scp_{\alpha,n}^+),\end{equation} where we have grouped $\scp_{\alpha,n}$ and $\scp_{-\alpha,n}$ into $2\scp_{\alpha,n}^+$ (which appears as two occurrences of $\scp_{\alpha,n}^+$ in the sum) using the fact that $c_{\alpha,n}(f)=c_{-\alpha,n}(f)$ because $gf_n$ is even as a polynomial in $x$. Theorem \ref{thm:eigen_form} provides a formula for $\Phi(\scp_{\alpha,n}^+)$. Furthermore, the coefficient of $y^{2n}$ in $(gf_n)^*=gf_n$ is $g(x^2)(x^2-4)^n$, from which Lemma \ref{lem:uv-sums} provides the formula $c_{\alpha,n}=\frac{1}{n}g(\alpha^2)(\alpha^2-4)^n$. Substituting these into the expression above completes the proof.\end{proof}

\begin{corollary}\label{cor:found_em}Let $d$ and $n$ be positive integers with $d$ a prime power, and let $\tilde{n}=d\lceil\frac{n}{d}\rceil$. If $n> 3d$ or if $d\,|\,n$ then $\Phi(\cP_{\!x}(R,d))$ contains a polynomial of degree at most $2n$ in which the coefficient of $x^{2n}$ is $2\binom{\tilde{n}+n-1}{\tilde{n}-n}(4-\kappa)^{\tilde{n}}\tilde{n}^{\tilde{n}}$.\end{corollary}

\begin{proof}Assume that $\text{char}(R)\centernot\mid n$ since otherwise the claim is immediate. We begin with the simplest case where $d\,|\,n$. Fix any $\alpha$ of rotation order $2n$ and set \[g(x^2)\coloneqq x^\delta\prod_{\mathclap{\tilde{\alpha}\in\widehat{\Lambda}_n\backslash\{\pm\alpha\}}}\,(x-\tilde{\alpha}),\] where $\delta = 1$ if $0\in\widehat{\Lambda}_n\backslash\{\pm\alpha_i\}_i$ and $\delta=0$ otherwise (to make $g$ an even polynomial). 

Observe that $c_{\tilde{\alpha},i}(gf_n)=g(\tilde{\alpha}^2)c_{\tilde{\alpha},i}(f_n)$ by Proposition \ref{prop:pull_out}, which vanishes unless $i=n$ (by Lemma \ref{lem:overZ}) and $\tilde{\alpha}=\pm\alpha$ (by definition of $g$). Since $2d\,|\,\text{ord}(\alpha)$ by assumption, we have $gf_n\in\cP_{\!x}(R,d)$ by Corollary \ref{cor:eigen_basis}.

Next we apply Corollary \ref{cor:Phi(gf_n)} to conclude that $\Phi(gf_n)$ has degree $2n$ and leading coefficient \[\frac{1}{n}\!\sum_{\tilde{\alpha}\in\widehat{\Lambda}_n}\!g(\tilde{\alpha}^2)(\tilde{\alpha}^2-4)^n = \frac{2}{n}g(\alpha^2)(\alpha^2-4)^n.\] In $\overline{R}$, $g(\alpha^2)(\alpha^2-4)^n$ divides \[\prod_{\mathclap{\tilde{\alpha}\in\Lambda_n\backslash\{\alpha\}}}\;(\alpha-\tilde{\alpha})^n,\] which equals $n^n$ up to a sign by Lemma \ref{lem:norm_form}. This completes the proof when $d\,|\,n$.

Now suppose $n > 3d$ and let $\tilde{n}=d\lceil\frac{n}{d}\rceil$. We proceed in similar fashion, but with a modified polynomial $g$. Let $p$ be the prime dividing $d$. If $\zeta$ is a primitive $2\tilde{n}^\text{th}$ root of unity and $\alpha=\zeta^i+\zeta^{-i}$ for some $i$ not divisible by $p$, then $2d\,|\, \text{ord}(\alpha)$. The number of positive integers $i\leq\frac{\tilde{n}}{2}$ not divisible by $p$ is $\ell+1\coloneqq\lceil\frac{\tilde{n}}{2p}\rceil(p-1)$. Let $\alpha_0,...,\alpha_\ell$ denote the corresponding $\alpha$-values in any order, and define \begin{equation}\label{eq:g_def}g(x^2)\coloneqq x^\delta\prod_{\mathclap{\alpha\in\widehat{\Lambda}_{\tilde{n}}\backslash\{\pm\alpha_i\}_i}}\,(x-\alpha),\end{equation} where $\delta = 1$ if $0\in\widehat{\Lambda}_{\tilde{n}}\backslash\{\pm\alpha_i\}_i$ and $\delta=0$ otherwise. 

Observe that $c_{\alpha,i}(gf_{\tilde{n}})=g(\alpha^2)c_{\alpha,i}(f_{\tilde{n}})$ by Proposition \ref{prop:pull_out}, which vanishes unless perhaps $i=\tilde{n}$ (by Lemma \ref{lem:overZ}) and $2d\,|\,\text{ord}(\alpha)$ (by definition of $g$). Hence $x^{2i}gf_{\tilde{n}}\in\cP_{\!x}(R,d)$ for any $i\geq 0$ by Corollary \ref{cor:eigen_basis}.

As before, we apply Corollary \ref{thm:eigen_form} to conclude that if $\tilde{n}-j > \lfloor\frac{3}{4}\tilde{n}\rfloor$ then the coefficient of $b_{\tilde{n}-j}(x^2)$ in some $\overline{F}$-linear combination $c_0\Phi(x^0gf_{\tilde{n}})+\cdots +c_{\ell}\Phi(x^{2\ell}gf_{\tilde{n}})$ is the $j^\text{th}$ entry (starting at $j=0$) of the product \begin{equation}\label{eq:mat_prod}D_1V_1D_2V_2\begin{bmatrix}c_0\\ \vdots\\ c_\ell\end{bmatrix}\end{equation} where \[D_1=\begin{bmatrix}(-1)^0\binom{2\tilde{n}-0-1}{0} & & \\&  \ddots & \\ & & (-1)^{\ell}\binom{2\tilde{n}-\ell-1}{\ell}\end{bmatrix}\] \vspace{0.1cm} \[V_1=\begin{bmatrix}\left(\frac{\alpha_0^2-\kappa}{\alpha_0^2-4}\right)^{\!0} & \cdots & \left(\frac{\alpha_\ell^2-\kappa}{\alpha_\ell^2-4}\right)^{\!0} \\ \vdots & & \vdots\\ \left(\frac{\alpha_0^2-\kappa}{\alpha_0^2-4}\right)^{\!\ell} & \cdots & \left(\frac{\alpha_\ell^2-\kappa}{\alpha_\ell^2-4}\right)^{\!\ell}\end{bmatrix}\] \vspace{0.1cm} \[ D_2 = \begin{bmatrix}(\alpha_0^2-4)^{\tilde{n}}g(\alpha_0^2) & & \\&  \ddots & \\ & & (\alpha_\ell^2-4)^{\tilde{n}}g(\alpha_\ell^2)\end{bmatrix}\] \vspace{0.1cm} \[ \text{and }V_2=\begin{bmatrix}\alpha_0^0 & \cdots & \alpha_0^{2\ell} \\ \vdots & & \vdots\\ \alpha_\ell^0 & \cdots & \alpha_\ell^{2\ell}\end{bmatrix}.\] The product $V_1D_2V_2$ has entries in $\overline{R}$ since the denominators in $V_1$ are canceled by entries in $D_2$. Thus for any $j=0,...,\ell$, there exist $c_0,..,c_\ell\in R$ such that the $i^\text{th}$ entry in the product (\ref{eq:mat_prod}) is $0$ when $i < j$ and $\smash{\binom{2\tilde{n}-j-1}{j}}\det(V_1D_2V_2)$ when $i=j$. Remember that these entries are the coefficients of $b_{\tilde{n}}(x^2),\dots ,b_{\tilde{n}-j}(x^2)$ provided $\tilde{n}-j > \lfloor\frac{3}{4}\tilde{n}\rfloor$. We wish to use this with $j = \tilde{n} - n$ for $n$ from the theorem statement, so we have to check that the hypothesis $n > 3d$ is enough to guarantee $\tilde{n}-j > \lfloor\frac{3}{4}\tilde{n}\rfloor$ and $j\leq \ell$. The claim will then follow provided $\det(V_1D_2V_2)$ divides $2(4-\kappa)^{\tilde{n}}\tilde{n}^{\tilde{n}}$. 

If $n> 3d$ then $\tilde{n} - j=n > \frac{3}{4}(n+d) >\frac{3}{4}d\lceil\frac{n}{d}\rceil \geq \lfloor\frac{3}{4}\tilde{n}\rfloor$ as desired. Next, using $\ell\geq\smash{\frac{(p-1)\tilde{n}}{2p}}-1$ followed by $\tilde{n}\leq n+d-1$, the desired inequality $j = \tilde{n} - n\leq \ell$ can be rewritten as $d\leq \smash{\frac{p-1}{2p}}(n+d-1)$.  This holds using $n> 3d$ and $p\geq 2$.

Finally, the determinant of a diagonal or a Vandermonde matrix has a standard product form. The result is \begin{equation}\label{eq:det}\det(V_1D_2V_2)=(4-\kappa)^\ell\Bigg(\prod_{i=0}^\ell\,(\alpha_i^2-4)g(\alpha_i^2)\Bigg)\Bigg(\prod_{0\leq i < j \leq \ell}\!\!\!\!(\alpha_i^2-\alpha_j^2)\Bigg).\end{equation} Observe that this divides \[(4-\kappa)^{\ell}\prod_{i=0}^\ell\prod_{\alpha\in\Lambda_{\tilde{n}}\backslash\{\alpha_i\}}(\alpha_i-\alpha)\] simply by accounting for every factor in (\ref{eq:det}). By Lemme \ref{lem:norm_form}, this last expression equals $(4-\kappa)^\ell\tilde{n}^{\ell+1}$ up to a sign, which divides $2(4-\kappa)^{\tilde{n}}\tilde{n}^{\tilde{n}}$.\end{proof}

The leading coefficient $2\smash{\binom{n+m-1}{n-m}}(4-\kappa)^nn^n$ in the previous corollary is important because we plan to quotient by prime ideals in $\oZ$ to extract information about $\cP(\bF_q)$ from $\cP_{\!x}(\bZ,d)$, and we need to know that the degrees of our polynomials are maintained by the quotient.

\begin{proposition}\label{prop:reduction}Let $\kappa\in\bF_q$. Fix a surjection $\pi:\oZ\to\oF_q$ and some $\tilde{\kappa}\in\pi^{-1}(\kappa)$. If $2d\in\bZ\backslash\{0\}$ does not divide $q\pm 1$, then $\Phi(\pi(\cP_{\!x}(\overline{\bZ},d)))$ (where $\cP_{\!x}(\overline{\bZ},d)$ is determined using $\tilde{\kappa}$) is a subspace of $\cP(\bF_q)$ that is independent of the choices of $\pi$ and $\tilde{\kappa}$.\end{proposition}

\begin{proof}Consider some $\tilde{f}\in\cP_{\!x}(\overline{\bZ},d)$, and let $f\in\overline{\bF}_q[x,y,z]$ be its image under $\pi$. Our goal is to show $\Phi(f)\in \cP(\bF_q)$.

Each $\alpha\in\bF_q$ lifts to some $\pm(\zeta+\zeta^{-1})\in \overline{\bZ}$ with $\zeta$ a primitive $\text{ord}(\alpha)^{\text{th}}$ root of unity. Call this lift $\tilde{\alpha}$. Recall that $2d$ does not divide $\text{ord}(\alpha)=\text{ord}(\tilde{\alpha})$ by Corollary~\ref{cor:total_count} because $2d$ does not divide $q\pm 1$. Now lift each $\cO_{\alpha}\subset\cM(\bF_q)$ to some first-coordinate orbit $\cO_{\tilde{\alpha}}\subset\cM(\overline{\bZ})$ of the same size; it does not matter which one of the infinitely many lifts we choose.

Let $\cO$ be any $\Gamma$-invariant subset of $\cM(\bF_q)$, and let $\tilde{\cO}\subset\cM(\overline{\bZ})$ be the union of those $\cO_{\tilde{\alpha}}$ for which $\cO_{\alpha}\subseteq\cO$. The first coordinate of every triple in $\tilde{\cO}$ has order not divisible by $2d$, so $\sum_{\tilde{\cO}}\tilde{f}(\bt)=0$. But then \[0= \pi\Bigg(\sum_{\bt\in\tilde{\cO}}\tilde{f}(\bt)\Bigg)=\!\sum_{\bt\in\pi(\tilde{\cO})}\!\!\pi(\tilde{f})(\bt)=\sum_{\bt\in\cO}f(\bt)=\sum_{\bt\in\cO}\Phi(f)(\bt).\] As $\cO$ was arbitrary, this proves $\Phi(f)\in \cP(\bF_q)$.

We turn to the final claim regarding independence of the choices of $\pi$ and $\tilde{\kappa}$. Let $\pi_1,\pi_2:\oZ\to\oF_q$ be two surjections, and let $\tilde{\kappa}_1\in\pi_1^{-1}(\kappa)$ and $\tilde{\kappa}_2\in\pi^{-1}_2(\kappa)$. Let $\varphi\in\textup{Gal}(\oQ/\bQ)$ satisfy $\pi_1=\pi_2\circ\varphi$ (see pages 394--395 in \cite{washington}, for example). Consider some polynomial $\tilde{f}_1$ that belongs to $\cP_x(\oZ,d)$ as determined by $\tilde{\kappa}_1$. Assume that $f^*$ is even as a polynomial in $y$. This loses no generality because if a monomial with odd degree in $y$ appears in $f^*$, the $\Phi$ reduction of this monomial would be an odd polynomial by Proposition \ref{prop:degree}, and we already know that both $\Phi(\pi_1(\cP_{\!x}(\overline{\bZ},d)))$ and $\Phi(\pi_2(\cP_{\!x}(\overline{\bZ},d)))$ contain all odd polynomials (since $\Phi(\pi_i(x^{2n+1}))=x^{2n+1}$). Now, Corollary \ref{cor:eigen_basis} says that $\Phi(\tilde{f}_1) = \sum_{\alpha}c_{\alpha,i}(\tilde{f}_1)\Phi(\scp_{\alpha,i})$, where $c_{\alpha,i}(\tilde{f}_1)$ can be nonzero only if $2d\,|\,\text{ord}(\alpha)$. So we consider the polynomial $\tilde{f}_2 \coloneqq \sum_{\alpha}\varphi(c_{\alpha,i}(\tilde{f}_1)\scp_{\alpha,i})$. Since $\alpha$ and $\varphi(\alpha)$ have the same rotation order, $\tilde{f}_2$ belongs to $\cP_{\!x}(\overline{\bZ},d)$ as determined by $\varphi(\tilde{\kappa}_1)$. Furthermore, by repeatedly using the ring homomorphism property and linearity of $\Phi$, \begin{align*}\Phi(\pi_1(\tilde{f}_1)) = \pi_1(\Phi(\tilde{f}_1)) &=(\pi_2\circ\varphi)(\Phi(\tilde{f}_1))\\&=(\pi_2\circ\varphi)\Bigg(\sum_{\alpha,i}c_{\alpha,i}(\tilde{f}_1)\Phi(\scp_{\alpha,i})\Bigg)\\&=\pi_2\Bigg(\Phi\Bigg(\sum_{\alpha,i}\varphi(c_{\alpha,i}(\tilde{f}_1)\scp_{\alpha,i})\Bigg)\Bigg)\\&=\pi_2(\Phi(\tilde{f}_2))=\Phi(\pi_2(\tilde{f}_2)).\end{align*} Finally, recall that each $\varphi(\scp_{\alpha,i})$ that we have used to define $\tilde{f}_2$ is a linear combination of polynomials from $\cB_i$. The coefficients in this linear combination are determined by eigenvectors of the matrix $M_i$, the definition of which does not depend on the value of $\kappa$. Thus by replacing $\varphi(\tilde{\kappa}_1)$ with $\tilde{\kappa}_2$ in every element of $\cB_i$, each $\varphi(\scp_{\alpha,i})$ with $2d\,|\,\text{ord}(\alpha)$ becomes a polynomial in $\cP_{\!x}(\oZ,d)$ as determined by $\tilde{\kappa}_2$ rather than $\varphi(\tilde{\kappa}_1)$. This replacement does not change the image of $\varphi(\scp_{\alpha,i})$ (and thus of $\tilde{f}_2$) under $\pi_2$ because $\pi_2(\tilde{\kappa}_2) = \kappa = \pi_1(\tilde{\kappa}_1) = \pi_2(\varphi(\tilde{\kappa}_1))$. Altogether, we have shown that $\Phi(\pi_1(\tilde{f}_1))$, which is an arbitrary element of $\Phi(\pi_1(\cP_{\!x}(\overline{\bZ},d)))$ as determined by $\tilde{\kappa}_1$, belongs to $\Phi(\pi_2(\cP_{\!x}(\overline{\bZ},d)))$ as determined by $\tilde{\kappa}_2$.\end{proof}

\begin{notation}\label{not:P(F_q,d)}If $q\not\equiv\pm 1\,\textup{mod}\,2d$ let $\cP(\bF_q,d)$ denote the image of $\Phi(\cP_{\!x}(\oZ,d))$ in $\cP(\bF_q)$ from Proposition \ref{prop:reduction}, and if $q\equiv\pm 1\,\textup{mod}\,2d$ let $\cP(\bF_q,d)=\{0\}$. Let $\cP(\bF_q,\infty)$ denote the vector space sum $\sum_d\cP(\bF_q,d)$ for $d\geq 2$.\end{notation}

By construction, $\cP(\bF_q,d)\subseteq\cP(\bF_q,\infty)\subseteq\cP(\bF_q)$ for any $d$. In particular $\cP^{\perp}\!(\bF_q)\subseteq \cP^{\perp}\!(\bF_q,\infty)\subseteq\cP^{\perp}\!(\bF_q,d)$, where the orthogonal complement is taken in $\bF_q^q$ as in Notation \ref{not:P^perp(F_q)}. This is important in view of Theorem~\ref{thm:P^perp}, which rephrases the $Q$-Classification Conjecture and Theorem \ref{thm:main1} in terms of spanning sets for $\cP^{\perp}\!(\bF_q)$.

\begin{notation}For $n\geq 0$, let $\cP_{\!n}(\bF_q,d)$ and $\cP_{\!n}(\bF_q,\infty)$ denote the set of polynomials in $\cP(\bF_q,d)$ and $\cP(\bF_q,\infty)$, respectively, of degree at most $n$.\end{notation}


\section{Generalized eigenvectors for $\alpha=\pm2$}\label{sec:7}

Given any integer $d$, the purpose of this section is to reduce the proof of Theorem~\ref{thm:main1} (and the $Q$-Classification Conjecture) for primes $p\neq\pm 1\,\text{mod}\,2d$ to a finite computation with complexity depending on $d$. This is achieved by the following theorem.

\begin{theorem}\label{thm:local}Let $\kappa\in\bF_p\backslash\{4\}$ with $p\geq 7$ a prime. If there exist integers $d\geq 4$ and $n\geq 3d$ such that \[\dim(\cP_{2n}(\bF_p,d))\geq\begin{cases}2n-4 &\kappa= 3\text{ and }d\neq 4,5\\ 2n-3 &\kappa= 2,2+\varphi,2+\overline{\varphi}\text{ or }\kappa=3\text{ and }d=4,5\\2n-2 &\text{otherwise},\end{cases}\] then Theorem \ref{thm:main1} (and the $Q$-Classification Conjecture) holds for $\kappa$ and $p$.\end{theorem}

We first prove an initial consequence of the rank bound above. Then the proof is paused, and what remains of it is split into several pieces (depending on $\kappa$) at the end of this section. 

Recall the vectors in (\ref{eq:y_vecs}) and (\ref{eq:yM}). We also define three new vectors, \begin{align*}\by_0&=\be_2^T-12\be_4^T,\\ \by_p&=(4-\kappa)^{-1}\be_{p-1}^T,\\ \text{and }\yR&=-\sum_{i=1}^{\frac{p-1}{2}}\left(\!\binom{2i}{i}\sum_{j=1}^i\binom{2j}{j}^{-1}\frac{\kappa^{j-1}}{j}\right)\!\be_{2i}^T.\end{align*}

It turns out that even if $\cP_{\!n}(\bF_p,d)$ has the dimension required in the statement of Theorem \ref{thm:local}, $\cP(\bF_p,\infty)$ is still not big enough to prove that $\cP(\bF_q)$ equals what we expect it to. Specifically, $\cP(\bF_p,\infty)$ can never eliminate the possibility that $\by_0$, $\by_p$, or $\yR$ lies in $\cP^{\perp}(\bF_p)$, but none of those vectors appear in the spans in Theorem~\ref{thm:P^perp}.

\begin{lemma}\label{lem:yvecs}Let $\kappa\in\bF_p\backslash\{4\}$. If the hypothesis of Theorem \ref{thm:local} holds, then $\cP^{\perp}(\bF_p)$ is contained in the span of \begin{enumerate}\item[(0)] $\yM$, $\by_p$, $\by_{\kappa}$, and $\by_0$ if $\kappa=0$, \item[(1)] $\yM$, $\yR$, $\by_p$, and $\by_{\kappa}$ when $\kappa\neq 0,2,3,2+\varphi,2+\overline{\varphi}$, \item[(2)] $\yM$, $\yR$, $\by_p$, $\by_{\kappa}$, and $\by_1$ when $\kappa=2$, 
\item[(3a)] $\yM$, $\yR$, $\by_p$, $\by_{\kappa}$, and $\by_{\varphi}$ when $\kappa=2+\varphi$, \item[(3b)] $\yM$, $\yR$, $\by_p$, $\by_{\kappa}$, and $\by_{\overline{\varphi}}$ when $\kappa=2+\overline{\varphi}$, or \item[(4)] $\yM$, $\yR$, $\by_p$, $\by_{\kappa}$, $\by_{2}$, and $\by_5$ when $\kappa=3$.\end{enumerate}\end{lemma}

\begin{proof}Our first claim is that if the hypothesis of Theorem \ref{thm:local} holds, then \begin{equation}\label{eq:p-1/2}\dim(\cP_{\!p-2}(\bF_p,\infty))\geq\begin{cases}p-6 &\kappa= 3\text{ and }d\neq 4,5\\ p-5 &\kappa= 2,2+\varphi,2+\overline{\varphi}\text{ or }\kappa=3\text{ and }d=4,5\\p-4 &\text{otherwise},\end{cases}\end{equation} If $2n > p-2$, this is immediate by linear algebra, so suppose $2n < p-2$. Since $\cP(\bF_p,\infty)$ automatically contains every odd degree monomial, filling in the gap between $2n$ and $p-2$ amounts to producing a polynomial in $\cP(\bF_p,\infty)$ of degree $2m$ for all $n<m\leq\frac{p-3}{2}$.

Consider such an $m$, and let $\tilde{n}=d\lceil\frac{m}{d}\rceil$. By Corollary \ref{cor:found_em}, $\Phi(\cP_{\!x}(\oZ,d))$ contains a polynomial of degree $2m$ with leading coefficient $2\smash{\binom{\tilde{n}+m-1}{\tilde{n}-m}}(4-\tilde{\kappa})^{\tilde{n}}\tilde{n}^{\tilde{n}}$, where $\tilde{\kappa}$ is any lift of $\kappa$ with respect to some surjection $\oZ\to\oF_p$. The image of this polynomial in $\cP(\bF_q,d)$ has degree $2m$ provided $\tilde{n}+m\leq p$, because then the leading coefficient does not vanish in $\oF_p$. 

If it happens that $\tilde{n}+m > p$, then $m > p - \tilde{n}\geq p-(m+d-1)$. This rearranges to $2m > p-d+1$. Since $d\leq \frac{1}{3}n$ and $n\leq\frac{p-3}{2}$ by hypothesis, we see that $2m$ is much too large to divide $p+1$ or $p-1$. Thus we may take $d=m$ in Corollary~\ref{cor:found_em} to obtain a polynomial in $\Phi(\cP_{\!x}(\oZ,m))$ of degree $2m$ with leading coefficient $2(4-\tilde{\kappa})^mm^m$. Again, the degree is maintained under the map $\oZ\to\oF_p$.

We have shown that $\cP_{p-2}(\bF_p,\infty)$ contains a polynomial of every degree between $2n$ and $p-2$, thereby proving (\ref{eq:p-1/2}). In particular, the orthogonal complement $\smash{\cP_{\!p-2}^\perp(\bF_p,\infty)}$ has dimension 6, 5, or 4, depending on $\kappa$ and $d$ as stated in (\ref{eq:p-1/2}). Our final claim is that the 6, 5, or 4 vectors listed in the statement of this lemma span $\smash{\cP_{\!p-2}^\perp(\bF_p,\infty)}$ and thus $\cP^{\perp}\!(\bF_p)$. This will complete the proof.

The complement $\smash{\cP_{\!p-2}^\perp(\bF_p,\infty)}$ evidently contains $\by_p=(4-\kappa)^{-1}\be_{p-1}^T$ because degrees in $\smash{\cP_{\!p-2}(\bF_p,\infty)}$ are bounded by $p-2$.

Now we turn to $\by_{\kappa}=2\bx(0)+\frac{1}{2}(\bx(\sqrt{\kappa})+\bx(-\sqrt{\kappa}))$. Again let $\tilde{\kappa}$ be a lift of $\kappa$ with respect to some surjection $\oZ\to\oF_p$. We will use the case (1) orbit from Theorem \ref{thm:nonessential}, call it $\tilde{\cO}\coloneqq\Gamma\cdot(\sqrt{\tilde{\kappa}},0,0)\subset\cM(\oZ),$ to prove that $\by_\kappa$ belongs to $\smash{\cP_{\!p-2}^\perp(\bF_p,\infty)}$ for any $\kappa\in\bF_p\backslash\{4\}$. By construction of $\by_\kappa$, it suffices to show that $\sum_{\tilde{\cO}}\!f(\bt) = 0$ for an any $f\in\cP_x(\oZ,d)$ with $d\geq 4$. (Note that when $d=2$ or $3$, there are no primes $p\geq 7$ with $p\not\equiv\pm1\,\text{mod}\,2d$.) To do this, we will show that $\sum_{\tilde{\cO}}\!f^*(\bt) = 0$, and we may assume without loss of generality that $f^*$ is even as a polynomial in $y$ because $\tilde{\cO}$ is closed under $(x,y,z)\mapsto(x,-y,-z)$. Observe that $\tilde{\cO}$ breaks into first coordinate orbits $\tilde{\cO}_{\!\!\sqrt{\kappa}}\coloneq\Gamma_{\!x}\cdot (\sqrt{\kappa},0,0)$, $\tilde{\cO}_{\!-\!\sqrt{\kappa}}\coloneq\Gamma_{\!x}\cdot (-\sqrt{\kappa},0,0)$ and $\tilde{\cO}_0\coloneqq\Gamma_{\!x}\cdot(0,\sqrt{\kappa},0)$. By Proposition \ref{prop:P_equiv}, specifically (1) implies (5), if $f^*=f_n(x)y^{2n}+\cdots +f_0(x)$ then $(x^2-\kappa)\,|\,f_0(x)$. Since every $y$-coordinate in $\tilde{\cO}_{\!\!\sqrt{\kappa}}$ is 0, we have \[\sum_{\mathclap{\bt\in\tilde{\cO}_{\!\!\sqrt{\kappa}}}}f^*(\bt)=\sum_{\mathclap{\bt\in\tilde{\cO}_{\!\!\sqrt{\kappa}}}}f_0(x)=|\tilde{\cO}_{\!\!\sqrt{\kappa}}|f_0(\sqrt{\kappa})=0.\] The sum of $f$ over $\tilde{\cO}_{\!-\!\sqrt{\kappa}}$ vanishes by the same reasoning. Regarding $\tilde{\cO}_0$, since $\text{ord}(0)=4$ is not divisible by $2d$, we have $\sum_{\tilde{\cO}_0}\!f^*(\bt)=0$ by definition of $\cP_{\!x}(\oZ,d)$. Summing over all first-coordinate orbits gives $\sum_{\tilde{\cO}}\!f^*(\bt)=0$, as desired.

A similar argument using the case (2) orbit $\Gamma\cdot(1,1,0)$ from Theorem \ref{thm:nonessential} proves that $\by_1=\bx(0)+\frac{3}{2}(\bx(1)+3\bx(-1))$ belongs to $\smash{\cP_{\!p-2}^\perp(\bF_p,\infty)}$ when $\kappa=2$. Indeed, the orbit entries $0$ and $\pm 1$ have rotation order $4$ and $6$, which are not divisible by $2d$ (assuming there exist primes $p\geq 7$ with $p\not\equiv \pm 1\,\text{mod}\,2d$). 

In case (3a) and (3b), where $\kappa=2+\varphi$ or $2+\overline{\varphi}$, all orbit entries have order 4, 6, or 10. This appears to create potential trouble for proving that $\by_{\varphi}$ or $\by_{\overline{\varphi}})$ belongs to $\cP^{\perp}\!(\bF_p,d)$ in the special case $d=5$. However, in order for $\kappa=2+\varphi$ or $2+\overline{\varphi}$ to be possible when $p\geq 7$, we must have $p\equiv\pm1\,\text{mod}\,10$, implying $\cP(\bF_p,5)=0$ by definition. Thus $\by_{\varphi}\in\cP^\perp\!(\bF_q,\infty)$ when $\kappa=2+\varphi$ and $\by_{\overline{\varphi}}\in\cP^\perp\!(\bF_q,\infty)$ when $\kappa=2+\overline{\varphi}$.

When $\kappa=3$, the orbits in cases (4a) and (4b) of Theorem \ref{thm:nonessential} contain elements of order 4, 6, 8, and 10. There are two possibilities for failure: when $d=4$ and $p\not\equiv\pm 1\,\text{mod}\,8$ the vector \[\by_2= 2\bx(0)+\tfrac{3}{2}(\bx(1)+\bx(-1))+2(\bx(\sqrt{2})+\bx(-\sqrt{2}))\] may not lie in $\cP^\perp\!(\bF_p,4)$ because $\text{ord}(\pm\sqrt{2})=8$ is divisible by $2d$; and when $d=5$ and $p\not\equiv\pm 1\,\text{mod}\,10$ the vector \[\by_5= 2\bx(0)+3(\bx(1)+\bx(-1))+\tfrac{5}{2}(\bx(\varphi)+\bx(-\varphi))+\tfrac{5}{2}(\bx(\overline{\varphi})+\bx(-\overline{\varphi}))\] may not lie in $\cP^\perp\!(\bF_p,5)$ because $\text{ord}(\pm\varphi) = \text{ord}(\pm\overline{\varphi})=10$ is divisible by $2d$. We have accounted for this in the statement Theorem \ref{thm:local} and in (\ref{eq:p-1/2}) by increasing the rank requirement by one for $\kappa = 3$ when $d=4$ or $5$. This decreases the number of spanning vectors that we must find for $\cP_{\!p-2}^{\perp}(\bF_p,\infty)$ by one. (It turns out that $\by_2$ and $\by_5$ do \textit{not} belong to $\cP^\perp_{\!p-2}(\bF_p,4)$ or $\cP^\perp_{\!p-2}(\bF_p,5)$, respectively. In fact, experimentation suggests that they are not orthogonal to single nonzero even polynomial in $\cP_x(\oZ_p,4)$ or $\cP_x(\oZ,5)$.)

From the list of vectors in the lemma statement, it remains to consider $\yM$ and $\yR$ for general $\kappa$ and $\by_0$ for $\kappa=0$. 

Let $f\in\cP_x(\oZ,d)$, and suppose $f^*=f_n(x)y^{2n}+\cdots+f_0(x)$. Choose the lift $\tilde{\kappa}$ of $\kappa$ to lie in the interval $(0,4)$ so that we may employ the compact smooth surface $\cM^\circ\subset\cM(\bR)$. We have \begin{flalign}\label{eq:0_int}\nonumber&& 0&=\int_{-\sqrt{\tilde{\kappa}}}^{\sqrt{\tilde{\kappa}}}\frac{0}{\sqrt{4-x^2}}dx && \\ \nonumber && &=\int_{-\sqrt{\tilde{\kappa}}}^{\sqrt{\tilde{\kappa}}}\frac{\sum_j\!\binom{2j}{j}(\frac{x^2-\tilde{\kappa}}{x^2-4})^jf_j}{\sqrt{4-x^2}}dx && \text{by Proposition \ref{prop:P_equiv}}\\ \nonumber&& &= \frac{1}{2\pi}\int_{-\sqrt{\tilde{\kappa}}}^{\sqrt{\tilde{\kappa}}}\left(\sum_{j=0}^nf_j\!\int_{\cO_x}\!\!y^{2j}dA_x\right)dx && \text{by Corollary \ref{cor:integral}}\\ \nonumber&& &= \frac{1}{2\pi}\int_{\!\cM^\circ}\hspace{-1ex}f^*dA && \\\nonumber && & = \frac{1}{2\pi}\int_{\!\cM^\circ}\hspace{-1ex}\Phi(f^*)dA && \text{by }\Gamma\text{-invariance of }dA\text{ and }\cM^{\circ} \\&& &= \frac{1}{2\pi}\int_{\!\cM^\circ}\hspace{-1ex}\Phi(f)dA && \text{by Proposition \ref{prop:Phi_x(f^*)}.}\end{flalign} Only the variable $x$ appears in the argument of the last integral. The integral of any odd-degree monomial term of $\Phi(f)$ vanishes since $\cM^\circ$ is closed under $(x,y,z)\mapsto(-x,-y,z)$. Consider what happens when we integrate an even-degree monomial terms of $\Phi(f)$: \[\frac{1}{2\pi}\int_{\cM^{\circ}}\hspace{-1ex}x^{2j}dA=\int_{-\sqrt{\tilde{\kappa}}}^{\sqrt{\tilde{\kappa}}}\left(\frac{x^{2j}}{2\pi}\!\int_{\cO_x}\hspace{-1ex}dA_x\right)dx=\int_{-\sqrt{\tilde{\kappa}}}^{\sqrt{\tilde{\kappa}}}\frac{x^{2j}}{\sqrt{4-x^2}}dx\] by Corollary \ref{cor:integral}. This last integral can be computed using integration by parts and induction on $j$. The result is \[2\binom{2j}{j}\!\arcsin\!\left(\!\frac{\sqrt{\tilde{\kappa}}}{2}\right)-\sqrt{4\tilde{\kappa}-\tilde{\kappa}^3}\binom{2j}{j}\sum_{i=1}^j\binom{2i}{i}^{-1}\frac{\tilde{\kappa}^{i-1}}{i}.\] Let $\Phi(f)=c_{\tilde{n}}x^{2\tilde{n}}+\cdots +c_0$. By substituting the formula above into (\ref{eq:0_int}), we see that \[0=2\arcsin\!\left(\!\frac{\sqrt{\tilde{\kappa}}}{2}\right)\!\sum_{j=0}^{\tilde{n}}\binom{2j}{j}c_j-\sqrt{4\tilde{\kappa}-\tilde{\kappa}^3}\sum_{j=1}^{\tilde{n}}\binom{2j}{j}c_j\sum_{i=1}^j\binom{2i}{i}^{\!\!-1}\frac{\tilde{\kappa}^{i-1}}{i}.\] By the Hermite--Lindemann theorem, $\arcsin(\frac{1}{2}\sqrt{\tilde{\kappa}})$ is transcendental because $\frac{1}{2}\sqrt{\tilde{\kappa}}$ is algebraic and nonzero \cite[Chapter 1]{waldschmidt}. This forces both \[0=\sum_{j=0}^{\tilde{n}}\binom{2j}{j}c_j\hspace{\parindent}\text{and}\hspace{\parindent}0=\sum_{j=1}^{\tilde{n}}\binom{2j}{j}c_j\sum_{i=1}^j\binom{2i}{i}^{\!\!-1}\frac{\tilde{\kappa}^{i-1}}{i}.\] In particular, the $\Phi$ reduction of any linear combination of polynomials in $\cP_x(\oZ,d)$, with $d$ allowed to vary, also satisfies the two equations above. If the image in $\cP(\bF_p,\infty)$ of such a combination has degree at most $p-2$ (or even $p-1$), we see that it is orthogonal to $\yM$ and $\yR$. Note that no generality was lost by only considering $f$ for which $f^*$ is even as a polynomial in $y$---deleting odd powers of $y$ neither affects a polynomial's membership in $\cP_x(\oZ,d)$ nor the orthogonality to $\yM$ or $\yR$ of its image in $\cP_{\!p-2}(\bF_p,\infty)$.

(Remark that $\yR$ is also in $\cP_{\!p-2}^\perp(\bF_p,\infty)$ when $\kappa=0$, but it need not be listed in case (0) of the lemma statement because \[\yR=-\frac{1}{2}\sum_{i=1}^\frac{p-1}{2}\binom{2i}{i}\be_{2i}^T=-\frac{1}{2}\yM+\frac{1}{6}\by_{\kappa}.\] This is the one case where $\yR$ actually belongs in $\cP^\perp(\bF_p)$.)

Finally, consider $\kappa=0$ and $\by_0=\be_2^T-12\be_4^T$.  We take $\tilde{\kappa}=0$ as our lift of $\kappa$. As usual, let $f\in\cP_x(\oZ,d)$ for some $d\geq 4$, and assume without loss of generality that $f^*$ is even as a polynomial in $y$. By Corollary \ref{cor:eigen_basis}, $\Phi(f)$ is a linear combination of the polynomials $\Phi(\scp_{\alpha,i})$ for $i\geq 2$ and $\alpha\in\widehat{\Lambda}_i$ with $2d\,|\,\text{ord}(\alpha)$. To prove that $\Phi(f)$ is orthogonal to $\by_0$ (taking inner products in characteristic 0), we will show that every such $\Phi(\scp_{\alpha,i})$ is orthogonal to $\by_0$. We are only considering those $\alpha$ for which $2d\,|\,\text{ord}(\alpha)$, which excludes $\alpha=0$. Thus $\Phi(\scp_{\alpha,i})$ is orthogonal to $\by_0$ if and only if the same is true of $\alpha^5\Phi(\scp_{\alpha,i}) = \Phi(x^5\scp_{\alpha,i})$. We claim that the $\Phi$ reduction of each monomial term in $x^5\scp_{\alpha,i}$ is, by itself, orthogonal to $\by_0$. In fact, this is true of the $\Phi$ reduction of any monomial of total degree at least five. Indeed, up to permuting variables, the only monomial $x^\ell y^m z^n$ of total degree five with $\ell\equiv m\equiv n\,\text{mod}\,2$ is $x^3 yz$. We have $\Phi(x^3yz) = x^4+12x^2$, which is orthogonal to $\by_0$. The reduction of any monomial with odd total degree is also orthogonal to $\by_0$ by Proposition \ref{prop:degree}. Therefore, by induction on total degree, $\Phi(x^\ell y^m z^n)$ is orthogonal to $\by_0$ whenever $\ell+m+n\geq 5$.\end{proof}

The last step in the proof of Theorem \ref{thm:local} is to whittle the spanning vectors from Lemma \ref{lem:yvecs} down to those in the statement of Theorem \ref{thm:P^perp}. We achieve this using the extra polynomials $\Phi((x^{p+1}-x^2)f)\in\cP(\bF_p)$, where $f=f(x,y,z)$ is chosen to make the $\Phi$ computation as simple as possible (while avoiding those $f$ for which $\Phi((x^{p+1}-x^2)f)\in\cP(\bF_p,\infty)$). The most convienient polynomials $f$ turn out to come from the as yet unused eigenvalues of $A_n$ in Lemma \ref{lem:eigenvector}, namely $\pm 2$. (Recall that $\scp_{\alpha,n}$ has only been defined for $\alpha\in\widehat{\Lambda}_n= \Lambda_n\backslash\{\pm 2\}$.) The eigenvectors of $A_n$ of eigenvalue $2$ and $-2$ extend to generalized eigenvectors of $M_n$, that we denote $\scp_{2,n}$ and $\scp_{-2,n}$. It turns out that $\Phi((x^{p+1}-x^2)\scp_{2,n})$ and $\Phi((x^{p+1}-x^2)\scp_{-2,n})$ are not orthogonal to the extra vectors that appear in Lemma \ref{lem:yvecs} but do not appear in Theorem \ref{thm:P^perp}. Demonstrating this is the only goal for the remained of this section.

The next several pages are devoted to finding a formula for $\Phi((x^{p+1}-x^2)\scp_{2,n}^+)$ where $\scp_{2,n}^+ = \frac{1}{2}(\scp_{2,n}+\scp_{-2,n})$. Similar to Section \ref{sec:6}, we achieve this without ever finding the coefficients of $\scp_{2,n}^+$. (In particular, our strategy is not to first compute $\scp_{2,n}^+$ then apply Theorem \ref{thm:gen_form} to $\Phi(x^{p+1}\scp_{2,n}^+)$ and $\Phi(x^2\scp_{2,n}^+)$). This formula is essentially Proposition \ref{prop:qn_coef}, which provides the coefficient of $x^{2i}$ in $\Phi((x^{p+1}-x^2)\scp_{2,n}^+)$ when $i\geq n$. For general $\kappa$, we only end up needing a complete formula for $\Phi((x^{p+1}-x^2)\scp_{2,n}^+)$ for $n\leq 4$, so Proposition \ref{prop:qn_coef} leaves very little computational work to be done. Then Proposition \ref{prop:prodforms} helps computes the dot product of the coefficient vector of $\Phi((x^{p+1}-x^2)\scp_{2,n}^+)$ with $\by_p$, $\yR$, and $\by_{\kappa}$. This allows us to verify that no linear combination of $\by_p$  and $\yR$ when $\legendre{\kappa}{p} = 1$ (the Legendre symbol) or $\by_p$,  $\yR$, and $\by_{\kappa}$ when $\legendre{\kappa}{p}=-1$ is orthogonal to each $\Phi((x^{p+1}-x^2)\scp_{2,n}^+)$ for $n\leq 4$. Eliminating those two or three (depending on $\legendre{\kappa}{p}$) vectors from the span of $\cP^{\perp}\!(\bF_p)$ is all we must do for $\kappa\neq 0,3$.  For $\kappa=0$ we proceed similarly, but with the aim of eliminating $\by_p$ and $\by_0$ from $\cP^{\perp}\!(\bF_p)$; and for $\kappa=3$ we aim to eliminate $\by_p$, $\yR$, and potentially $\by_{\kappa}$, $\by_2$ and $\by_5$ depending on $\legendre{\kappa}{p}$, $\legendre{2}{p}$, and $\legendre{5}{p}$. 

Let us now proceed to defining $\scp_{2,n}^+$ and computing $\Phi((x^{p+1}-x^2)\scp_{2,n}^+)$.

Over any integral domain $R$ of characteristic $0$, the matrix $M_n$ defined in (\ref{eq:M_n}) has $n+1$ generalized eigenvectors with eigenvalue 2, denote them $\bp_0,...,\bp_n$, that satisfy $M_n\bp_0=2\bp_0$ and $M_n\bp_i = 2\bp_i + \bp_{i-1}$. Let us briefly justify this existence claim. Since $M_n$ is nearly a block diagonal matrix with blocks $A_0,...,A_n$, it is convenient to name the corresponding block vectors that concatenate to form $\bp_i$. Call them $\bp_{i,0},...,\bp_{i,n}$, so $\bp_{i,j}$ is a $(j+1)$-dimensional column vector over the field of fractions $F$. Recall from Lemma \ref{lem:eigenvector} that the algebraic and geometric multiplicities of the eigenvalue 2 for each of the diagonal blocks $A_0,...,A_n$ is exactly one and that the 2-eigenvector of $A_i$ is \[\bp_{i,i}\coloneqq c\begin{bmatrix}1\\2\\\vdots\\2\\1\end{bmatrix}\in R^{i+1}\] for $c\in F$. Extend this vector upward with zeros by letting $\bp_{i,j}$ be the zero vector when $j > i$. Linear algebra then guarantees that there exists a nonzero choice of scalar $c$ and components $\bp_{i,j}$ for $j < i$ that make $\bp_i$ a generalized eigenvector. Furthermore, $\bp_i$ is not a genuine eigenvector for $i\geq 1$ because there is no solution $\bp_{i,i-1}$ to the vector equation \[B_{i-1}\bp_{i,i}+A_{i-1}\bp_{i,i-1}=2\bp_{i,i-1}.\] Indeed, the sum of the entries in $B_{i-1}\bp_{i,i}$ is $2i-1\neq 0$, but the sum of the entries in any one of the columns of $A_{i-1}-2\text{Id}_i$ is 0. Thus $\bp_i$ must satisfy $M_n\bp_i=2\bp_i+\bp_{i-1}$ (assuming $\bp_{i-1}$ was constructed similarly with $\bp_{i-1,i-1}$ being the largest nonzero block component). 

Note that these vectors are not yet uniquely determined. Since $M_n\bp_0=2\bp_0$, adding any multiple of $\bp_0$ to $\bp_i$ preserves the equation $M_n\bp_i=2\bp_i+\bp_{i-1}$. We take the unique choice of $\bp_i$ for which \begin{equation}\label{eq:bottomzero}\bp_{i,0}=[0]\text{ for }i\geq 1.\end{equation} Setting $\bp_{0,0} = 1$, we have now defined $\bp_i$ for $i\geq 1$. 

The first three generalized eigenvectors are \begin{equation}\label{eq:first3}\bp_0=\begin{bmatrix}0\\ \smash{\vdots}\\ \smash{\vdots}\\\smash{\vdots}\\\smash{\vdots}\\\smash{\vdots}\\\smash{\vdots}\\0\\1\end{bmatrix},\hspace{\parindent}\bp_1=\begin{bmatrix}0\\ \smash{\vdots}\\ \smash{\vdots}\\\smash{\vdots}\\\smash{\vdots}\\ 0 \\ 1 \\1\\0\end{bmatrix},\hspace{\parindent}\text{and }\bp_2=\frac{1}{6}\begin{bmatrix}0\\ \smash{\vdots}\\ 0\\4\\8\\4\\1\\0\\0\end{bmatrix}.\end{equation}

Not decorating $\bp_i$ with an ``$n$" is an abuse of notation because there is otherwise no indication of what dimension the vector lives in. That is, the number of zero entries above $\bp_{i,i}$ in (\ref{eq:first3}) is not specified. We eliminate this abuse of notation by switching to polynomials. 

\begin{notation}Let $\scp_{2,n}\in F[x,y,z]$ denote the polynomial corresponding $\bp_n$ with respect to the basis $\cB_n$ defined at the start of Section \ref{sec:6}. Also let $\scp_{-2,n}(x,y,z) \coloneqq \scp_{2,n}(-x,-y,z)$ and $\scp_{2,n}^+=\frac{1}{2}(\scp_{2,n}+\scp_{-2,n})$.\end{notation}

Care must be taken in translating $M_n\bp_n = 2\bp_n+\bp_{n-1}$ into a polynomial equation since $\scp_{2,n}$ is defined according to $\cB_n$ while $\scp_{2,n-1}$ is defined according to $\cB_{n-1}$. For $i\leq n-1$, the monomials in $\cB_{i,n-1}$ must be scaled by $(x^2-\kappa)$ to match those in $\cB_{i,n}$.  Thus $M_n\bp_n = 2\bp_n+\bp_{n-1}$ becomes \begin{equation}\label{eq:recursion}\Phi_x(x\scp_{2,n})=2\scp_{2,n}+(x^2-\kappa)\scp_{2,n-1}\end{equation} by Proposition \ref{prop:Phi_to_mat}. Let us apply this formula to help compute $\Phi(x^m\scp_{2,n})$ for any $m$. (We need not reference the vectors $\bp_i$ again. They are simply a convenient way to prove that the $\scp_{2,i}$ exist and to compute their coefficients.) 

The lemma below holds for any $m$ and $n$ by regarding empty sums and improper binomial coefficients as 0. Its forthcoming application, however, is to compute $\Phi(x^m\scp_n)$ for $n=1,2,3,4,5$ and large $m$. So coefficients of the vast majority of powers of $x$, namely $x^i$ for $i > n$, are fully determined by the next lemma's right-side sum. As we will show, the left-side sum only contributes to coefficients of $x^i$ for $i\leq n$.

\begin{lemma}\label{lem:gen_eigen}For any $m\geq 0$ and $n\geq 1$, \begin{align*}\Phi(x^m\scp_{2,n})=\sum_{i=0}^{n-1}\binom{m}{i}2^{m-i}\Phi\big((x^2&-\kappa)^i\scp_{2,n-i}\big)\\[-2.5ex]&+2^{m-n}(x^2-\kappa)^n\sum_{i=0}^{m-n}\binom{m-i-1}{n-1}\frac{x^i}{2^i}.\end{align*}\end{lemma}

\begin{proof}We will prove the equality above with both occurrences of ``$\Phi$" replaced by ``$\Phi_x$". That is, \begin{align}\label{eq:phiX}\nonumber\Phi_x(x^m\scp_{2,n})=\sum_{i=0}^{n-1}\binom{m}{i}2^{m-i}\Phi_x\big((x^2&-\kappa)^i\scp_{2,n-i}\big)\\[-2.5ex]&+2^{m-n}(x^2-\kappa)^n\sum_{i=0}^{m-n}\binom{m-i-1}{n-1}\frac{x^i}{2^i}.\end{align} This would imply the lemma by applying $\Phi$ to (\ref{eq:phiX}) and using $\Phi\circ\Phi_x = \Phi$. 

To prove (\ref{eq:phiX}) we use induction on $m$. In the base case, $m=0$, (\ref{eq:phiX}) becomes the vacuous assertion $\Phi_x(\scp_{2,n})=\Phi_x(\scp_{2,n})$. Now assume (\ref{eq:phiX}) holds for some $m\geq 0$. Observe that \begin{align}\nonumber\Phi_x(x^{m+1}\scp_{2,n})&=\Phi_x(x^m\Phi_x(x\scp_{2,n}))\\\nonumber&=\Phi_x(x^m\Phi_x(2\scp_{2,n} + (x^2-\kappa)\scp_{2,n-1}))\\\label{eq:twoterms}&=2\Phi_x(x^m\scp_{2,n}) + \Phi_x((x^2-\kappa)\Phi_x(x^m\scp_{2,n-1})).\end{align}

Here we consider the case $n=1$ separately from $n\geq 2$. Recall that $\scp_{2,0} = 1$. Applying the equation above when $n=1$ followed by the induction hypothesis gives \begin{align*}\Phi_x(x^{m+1}\scp_{2,1})&=2\Phi_x(x^m\scp_{2,1}) + (x^2-\kappa)x^m\\&=2\!\left(\!2^m\Phi_x(\scp_{2,1})+2^{m-1}(x^2-\kappa)\sum_{i=0}^{m-1}\frac{x^i}{2^i}\!\right)+(x^2-\kappa)x^m\\&=2^{m+1}\Phi_x(\scp_{2,1})+2^m(x^2-\kappa)\sum_{i=0}^m\frac{x^i}{2^i},\end{align*} which is (\ref{eq:phiX}) with $n=1$ and $m$ replaced by $m+1$. This completes the induction step when $n=1$. 

For $n\geq 2$, we apply the induction hypothesis to both terms $\Phi_x(x^m\scp_{2,n})$ and $\Phi_x(x^m\scp_{2,n-1})$ appearing in (\ref{eq:twoterms}). Each application results in an expression of form (\ref{eq:phiX}), and each of these expressions contains two summations that we call the ``left-side sum" and the ``right-side sum". We will show that the two left-side sums from $\Phi_x(x^m\scp_{2,n})$ and $\Phi_x(x^m\scp_{2,n-1})$ combine to give a new left-side sum of the same form with index incremented by one, and that the same thing happens when the two right-side sums are combined. 

By applying the induction hypothesis to $\Phi_x(x^m\scp_{2,n-1})$, the left-side sum contribution from $\Phi_x((x^2-\kappa)\Phi_x(x^m\scp_{2,n-1}))$ is \begin{align}\label{eq:leftside}\nonumber\Phi_x\Bigg((x^2-\kappa)\sum_{i=0}^{n-2}\binom{m}{i}2^{m-i}&\Phi_x\big((x^2-\kappa)^i\scp_{2,n-1-i}\big)\Bigg)\\ \nonumber&=\sum_{i=0}^{n-2}\binom{m}{i}2^{m-i}\Phi_x\big((x^2-\kappa)^{i+1}\scp_{2,n-1-i}\big)\\&=\sum_{i=1}^{n-1}\binom{m}{i-1}2^{m+1-i}\Phi_x\big((x^2-\kappa)^i\scp_{2,n-i}\big).\end{align} By applying the induction hypothesis to $\Phi_x(x^m\scp_n)$, the left-side sum contribution from $2\Phi_x(x^m\scp_n)$ is \[2\sum_{i=0}^{n-1}\binom{m}{i}2^{m-i}\Phi_x\big((x^2-\kappa)^i\scp_{2,n-i}\big)=\sum_{i=0}^{n-1}\binom{m}{i}2^{m+1-i}\Phi_x\big((x^2-\kappa)^i\scp_{2,n-i}\big).\] Adding this to (\ref{eq:leftside}) gives \begin{align*}\sum_{i=0}^{n-1}\left(\!\binom{m}{i}+\binom{m}{i-1}\!\right)2^{m+1-i}&\Phi_x\big((x^2-\kappa)^i\scp_{2,n-i}\big)\\&=\sum_{i=0}^{n-1}\binom{m+1}{i}2^{m+1-i}\Phi_x\big((x^2-\kappa)^i\scp_{2,n-i}\big),\end{align*} which is the left-side sum in (\ref{eq:phiX}) with $m$ replaced by $m+1$.

As for the two right-side sums resulting from (\ref{eq:twoterms}), we begin with the contribution from $\Phi_x((x^2-\kappa)\Phi_x(x^m\scp_{2,n-1}))$ again: \begin{align*}\Phi_x\Bigg(\!(x^2-\kappa)\Bigg(\!2^{m-n+1}(x^2-\kappa)^{n-1}\!\!\!&\sum_{i=0}^{m-n+1}\!\binom{m-i-1}{n-2}\frac{x^i}{2^i}\Bigg)\!\Bigg)\\&=2^{m+1-n}(x^2-\kappa)^n\!\!\!\sum_{i=0}^{m+1-n}\!\binom{m-i-1}{n-2}\frac{x^i}{2^i}.\end{align*} The contribution from $2\Phi_x(x^m\scp_{2,n})$ is \[2\!\left(\!2^{m-n}(x^2-\kappa)^n\! \sum_{i=0}^{m-n}\!\binom{m-i-1}{n-1}\frac{x^i}{2^i}\right) = 2^{m+1-n}(x^2-\kappa)^n\!\!\! \sum_{i=0}^{m+1-n}\!\!\binom{m-i-1}{n-1}\frac{x^i}{2^i}.\] Again we combine the two contributions to get \begin{align*}2^{m+1-n}(x^2-\kappa)^n\!\!\!\sum_{i=0}^{m+1-n}\!\bigg(\!\binom{m-i-1}{n-1}&+\binom{m-i-1}{n-2}\!\bigg)\frac{x^i}{2^i}\\&=2^{m+1-n}(x^2-\kappa)^n\!\!\!\sum_{i=0}^{m+1-n}\!\binom{m-i}{n-1}\frac{x^i}{2^i},\end{align*} which is the right-side sum in (\ref{eq:phiX}) with $m$ replaced by $m+1$. This completes the induction step when $n\geq 2$ and thus the proof.\end{proof}

This lemma applies equally well to computing coefficients of $\Phi(x^m\scp_{-2,n})$. Indeed, if $m$ is even then the coefficients of even powers of $x$ in $\Phi(x^m\scp_{-2,n})$ match those of $\Phi(x^m\scp_{2,n})$ by Proposition \ref{prop:degree}, while the coefficients of odd powers of $x$ in $\Phi(x^m\scp_{-2,n})$ are the negations of those of $\Phi(x^m\scp_{2,n})$. (If $m$ is odd, this gets reversed.) In particular, setting $\scp_{2,n}^+=\frac{1}{2}(\scp_{2,n}+\scp_{-2,n})$ means $\Phi(x^m\scp_{2,n}^+)$ only picks up the even-power coefficients of $\Phi(x^m\scp_{2,n})$ when $m$ is even.

\begin{notation}\label{not:qn}For $n\geq 1$, let $\bq_n\in\bF_p^p$ be the vector corresponding to $\Phi((x^{p+1}-x^2)\scp_{2,n}^+)$ as described in Notation \ref{not:P^perp(F_q)}.\end{notation}

\begin{proposition}\label{prop:qn_coef}Let $1\leq n\leq \frac{p-1}{2}$. If $i\geq n$, the entry in the $2i^\text{th}$ coordinate of $\bq_n$ is \begin{equation}\label{eq:large_coef}-2^{2-n-2i}\sum_{j=0}^n\binom{n}{j}\binom{2i-2j+n-2}{n-1}(-4)^j\kappa^{n-j}.\end{equation} If $1\leq i < n$ and $n\geq 3$, the entry in the $2i^\text{th}$ coordinate of $\bq_n$ is the sum of (\ref{eq:large_coef}) and \begin{equation}\label{eq:small_coef}c_{2i}+2^{2-n-2i}\sum_{j=0}^n\binom{n}{j}\binom{2j-2i+1}{n-1}4^j(-\kappa)^{n-j},\end{equation} where $c_i$ is the coefficient of $x^i$ in $-2\Phi((x^2-\kappa)\scp_{2,n-1})-\Phi((x^2-\kappa)^2\scp_{2,n-2})$. If $n\geq 2$, the entry in the zeroth coordinate of $\bq_n$ is $c_0$. The entry in the second coordinate of $\bq_2$ is $\frac{1}{2}\kappa(4-\kappa)$, and the entry in the zeroth coordinate of $\bq_1$ is $0$.\end{proposition}

\begin{proof}It is a small computation using Lemma \ref{lem:gen_eigen} to verify that the first one or two entries of $\bq_1$ and $\bq_2$ are as claimed. Let us turn to the $2i^\text{th}$ coordinate formula when either $i\geq 3$ and $n\geq 1$ or $i\geq 1$ and $n\geq 3$.

Determining $\bq_n$ requires computing $\Phi(x^2\scp_{2,n})$ and $\Phi(x^{p+1}\scp_{2,n})$. For each of these reductions, Lemma \ref{lem:gen_eigen} provides a formula involving two summations that we again call the left-side sum and the right-side sum. We evaluate the two left-side sums first.

When $n=1$, the left-side sums from $\Phi(x^{p+1}\scp_{2,1})$ and $-\Phi(x^2\scp_{2,1})$ cancel and therefore contribute nothing to the coefficient of $x^{2i}$. When $n=2$, the difference of the two left-side sums is $-2\Phi((x^2-\kappa)\scp_{2,1}) = -2\Phi((x^2-\kappa)(y^2+yz))$, which has degree $2$ and thus contributes nothing to the coefficient of $x^{2i}$ again. For $n\geq 3$, we have the following two left-side sums: \begin{align*}\sum_{i=0}^{n-1}&\binom{p+1}{i}2^{p+1-i}\Phi((x^2-\kappa)^i\scp_{2,n-i})-\sum_{i=0}^{n-1}\binom{2}{i}2^{2-i}\Phi((x^2-\kappa)^i\scp_{2,n-i})\\&=\sum_{i=0}^{1}\binom{p+1}{i}2^{2-i}\Phi((x^2-\kappa)^i\scp_{2,n-i})-\sum_{i=0}^2\binom{2}{i}2^{2-i}\Phi((x^2-\kappa)^i\scp_{2,n-i})\\&=-2\Phi((x^2-\kappa)\scp_{2,n-1})-\Phi((x^2-\kappa)^2\scp_{2,n-2})\\&=\sum_{i\geq 0}c_ix^i,\end{align*} where the final equality is simply the definition of the $c_i$. To justify the appearance of $c_{2i}$ in (\ref{eq:small_coef}) but not in (\ref{eq:large_coef}), we must show that $c_{2i}=0$ when $i\geq n$. In other words, we must show that $-2\Phi((x^2-\kappa)\scp_{2,n-1})-\Phi((x^2-\kappa)^2\scp_{2,n-2})$ has degree strictly less than $2n$. Indeed, by definition of the basis $\cB_{n-1}$ from Section \ref{ss:6.1}, $(x^2-\kappa)\scp_{2,n-1}$ is a linear combination of polynomials of the form $(x^2-\kappa)^{i+1}y^jz^k$ with $2i + j + k = 2n-2$ and $j\geq k$. By Proposition \ref{prop:degree}, the $\Phi$ reduction of such a polynomial has degree at most $2i+2 + k$, which is strictly less than $n$ unless $j = k = 0$ and $i = n-1$. But recall from the construction of $\scp_{2,n-1}$ that the coefficient of $(x^2-\kappa)^{2n-2}$ is zero (the last entry in $\bp_{n-1}$ as shown in (\ref{eq:first3}). Thus $\Phi((x^2-\kappa)\scp_{2,n-1})$ has degree at most $2n-2$. The same argument applies to $\Phi((x^2-\kappa)^2\scp_{2,n-2})$. Altogether, we have shown that the left-side sums from Lemma~\ref{lem:gen_eigen} are fully accounted for by $c_{2i}$ in (\ref{eq:small_coef}). 

Regarding the right-side sums, $\Phi(x^2\scp_{2,n})$ contributes nothing when $n\geq 3$ because the summation interval of $i=0$ to $2-n$ is empty. When $n=1$ or $n=2$ the right-side sum contributes monomials of degree 3 or 4, respectively, which also do not contribute to the coefficient of $x^{2i}$. The right-side sum from $\Phi(x^{p+1}\scp_{2,n})$ is \[2^{p+1-n}(x^2-\kappa)^n\!\!\sum_{i=0}^{p+1-n}\!\!\binom{p-i}{n-1}\frac{x^i}{2^i}.\] Here we replace $2^{p+1}$ with $2^2$ and expand and distribute $(x^2-\kappa)^n$. Collecting all terms of degree $2i$ gives \begin{align}\label{eq:pn_simple}\nonumber2^{2-n}\!\!\!\sum_{j=0}^{\min(i,n)}\!\bigg(\!\binom{n}{j}&x^{2j}(-\kappa)^{n-j}\bigg)\!\left(\!\binom{p-2i+2j}{n-1}\frac{x^{2i-2j}}{2^{2i-2j}}\right)\\&=2^{2-n-2i}\!\!\!\sum_{j=0}^{\min(i,n)}\!\!\binom{n}{j}\binom{p-2i+2j}{n-1}4^j(-\kappa)^{n-j}x^{2i}\end{align} 

Let us first suppose that the degree $2i$ above exceeds $p$ (making $\min(i,n)=n$). Let $\tilde{i}=i-\frac{p-1}{2}$ so that $x^{2i}\equiv x^{2\tilde{i}}\,\text{mod}\,(x^p-x)$. Substituting $\tilde{i}$ into the expression above, we get the following contribution to the $\tilde{i}^\text{th}$ coordinate of $\bq_n$: \[2^{2-n-2\tilde{i}}\sum_{j=0}^{n}\binom{n}{j}\binom{2j-2\tilde{i}+1}{n-1}4^j(-\kappa)^{n-j}.\] This is precisely the sum in (\ref{eq:small_coef}). Note that this summation neither appears in (\ref{eq:large_coef}) nor in the formula for the zeroth coefficient of $\bq_n$. It does not appear in (\ref{eq:large_coef}) because the degree of the right-side sum in Lemma \ref{lem:gen_eigen} (including the factor $(x^2-\kappa)^n$) is $p+1+n$. Thus $2\tilde{i}=2i-(p-1)\leq (p+1+n)-(p-1)=n+2$, which is strictly less than $2n$ when $n\geq 3$. It does not appear in the formula for the zeroth coefficient of $\bq_n$ because $2\tilde{i} = 2i-(p-1) > p - (p-1) > 0$.

Next let us suppose the degree $2i$ in (\ref{eq:pn_simple}) is less than $p$. In this case we use \[\binom{p-2i+2j}{n-1}\equiv (-1)^{n-1}\binom{2j-2i+n-2}{n-1}\,\text{mod}\,p\] (which can fail if $2i-2j > p$) to rewrite the coefficient of $x^{2i}$ in (\ref{eq:pn_simple}) as \[-2^{2-n-2i}\sum_{j=0}^n\binom{n}{j}\binom{2i-2j+n-2}{n-1}(-4)^j\kappa^{n-j}.\] Remark that replacing the summation bound $\min(i,n)$ in (\ref{eq:pn_simple}) with $n$ makes no difference since the binomial coefficient above vanishes when $j\geq i$. The expression above matches (\ref{eq:large_coef}).\end{proof}

To facilitate forthcoming computations, especially the dot product $\bq_n(\yR)$, we express $\bq_n$ in terms of the following vectors.

\begin{notation}\label{not:e,f}For $j=0,1,...,\frac{p-1}{2}$, let \[\bf_j=(4-\kappa)^{j+1}\sum_{i=j}^\frac{p-1}{2}\binom{i}{j}\frac{\be_{2i}}{4^i}.\]\end{notation}

Thanks to Proposition \ref{prop:qn_coef}, a complete formula for $\bq_n$ is obtained by computing $2\Phi((x^2-\kappa)\scp_{n-1})+\Phi((x^2-\kappa)^2\scp_{n-2})$ in order to find the coefficients $c_{2i}$ in (\ref{eq:small_coef}). This can be done by hand for the small values of $n$ we will need (at most 5, though computing $\bq_5$ by hand could take a few hours), but it is also easy to code an algorithm to do the work. Either way, once every entry of $\bq_n$ has been determined, we may express it as a combination of $\bf_0,\bf_1,\dots,\bf_{n-1}$ and $\be_0,\be_2,\dots,\be_{2n-2}$ as follows: First, the binomial coefficients in (\ref{eq:large_coef}) can be viewed as polynomials in the variable $i$ and written in terms of $\binom{i}{0},\binom{i}{1},...,\binom{i}{n-1}$, say \begin{equation}\label{eq:binom}\binom{2i-2k+n-2}{n-1}=\sum_{j=0}^{n-1}a_{j,k}\binom{i}{j}.\end{equation} The correct choice of coefficients makes this equation hold provided $2i-2k+n-2\geq 0$. (The equation above can fail when $2i-2k+n-2< 0$ because in this case the left-side binomial coefficient vanishes while the right side may not. For example, consider which $i$ make \[\binom{2i-3}{2}=6\binom{i}{0}-5\binom{i}{1}+4\binom{i}{2}\] true or false. Note that if we were treating the left side as a \textit{generalized} binomial coefficient, which we are not, the equation would hold for all $i$.) Once the $a_{j,k}$ are computed, we may substitute (\ref{eq:binom}) into (\ref{eq:large_coef}) and swap the order of summation. We conclude that all but the first $n$ coordinates of $\bq_n$ match those of \[-\sum_{j=0}^{n-1}\left(\frac{2^{2-n}}{(4-\kappa)^{j+1}}\sum_{k=0}^n\binom{n}{k}a_{j,k}(-4)^k\kappa^{n-k}\right)\bf_j.\] Then the first $n$ coordinates of the expression above can be substracted from their correct values to obtain the ``adjustment" linear combination of $\be_0,\be_2,\dots,\be_{2n-2}$ that provides an exact formula for $\bq_n$. This is also straightforward to compute by hand or by code. Results for $n\leq 4$ are below. The coefficient matrices have been transposed to save space.

{\renewcommand{\arraystretch}{1.2}\begin{align}\label{eq:q1toq4}\nonumber\begin{bmatrix}\bq_1\\\bq_2\\\bq_3\\\bq_4\end{bmatrix}=&\begin{bmatrix}2 & 16 & 120-6\kappa & 896-96\kappa\\ 0 & -2 & -18-\frac{3}{2}\kappa & -144-8\kappa-\kappa^2\\ 0 & 0 & 2 & 20+3\kappa\\ 0 & 0 & 0 & -2\end{bmatrix}^{\!T}\begin{bmatrix}\bf_0\\\bf_1\\\bf_2\\\bf_3\end{bmatrix}\\&-(4-\kappa)\!\begin{bmatrix}2 & 16-\kappa & 120-\frac{46}{3}\kappa & 896-\frac{7816}{45}\kappa+\frac{166}{45}\kappa^2 \\ 0 & 2 & 20-\frac{4}{3}\kappa & \frac{2596}{15}-\frac{967}{45}\kappa+\frac{7}{45}\kappa^2\\ 0 & 0 & \frac{4}{3} & \frac{604}{45}-\frac{11}{9}\kappa\\ 0 & 0 & 0 & \tfrac{16}{15}\end{bmatrix}^{\!T}\begin{bmatrix}\be_0\\\be_2\\\be_4\\\be_6\end{bmatrix}\end{align}} To reinforce Notations \ref{not:qn} and \ref{not:e,f}, the second line \[\bq_2=16\bf_0-2\bf_1-(4-\kappa)((16-\kappa)\be_0+2\be_2)\] above asserts that \begin{align*}\Phi\big((x^{p+1}-&x^2)\big(\tfrac{1}{6}(4y^4+4y^2z^2+(x^2-\kappa)y^2)\big)\big)\\&\equiv (4-\kappa)\!\left(\sum_{i=0}^{\frac{p-1}{2}}(16-2(4-\kappa)i)\frac{x^{2i}}{4^i}-16+\kappa-2x^2\right)\text{mod}\,(x^p-x).\end{align*} Note that the term $\frac{1}{6}(8y^3z)$, which might be expected to appear based on the ``8" in (\ref{eq:first3}), is absent because $\bq_2$ is defined using $\scp_{2,2}^+$ rather than $\scp_{2,2}$.

\begin{lemma}\label{lem:centbinom}For any nonnegative integers $n$ and $j$, \[4^n\sum_{i=j}^n\binom{2i}{i}\binom{i}{j}\frac{1}{4^i}=\frac{2n+1}{2j+1}\binom{n}{j}\binom{2n}{n}.\]\end{lemma}

\begin{proof}Letting $a_j(n)$ and $b_j(n)$ denote the left- and right-side expressions above, we observe that both satisfy the same recursion for $n\geq 1$ and $j\geq 0$: \[a_j(n)=4a_j(n-1)+\binom{2n}{n}\binom{n}{j}\hspace{\parindent}\text{and}\hspace{\parindent}b_j(n)=4b_j(n-1)+\binom{2n}{n}\binom{n}{j}.\] Since $a_0(0)=1=b_0(0)$ and $a_j(0)=0=b_j(0)$ when $j\geq 1$, the claim follows by induction on $n$.\end{proof}

Recall the variable vector $\bx=\sum_ix^i\be_i$ as well as the vectors $\bx(\alpha)=\sum_i\alpha^i\be_i$ for $\alpha\in\bF_p$ defined in Notation \ref{not:y_M}.

\begin{lemma}\label{lem:fj(x)}For any nonnegative integer $j$ and $x\in\overline{\bF}_p\backslash\{4\}$, \[\bf_j(\bx)=4x^{2j}\!\left(\!\frac{4-\kappa}{4-x^2}\!\right)^{\!\!j+1}\!\!\left(\!1-x^{p-1}\!\sum_{i=0}^j\!\binom{2i}{i}\!\!\left(\frac{1}{4}-\frac{1}{x^2}\right)^{\!\!i}\right)-4x^{p-1}\!\!\left(\frac{\kappa}{4}-1\right)^{\!j+1}\!\!\binom{2j}{j}.\]\end{lemma}

\begin{proof}Directly from the definition of $\bf_j$, \begin{equation}\label{eq:fj(x)}\frac{\bf_j(\bx)}{(4-\kappa)^{j+1}}\;=\;\sum_{i=j}^\frac{p-1}{2}\binom{i}{j}\frac{x^{2i}}{4^i}\;=\;4x^{2j}\sum_{i=j}^{\frac{p-3}{2}}\binom{i}{j}4^{\frac{p-3}{2}-i}x^{2(i-j)}+\binom{\frac{p-1}{2}}{j}x^{p-1}.\end{equation}
The summation in the final expression above (as well as in the middle expression) is the $j^\text{th}$ derivative with respect to $x^2$ of a geometric series, which can be evaluated with the product rule: \begin{align*}4x^{2j}\sum_{i=j}^{\frac{p-3}{2}}\binom{i}{j}4^{\frac{p-3}{2}-i}(x^2)^{i-j}&=\frac{4x^{2j}}{j!}\frac{d^j}{dx^{2j}}\!\left(\frac{4^{\frac{p-1}{2}}-x^{p-1}}{4-x^2}\right)\\&=\frac{4x^{2j}}{(4-x^2)^{j+1}}\left(1-x^{p-1}\sum_{i=0}^j\binom{\frac{p-1}{2}}{i}\!\left(\frac{4}{x^2}-1\right)^{\!\!i}\right).\end{align*} Now we substitute \begin{equation}\label{eq:centbinom}\binom{\frac{p-1}{2}}{i}\equiv \frac{(-1)^i}{4^i}\binom{2i}{i}\,\text{mod}\,p\end{equation} into the expression above as well as in the final term of (\ref{eq:fj(x)}). Scaling both sides of (\ref{eq:fj(x)}) by $(4-\kappa)^{j+1}$ completes the proof.\end{proof}

We make use of the Legendre symbol $\legendre{\alpha}{p}=\alpha^{\frac{p-1}{2}}$ for $\alpha\in\bF_p$ (including $\alpha=0$).

\begin{proposition}\label{prop:prodforms}If $1\leq j \leq \tfrac{p-3}{2}$ and $\kappa\in\bF_p\backslash\{4\}$, then \begin{align*}\be_{2j}(\by_p)&=0,&\bf_j(\by_p)&=\left(\frac{\kappa}{4}-1\right)^{\!j}\!\binom{2j}{j},\\\be_{2j}(\by_{\bR})&=-\binom{2j}{j}\!\sum_{i=1}^{j}\!\binom{2i}{i}^{\!\!-1}\!\frac{\kappa^{i-1}}{i},&\!\bf_j(\yR)&=\frac{2\kappa^j}{2j+1}\!\left(\!1-\legendre{\kappa}{p}\!\sum_{i=0}^j\!\binom{2i}{i}\!\!\left(\frac{1}{4}-\frac{1}{\kappa}\right)^{\!i}\,\right)\!,\\\be_{2j}(\by_{\kappa})&=\kappa^j,& \bf_j(\by_{\kappa})&=(4j+2)\bf_j(\by_{\bR})+\legendre{\kappa}{p}(4-\kappa)\bf_j(\by_p).\end{align*} Furthermore, when $j=0$ the first four formulas hold, while \[\be_0(\by_{\kappa})=3\hspace{\parindent}\text{and}\hspace{\parindent}\bf_0(\by_{\kappa})=12-\left(2+\legendre{\kappa}{p}\right)\!\kappa.\]\end{proposition}

\begin{proof}The formulas involving $\be_{2j}$ are immediate from the definitions of $\by_p$, $\yR$, and $\by_{\kappa}$. Also immediate from the definitions of $\by_p$ and $\bf_j$ is \[\bf_j(\by_p)=(4-\kappa)^j\binom{\frac{p-1}{2}}{j}.\] Combining this with (\ref{eq:centbinom}) proves the formula for $\bf_j(\by_p)$.

Next we establish the formula for $\bf_j(\by_{\kappa})$. Observe that \begin{flalign}\label{eq:fjug}\nonumber&&\frac{\bf_j(\by_{\bR})}{(4-\kappa)^{j+1}}&=-\sum_{k=j}^{\frac{p-1}{2}}\binom{k}{j}\binom{2k}{k}\frac{1}{4^k}\sum_{i=1}^k\binom{2i}{i}^{\!\!-1}\frac{\kappa^{i-1}}{i}&& \\\nonumber && &=-\sum_{i=0}^\frac{p-3}{2}\binom{2i+2}{i+1}^{\!\!-1}\!\frac{\kappa^i}{i+1}\sum_{k=i+1}^\frac{p-1}{2}\binom{k}{j}\binom{2k}{k}\frac{1}{4^k}&& \\\nonumber && &=-\sum_{i=0}^\frac{p-3}{2}\binom{2i+2}{i+1}^{\!\!-1}\!\frac{\kappa^i}{i+1}\left(\sum_{k=j}^\frac{p-1}{2}\binom{k}{j}\binom{2k}{k}\frac{1}{4^k}-\sum_{k=j}^{i}\binom{k}{j}\binom{2k}{k}\frac{1}{4^k}\right)&& \\\nonumber && &=-\sum_{i=0}^\frac{p-3}{2}\binom{2i+2}{i+1}^{\!-1}\!\frac{\kappa^i}{i+1}\left(0-\frac{2i+1}{2j+1}\binom{i}{j}\binom{2i}{i}\frac{1}{4^i}\right) && \hspace{-2cm}\text{by Lemma \ref{lem:centbinom}}\\\nonumber && &=\frac{1}{4j+2}\sum_{i=0}^\frac{p-3}{2}\binom{i}{j}\frac{\kappa^i}{4^i} && \\\nonumber && &=\frac{1}{4j+2}\left(\sum_{i=0}^\frac{p-1}{2}\binom{i}{j}\frac{\kappa^i}{4^i}-\binom{\frac{p-1}{2}}{j}\kappa^{\frac{p-1}{2}}\right) && \\ && &=\frac{1}{4j+2}\left(\frac{\bf_j(\bx_\kappa)}{(4-\kappa)^{j+1}}-\legendre{\kappa}{p}\frac{\bf_j(\by_p)}{(4-\kappa)^j}\right). && \end{flalign} Since $\bf_j(\bx_\kappa)=\bf_j(\by_{\kappa})$ when $j\geq 1$, this establishes the claimed linear relation.

It remains only to verify the formula for $\bf_j(\by_{\bR})$. Letting $x=\sqrt{\kappa}$ in Lemma \ref{lem:fj(x)} gives \begin{align*}\bf_j(\bx(\sqrt{\kappa}))&=4\kappa^j\!\left(1-\legendre{\kappa}{p}\sum_{i=0}^j\binom{2i}{i}\left(\frac{1}{4}-\frac{1}{\kappa}\right)^{\!i}\right)-4\legendre{\kappa}{p}\!\left(\frac{\kappa}{4}-1\right)^{j+1}\!\binom{2j}{j}\\&=4\kappa^j\!\left(1-\legendre{\kappa}{p}\sum_{i=0}^j\binom{2i}{i}\left(\frac{1}{4}-\frac{1}{\kappa}\right)^{\!i}\right)+\legendre{\kappa}{p}(4-\kappa)\bf_j(\by_p).\end{align*} Substituting this into (\ref{eq:fjug}) completes the proof.\end{proof}

\begin{proof}[Proof of Theorem \ref{thm:local} when $\kappa\neq 0,3$] Several of these vectors from the list in Lemma \ref{lem:yvecs} must be eliminated to match Theorem \ref{thm:P^perp}. Specifically, if $\legendre{\kappa}{p}=1$, we must prove that no nonzero linear combination of $\by_{\bR}$ and $\by_p$ lies in $\cP^\perp\!(\bF_p)$; all of the other vectors Lemma \ref{lem:yvecs} belong when $\kappa\neq 0,3$. If $\legendre{\kappa}{p}=-1$ we must prove that no nonzero linear combination of $\by_{\bR}$, $\by_p$, and $\by_{\kappa}$ lies in $\cP^\perp(\bF_p)$; again, all other vectors belong.

Suppose $\kappa\neq 0,3$ and  $\legendre{\kappa}{p}=1$. Proposition \ref{prop:prodforms} gives \[\begin{bmatrix}\bf_0\\\bf_1\end{bmatrix}\begin{bmatrix}\by_{\bR} & \by_p\end{bmatrix}=\begin{bmatrix} 0 & 1\\ \frac{4}{3}-\frac{1}{3}\kappa & -2+\frac{1}{2}\kappa\end{bmatrix}\] and \[\begin{bmatrix}\be_0\\\be_2\end{bmatrix}\begin{bmatrix}\by_{\bR} & \by_p\end{bmatrix}=\begin{bmatrix} 0 & 0\\ -1 & 0\end{bmatrix}.\] The top-left $2\times 2$ corner of \ref{eq:q1toq4} expresses $[\bq_1\;\bq_2]^T$ as a linear combination (with matrix coefficients) of $[\bf_0\;\bf_1]^T$ and $[\be_0\;\be_2]^T$. This linear combination combines with the equations above to give \[\begin{bmatrix}\bq_1\\\bq_2\end{bmatrix}\begin{bmatrix}\by_{\bR} & \by_p\end{bmatrix}=\begin{bmatrix}0 & 2 \\ \frac{16}{3}-\frac{4}{3}\kappa & 20-\kappa\end{bmatrix}.\] The matrix above is nonsingular with determinant $-\frac{8}{3}(4-\kappa)\neq 0$.

Now suppose $\legendre{\kappa}{p}=1$. Since we need only eliminate a 3-dimensional space, namely $\textup{span}\{\yR,\by_p,\by_{\kappa}\}$, we might hope that the same strategy used above works with only the three vectors $\bq_1$, $\bq_2$, and $\bq_3$. Unfortunately, \[\begin{bmatrix}\bq_1 \\ \bq_2\\\bq_3\end{bmatrix}\begin{bmatrix}\by_{\bR} &\by_p &\by_{\kappa}\end{bmatrix}\] is singular when $p\equiv1\,\text{mod}\,4$ and $\kappa=\pm 4\sqrt{-1}$ (nonsingular otherwise). But by including $\bq_4$, we can obtain a matrix of rank three. Indeed, let \[\tilde{\bq}_3=15(272+72\kappa-3\kappa^2)\bq_3-105(4+\kappa)\bq_4.\] Then (\ref{eq:q1toq4}) tells us {\renewcommand{\arraystretch}{1.2}\begin{align*}\begin{bmatrix}\bq_1\\\bq_2\\\tilde{\bq}_3\end{bmatrix}=&\begin{bmatrix}2 & 16 & 113280+51360\kappa-1800\kappa^2+270\kappa^3\\ 0 & -2 & -12960-7080\kappa+450\kappa^2+\frac{345}{2}\kappa^3\\ 0 & 0 & -240-1200\kappa-405\kappa^2\\ 0 & 0 & 840+210\kappa\end{bmatrix}^T\begin{bmatrix}\bf_0\\\bf_1\\\bf_2\\\bf_3\end{bmatrix}\\&-(4-\kappa)\!\begin{bmatrix}2 & 16-\kappa & 113280+\frac{137728}{3}\kappa-5272x^2+\frac{908}{3}\kappa^3 \\ 0 & 2 & 8912+\frac{21040}{3}\kappa-149\kappa^2+\frac{131}{3}\kappa^3\\ 0 & 0 & -\frac{592}{3}+544\kappa+\frac{205}{3}\kappa^2\\ 0 & 0 & -448-112\kappa\end{bmatrix}^T\begin{bmatrix}\be_0\\\be_2\\\be_4\\\be_6\end{bmatrix}.\end{align*} With another application of Proposition \ref{prop:prodforms}, we obtain
\begin{align*}\begin{bmatrix}\bf_0 \\ \bf_1 \\ \bf_2 \\ \bf_3 \end{bmatrix}\begin{bmatrix}\by_{\bR} & \by_p & \by_{\kappa}\end{bmatrix}&\\&\hspace{-2cm}=\begin{bmatrix} 4 & -\frac{4}{3}+\frac{5}{3}\kappa & \frac{12}{5}-2\kappa+\frac{23}{20}\kappa^2 & -\frac{40}{7}+6\kappa-\frac{5}{2}\kappa^2+\frac{51}{56}\kappa^3\\ 1 & -2+\frac{1}{2}\kappa & 6-3\kappa+\frac{3}{8}\kappa^2 & -20+15\kappa-\frac{15}{4}\kappa^2+\frac{5}{16}\kappa^3\\ 12 - \kappa & 6\kappa+\frac{1}{2}\kappa^2 & -2\kappa+7\kappa^2+\frac{3}{8}\kappa^3 & 4\kappa-5\kappa^2+\frac{31}{4}\kappa^3+\frac{5}{16}\kappa^4\end{bmatrix}^T\end{align*} and \[\begin{bmatrix} \be_0 \\ \be_2 \\ \be_4 \\ \be_6 \end{bmatrix}\begin{bmatrix} \by_{\bR} & \by_p & \by_{\kappa}\end{bmatrix}=\begin{bmatrix} 0 & 0 & 3\\ -1 & 0 & \kappa\\ -3-\frac{1}{2}\kappa & 0 & \kappa^2\\ -10-\frac{5}{3}\kappa-\frac{1}{3}\kappa^2 & 0 & \kappa^3\end{bmatrix}\]

The appropriate multiplications and additions with the matrices above gives \begin{align*}\begin{bmatrix}\bq_1\\\bq_2\\\tilde{\bq}_3\end{bmatrix}\begin{bmatrix} \by_{\bR} & \by_p & \by_{\kappa}\end{bmatrix}&\\&\hspace{-2.4cm}=\begin{bmatrix}8 & \frac{224}{3}-\frac{16}{3}\kappa & 480384+217600\kappa-26000\kappa^2+1440\kappa^3-\frac{55}{2}\kappa^4 \\ 2 & 20-\kappa & 120960+60960\kappa-5160\kappa^2+390\kappa^3\\ 4\kappa & 24\kappa-2\kappa^2  & 182400\kappa+58080\kappa^2-10440\kappa^3+450\kappa^4\end{bmatrix}^T\!\!.\end{align*}} This matrix is nonsingular (if $\kappa\neq 0$) with determinant $-2^{19}\kappa$.\end{proof}

\begin{proof}[Proof of Theorem \ref{thm:local} when $\kappa=0$]The space that must be eliminated from $\cP^{\perp}\!(\bF_p)$ is only two-dimensional, spanned by $\by_p$ and $\by_0\coloneqq\be_2^T-12\be_4^T$. (Recall that $\kappa=0$ is the unique value for which $\by_{\bR}$ actually belongs in $\cP^\perp\!(\bF_p)$. ) To eliminate it, we need only \[\bq_1=2\bf_0-8\be_0\hspace{\parindent}\text{and}\hspace{\parindent}\bq_2=16\bf_0-2\bf_1-64\be_0-8\be_2.\] Applying Proposition \ref{prop:prodforms} gives

\[\det\left(\begin{bmatrix}\bq_1\\\bq_2\end{bmatrix}\begin{bmatrix}\by_0 & \by_p\end{bmatrix}\right)=\det\begin{bmatrix}-4 & 2\\ 0 & 20\end{bmatrix}=-80,\]
which is nonzero in $\bF_p$ for $p\geq 7$.\end{proof}

\begin{proof}[Proof of Theorem \ref{thm:local} when $\kappa=3$] We consider $\legendre{3}{p}=-1$ and $\legendre{3}{p}=1$ separately.

If $\legendre{3}{p}=-1$ then $\by_{\kappa}$ must be removed from the spanning set for $\cP^{\perp}\!(\bF_p)$ along with $\by_{\bR}$ and $\by_p$. If $\legendre{2}{p}=-1$ or $\legendre{5}{p}=-1$ then $\by_2$ or $\by_5$ must also be removed. This requires at least five vectors, and the natural choice $\bq_n$ for $n\leq 5$ turns out to work. The following coefficients have been calculated with computer assistance using Proposition \ref{prop:qn_coef}: {\renewcommand{\arraystretch}{1.2}\begin{align}\label{eq:q1toq5at3}\nonumber\begin{bmatrix}\bq_1\\ \bq_2\\ \bq_3\\ \bq_4\\ \bq_5\end{bmatrix}=& \begin{bmatrix}2&0&0&0&0\\ 16&-2&0&0&0\\ 102&-\frac{45}{2}&2&0&0\\ 608&-177&29&-2&0\\ 3540&-\frac{9695}{8}&\frac{2185}{8}&-\frac{71}{2}&2\end{bmatrix} \begin{bmatrix}\bf_0\\\bf_1\\\bf_2\\\bf_3\\\bf_4\end{bmatrix}\\  &\hspace{1cm}-\begin{bmatrix}2&0&0&0&0\\ 13&2&0&0&0\\ 74&16&\frac{4}{3}&0&0\\ \frac{6122}{15}&110&\frac{439}{45}&\frac{16}{15}&0\\ \frac{706396}{315}&\frac{74516}{105}&\frac{59084}{945}&\frac{2504}{315}&\frac{32}{35}\end{bmatrix}\begin{bmatrix}\be_0\\\be_2\\\be_4\\\be_6\\\be_8\end{bmatrix}.\end{align}

Next, it is straightforward to verify from Proposition \ref{prop:prodforms} (for $\by_{\bR}$) or directly from the definition of $\by_p$, $\by_{\kappa}$, $\by_2$ or $\by_5$ that \begin{equation}\label{eq:e0toe5at3}\begin{bmatrix}\be_0\\\be_2\\\be_4\\\be_6\\\be_8\end{bmatrix}\begin{bmatrix}\by_{\bR} & \by_p & \by_{\kappa} & \by_2 & \by_5\end{bmatrix}=\begin{bmatrix}0&0&3&9&18\\ -1&0&3&11&21\\ -\frac{9}{2}&0&9&19&41\\ -18&0&27&35&96\\ -\frac{279}{4}&0&81&67&241\end{bmatrix}.\end{equation} To evaluate $\bf_j$ at $\by_2$ or $\by_5$, we first evaluate $\bf_j$ at $\bx(0)$, $\bx(1)$, $\bx(\sqrt{2})$ and $\bx(\varphi)+\bx(\overline{\varphi})$ using Lemma \ref{lem:fj(x)}. Note that $x^{p-1}$ appears in the Lemma \ref{lem:fj(x)}'s formula for $\bf_j(\bx)$. We only need to eliminate $\by_2$ when $\legendre{2}{p}=-1$, so $\sqrt{2}^{\,p-1}$ takes the value $-1$ when computing $\bf_j(\bx(\sqrt{2}))$. Similarly, we need only eliminate $\by_5$ when $\legendre{5}{p}=-1$, so $\varphi^{p-1}$ takes the value \begin{flalign*}&& \left(\frac{1+\sqrt{5}}{2}\right)^{\!p-1}&=\sum_{i=0}^{p-1}\binom{p-1}{i}\sqrt{5}^i && \\ && &=\sum_{i=0}^\frac{p-1}{2}5^i-\sqrt{5}\sum_{i=0}^\frac{p-3}{2}5^i && \text{since }\binom{p-1}{i}\equiv (-1)^i\,\text{mod}\,p\\ && &=5^\frac{p-1}{2}+(1-\sqrt{5})\sum_{i=0}^\frac{p-3}{2}5^i\hspace{-0.5cm} && \\ && &=\frac{1-\sqrt{5}}{1+\sqrt{5}}=\frac{\overline{\varphi}}{\varphi} && \text{by telescoping and using }{\textstyle\legendre{5}{p}} = -1,\end{flalign*} 
and $\overline{\varphi}^{p-1}$ takes the value $\varphi/\overline{\varphi}$. The result from Lemma \ref{lem:fj(x)} is then \[\begin{bmatrix}\bf_0\\\bf_1\\\bf_2\\\bf_3\\\bf_4\end{bmatrix}\begin{bmatrix}\bx(0) & \bx(1) & \bx(\sqrt{2}) & \bx(\varphi) +\bx(\overline{\varphi})\end{bmatrix}=\begin{bmatrix}1 & 1 & 3 & 5\\ 0 & \frac{1}{6} & \frac{7}{2} & \frac{15}{2}\\ 0 & \frac{7}{72} & \frac{27}{8} & \frac{543}{40}\\ 0 & \frac{5}{432} & \frac{55}{16} & \frac{2063}{80} \\ 0 & \frac{175}{10368} & \frac{435}{128} & \frac{156107}{3200}\end{bmatrix}\] The last three columns are also the values of $\bf_j(\bx(-1))$, $\bf_j(\bx(-\sqrt{2}))$, and $\bf_j(\bx(-\varphi) + \bx(-\overline{\varphi}))$ because every odd coordinate in $\bf_j$ is $0$. Combining this with the definitions of $\by_2$ and $\by_5$ as linear combinations of $\bx(\alpha)$ vectors provides the last two columns below. Proposition \ref{prop:prodforms} provides the values of $\bf_j(\by_{\bR})$, $\bf_j(\by_p)$, and $\bf_j(\by_{\kappa})$ in the first three columns:
\begin{equation}\label{eq:f0tof5at3}\begin{bmatrix}\bf_0\\\bf_1\\\bf_2\\\bf_3\\\bf_4 \end{bmatrix}\begin{bmatrix}\by_{\bR} & \by_p & \by_3 & \by_2 & \by_5\end{bmatrix}=\begin{bmatrix}4 & 1 & 9& 17 &33\\ \frac{11}{3} & -\frac{1}{2} & \frac{45}{2}& \frac{29}{2} &\frac{77}{2}\\ \frac{27}{4} & \frac{3}{8} & \frac{537}{8}& \frac{331}{24} &\frac{1643}{24} \\ \frac{115}{8} &  -\frac{5}{16} & \frac{3225}{16} & \frac{1985}{144} &\frac{18577}{144}\\ \frac{19355}{576} & \frac{35}{128} & \frac{77385}{128} & \frac{47155}{3456} &\frac{4216639}{17280} \end{bmatrix}\end{equation}

Performing the appropriate multiplications and addition with the matrices in (\ref{eq:q1toq5at3}), (\ref{eq:e0toe5at3}), and (\ref{eq:f0tof5at3}) results in \[\begin{bmatrix}\bq_1\\ \bq_2\\ \bq_3\\ \bq_4\\ \bq_5\end{bmatrix}\begin{bmatrix}\by_{\bR} & \by_p & \by_{\kappa} & \by_2 & \by_5\end{bmatrix}=\begin{bmatrix}8&2&12&16&30\\ \frac{176}{3}&17&54&104&175\\ 361&114&264&568&914\\ \frac{21231}{10}&708&1362&3036&4818\\ \frac{7758343}{630}&4260&7272&16396&\frac{129766}{5}\end{bmatrix},\] which is nonsingular at all primes $p \geq 29$ with determinant $-\frac{2508193792}{525} = -2^{23}\cdot 3^{-1}\cdot 5^{-2}\cdot 7^{-1}\cdot 13\cdot 23$. (The McCullough--Wanderley conjectures have already been verified for $p < 29$ \cite{wanderley}.) Now, if it happens that $\legendre{2}{p}=1$ or $\legendre{5}{p}=1$, then the $\by_2$ or $\by_5$ columns simply become all $0$, in which case we are no longer concerned with $\by_2$ or $\by_5$ because they belong in $\cP^{\perp}\!(\bF_p)$ as per Theorem \ref{thm:P^perp}. The remaining $5\times 3$ or $5\times 4$ matrix must still have full column rank, so we are done. This completes the proof when $\legendre{3}{p}=-1$.

When $\legendre{3}{p}=1$, not only does the $\by_{\kappa}$ column vanish in the matrix above, but the $\by_{\bR}$ column changes due to the appearance of $\legendre{\kappa}{p}$ in Proposition \ref{prop:prodforms}'s formula for $\bf_j(\by_{\bR})$. Our new version of (\ref{eq:f0tof5at3}) is \[\begin{bmatrix}\bf_0\\\bf_1\\\bf_2\\\bf_3\end{bmatrix}\begin{bmatrix}\by_{\bR} & \by_p & \by_2 & \by_5\end{bmatrix}=\begin{bmatrix}0&1&17&33\\ \frac{1}{3}&-\frac{1}{2}&\frac{29}{2}&\frac{77}{2}\\ \frac{9}{20}&\frac{3}{8}&\frac{331}{24}&\frac{1643}{24}\\ \frac{59}{56}&-\frac{5}{16}&\frac{1985}{144}&\frac{18577}{144}\end{bmatrix}\] The only change to (\ref{eq:q1toq5at3}) is the deletion of the last row and column in eqch matrix because we only need to use $\bq_n$ for $n\leq 4$ now. Similarly, the only change to (\ref{eq:f0tof5at3}) is the deletion of the middle (corresponding to $\by_\kappa$) column and the last row. Performing the appropriate matrix multiplications and addition gives \[\begin{bmatrix}\bq_1\\ \bq_2\\ \bq_3\\ \bq_4\end{bmatrix}\begin{bmatrix}\by_{\bR} & \by_p & \by_2 & \by_5\end{bmatrix}=\begin{bmatrix}0&2&16&30\\ \frac{4}{3}&17&104&175\\ \frac{77}{5}&114&568&914\\ \frac{8753}{70}&708&3036&4818\end{bmatrix},\] which is again nonsingular with determinant $\frac{393216}{35} = 2^{17}\cdot 3\cdot 5^{-1}\cdot 7^{-1}$. As before, if it happens that $\legendre{2}{p}=-1$ or $\legendre{5}{p}=-1$, the $\by_2$ column or the $\by_5$ column vanishes, and no other columns change. This yields either a $4\times 3$ or $4\times 2$ matrix with full column rank, which completes the proof.}\end{proof}

\section{Computing $\dim (\cP_{2n}(\bF_p,d))$}\label{sec:8}

Given positive integers $d\leq n$, we provide an algorithm (Algorithm \ref{alg}) to verify the hypothesis of Theorem \ref{thm:local} for all $\kappa$ and $p\geq 2n$ with $p\not\equiv\pm 1\,\text{mod}\,2d$. Executing this Algorithm \ref{alg} in Sage for $d=5$, 7, 9, 11, 13, 16, and 17 proves Theorem \ref{thm:main1}. In particular, the McCullough--Wanderley conjectures hold when $2\,\text{lcm}(1,...,17)\centernot\mid p^2-1$.

The polynomials we use to build rank in $\cP(\oZ,d)$ take the form $gf_n$, where $f_n$ is defined in Notation \ref{not:f_n}. As in (\ref{eq:g_def}), we must choose $g$ to kill those $c_{\alpha,n}(f_n)$ for which $2d\centernot\mid\text{ord}(\alpha)$. Hence we define \begin{equation}\label{eq:u_d,n}g_{d,n}(x^2)\coloneqq x^\delta\prod_{\mathclap{\substack{\alpha\in\widehat{\Lambda}_n\\2d\,\centernot\mid\,\text{ord}(\alpha)}}}\;(x-\alpha),\end{equation} where $\delta$ is either $0$ or $1$, whichever makes $g_{d,n}$ an even polynomial (just as in (\ref{eq:g_def})). The roots of $g_{d,n}$ are precisely the undesired $\alpha$-values. 

\begin{proposition}For any $\ell\geq 0$, $x^{2\ell}g_{d,n}f_n\in \cP^n(\bZ,d)$.\end{proposition}

\begin{proof}Just as in the proof of Corollary \ref{cor:found_em}, this follows from Corollary \ref{cor:eigen_basis} and the fact that $c_{\alpha,n}(x^{2\ell}g_{d,n}f_n)=\alpha^{2\ell}g_{d,n}(\alpha^2)(\alpha^2-4)^n$, which vanishes when it needs to by definition of $g_{d,n}$.\end{proof}

Recall that if $\text{ord}(\alpha)\centernot\mid 2n$ then $c_{\alpha,n}(f)=0$ for any polynomial $f$. In particular, the choice of $g_{d,n}$ makes $\Phi(x^{2m}g_{d,n}f_n)$ the zero polynomial when $d\centernot\mid n$; these polynomials are only useful when $d\,|\,n$.

For computational convenience, we note that $g_{d,n}$ is a product/quotient of (slight variants of) Chebyshev polynomials of the second kind, which can be called directly by a Sage command. Let us demonstrate this when $d$ is a prime power, which is all we ultimately work with. The usual Chebyshev polynomials of the second kind are denoted $U_n(x)$ and defined recursively by $U_0(x)=1$, $U_1(x)=2x$, and $U_{n+1}(x)=2x U_n(x)-U_{n-1}(x)$. The roots of $U_{n-1}(x)$ are $\cos(\frac{m\pi}{n})$ for $m=1,...,n-1$. Now define \[u_n(x^2)\coloneqq\begin{cases}U_{n-1}(\frac{x}{2}) & n\text{ is odd}\\xU_{n-1}(\frac{x}{2}) & n\text{ is even,}\end{cases}.\] defined for $n\geq 1$. These are monic, integral polynomials, as can be seen from the recursion $U_0(\frac{x}{2})=1$, $U_1(\frac{x}{2})=x$, and $U_{n+1}(\frac{x}{2})=x U_n(\frac{x}{2})-U_{n-1}(\frac{x}{2})$. Furthermore, if $\alpha\in\widehat{\Lambda}_n$, then $\alpha=2\cos(\frac{m\pi}{n})$ for some $m=1,...,n-1$, so $u_n(\alpha^2)=0$. Therefore, if $d=p^a\neq 2$ and $2n=mp^b$, the equality \[g_{p^a,n}(x^2)=\begin{cases}u_{mp^{a-1}}(x^2) &p\neq 2\\ u_{mp^{a-2}}(x^2) &p=2,\end{cases}\] can be seen by comparing roots. Similar expressions exist when $d$ is not a prime power, but they involve quotients of Chebyshev polynomials in order to get each desired root exactly once from an inclusion-exclusion argument.

For each $n$ divisible by a given $d$, we now ask what rank is contributed by the polynomials $\Phi(x^{2m}g_{d,n}f)$ as $m$ ranges over nonnegative integers. This will tell us when to stop incrementing $m$ in Algorithm \ref{alg} for a given $n$, and it will help us find some $n_d$ for which the polynomials $\Phi(x^{2m}g_{d,n}f)$ for $n\leq n_d$ are expected to have the rank Theorem \ref{thm:local} requires, namely $2n_d-2$ for general $\kappa$. We know that $\Phi(x^{2m}g_{d,n}f_n)$ lives in the $\oQ$-span of $\Phi(\scp_{\alpha,n}^+)$ for $\alpha\in\Lambda_n$ of order divisible by $2d$. This means if $\alpha=\zeta^i+\zeta^{-i}$ for a primitive $2n^\text{th}$ root of unity $\zeta$, $2d$ must divide $\frac{2n}{(i,2n)}$. For each divisor $\tilde{d}$ of $\frac{n}{d}$, there are $\phi(\frac{2n}{\tilde{d}})$ values of $i\in\{1,...,2n\}$ that make $(i,2n)=\tilde{d}$. Replacing $i$ with $2n-i$, $n+i$, or $n-i$ produces the same value of $\alpha$ up to a sign, so for any fixed $d$ and $n$ the polynomials $\Phi(x^{2m}g_{d,n}f)$ span a space of dimension at most \begin{equation}\label{eq:m_d,n}m_{d,n}\coloneqq\left\lceil\frac{1}{4}\sum_{\tilde{d}\,|\frac{n}{d}}\phi\!\left(\frac{2n}{\tilde{d}}\right)\right\rceil.\end{equation} So a natural choice is to compute $\Phi(x^{2m}g_{d,n}f)$ for all nonnegative $m < m_{d,n}$ in Algorithm \ref{alg}.

Now we sum $m_{d,n}$ as $n$ ranges over consecutive multiples of $d$ to get the total expected dimension. If $n_d$ denotes the largest multiple of $d$ we use, then we require the sum of $m_{d,n}$ (which is the total number of polynomials to be reduced) to be at least $n_d-2$; nothing less could satisfy the bound in Theorem \ref{thm:local} when combined with the odd-degree monomials $x^1,x^3,\dots x^{2n_d-1}\in\cP_{2n_d}(\bF_p,d)$. (Note that the sum of $m_{d,n}$ for $n\leq n_d$ must eventually surpass any constant multiple of $n_d$ since $m_{d,n}$ grows roughly linearly in $n$.) In practice, however, to prove full dimension we typically need at least $n_d$ polynomials, not $n_d-2$. 

When $d$ is a prime power and $n$ is a small multiple of $d$, $m_{d,n}$ is particularly easy to compute. The reader may verify that for $d=p^a \geq 4$, defining \begin{equation}\label{eq:n_d}n_d\coloneqq\begin{cases}4d & p\neq 2,3\\ 5d & p= 3\\ 6d & p=2\end{cases}\end{equation} makes $\sum_n m_{d,n}\geq n_d$, the sum over $n = d,2d,\dots n_d$. (Remark that this inequality need not be verified to guarantee correctness of Algorithm \ref{alg}. Indeed, if $\sum_n m_{d,n}< n_d$, then there will not be enough polynomials to produce the desired rank in $\cP_{2n_d}(\bF_p,d)$, and the algorithm output will be inconclusive. Or conversely, if a smaller choice of $n_d$ would have made $\sum_n m_{d,n}< n_d$, it means only that the algorithm will work harder than necessary to achieve its output. In similar fashion, the choice of $m_{d,n}$ in (\ref{eq:m_d,n}) cannot affect the validity of Algorithm \ref{alg}. These are simply parameters whose optimal values we hope to have found.)

Finally, let us discuss how to test the dimension hypothesis in Theorem \ref{thm:local} for all $\kappa$ and $p\not\equiv\pm 1\,\text{mod}\,2d$ at once. Evidently, we must treat $\kappa$ as a variable when we compute $\Phi(x^{2m}g_{d,n}f_n)$. Thus all polynomial coefficients lie in $\bZ[\kappa]$. For generic $\kappa$, we only need coefficient vectors to have rank $n_d-2$, so we only store the top $n_d-2$ even-degree coefficients of each polynomial. (All odd-degree coefficients are 0.) That is, for $m<m_{d,n}$ and $n\leq n_d$, the coefficients of $x^{2i}$ in $\Phi(x^{2m}g_{d,n}f_n)$ for $i=3,4,...,n_d$ are stored as columns of a matrix. (This procedure can be found in lines 1--8 of Algorithm \ref{alg}.) Let $M_d$ denote the resulting matrix.

\begin{notation}\label{not:h}For $a,b_1,b_2,b_3,b_4\in\bZ$ with each $b_j\geq 0$, let \[h(\kappa;a,b_1,b_2,b_3,b_4) = a(\kappa-4)^{b_1}(\kappa-3)^{b_2}(\kappa-2)^{b_3}(\kappa^2-5\kappa+5)^{b_4}\in\bZ[\kappa].\]\end{notation}

\begin{proposition}Given $d$, let $n_d$ and $M_d$ be as defined above. Let $p > 2n_d$ be a prime with $p\not\equiv \pm 1\,\textup{mod}\,2d$. If the ideal in $\bZ[\kappa]$ generated by the $(n_d-2)\times (n_d-2)$ minor determinants of $M_d$ contains $h(\kappa;a,b,2,1,1)$ for some $a$ and $b\geq 0$, then the hypothesis of Theorem~\ref{thm:local} holds for all $\kappa\in\bF_p\backslash\{4\}$ provided $p\centernot\mid a$.\end{proposition}

\begin{proof}Reduce $M_d$ modulo $p$ so that the entries lie in $\bF_p[\kappa]$. The columns of $M_d\,\text{mod}\,p$ generate a free $\bF_p[\kappa]$-module because $\bF_p[\kappa]$ is a PID. Furthermore, if $p\centernot\mid a$ as in the proposition statement, this free module must have full rank $n_d-2$ because the ideal in $\bF_p[\kappa]$ generated by the $(n_d-2)\times(n_d-2)$ minor determinants of $M_d\,\text{mod}\,p$ contains the nonzero element $h(\kappa;a,b,2,1,1)$. So fix a basis for our free module and use it to form the columns of a new $(n_d-2)\times(n_2-d)$ matrix $M$. Without loss of generality, assume $M$ is upper triangular (again, $\bF_p[x]$ is a PID, so column operations put $M$ in Hermite normal form) so that $\det M$ is a product of its diagonal entries. Since $\det M$ divides every $(n_d-2)\times(n_d-2)$ minor determinant of $M_d\,\text{mod}\,p$, it must also divide $h(\kappa;a,b,2,1,1)$ in $\bF_p[\kappa]$. So up to scaling by a unit, every diagonal entry of $M$ is $\kappa-4$, $\kappa-3$, $\kappa-2$, $\kappa-2-\varphi$, $\kappa-2-\overline{\varphi}$, or some (perhaps empty) product of them. Furthermore, at most one diagonal entry can be divisible by $\kappa-2$, $\kappa-2-\varphi$, or $\kappa-2-\overline{\varphi}$, and at most two can be divisible by $\kappa-3$. Thus substituting a specific element of $\bF_p\backslash\{4\}$ in for $\kappa$ in the entries of $M$ produces a matrix of rank at least $n_d-4$ if $\kappa=3$, $n_d-3$ if $\kappa\in\{2,2+\varphi,2+\overline{\varphi}\}$, and $n_d-2$ otherwise. The same rank bound must hold for such a $\kappa$ substitution into $M_d\,\text{mod}\,p$ because the columns of $M_d\,\text{mod}\,p$ generate those of $M$. Extending $M_d\,\text{mod}\,p$ by three rows---the coefficients mod$\,p$ of $1$, $x^2$, and $x^4$ in each $\Phi(x^{2m}g_{d,n}f_n)$---cannot decrease the column rank. But now each extended column corresponds to an element of $\cP_{2n_d}(\bF_p,d)$, which therefore has the desired dimension.\end{proof}

The goal of Algorithm \ref{alg} is to find such an ideal element $a(\kappa-4)^b(\kappa-3)^2(\kappa-2)(\kappa^2-5\kappa+5)$, but with $a$ only divisible by primes to which Theorem \ref{thm:local} would not apply anyway: $p < 2n$ or $p\equiv\pm 1\,\text{mod}\,2d$. There are several natural ways to accomplish this in Sage, and the fastest appears to be with resultants. Specifically, for a given $(n_d-2)\times(n_d-2)$ minor of $M_d$, we compute its determinant in $\bZ[\kappa]$ and divide out all acceptable factors: up to one factor of $\kappa-2$ and $\kappa^2-5\kappa+5$, up to two factors of $\kappa-3$, and as many as possible factors of $\kappa-4$ or rational primes less than $2d$. (It would also be permissible to divide out primes $p\equiv\pm 1\,\text{mod}\,2d$, but we do not bother.) We keep this reduced determinant, and compute its resultant with each of the previously computed reduced minor determinants. The resultant of two polynomials in $\bZ[\kappa]$ is an integer that lies in the ideal they generate, so if the greatest common divisor of all our resultants is ever equal to $1$, we are done! This procedure can be found in lines 9--17 of Algorithm \ref{alg}.

\begin{algorithm}
    \DontPrintSemicolon
    \KwIn{A prime power $d\geq 5$}
    \KwOut{$\Delta\in\bZ$; conjectures hold if $p\centernot\mid\Delta$, $p\not\equiv\pm1\,\text{mod}\,2d$, and $p > 2n_d$.}
    $M_d\gets$ empty matrix over $\bZ[\kappa]$\;
    $n_d\gets$ integer from (\ref{eq:n_d})\;
    
    \For{$0<n\leq n_d$ \textup{such that} $d\,|\,2n$}{
        $f_n,g_{d,n}\gets$ polynomials from (\ref{not:f_n}),\,(\ref{eq:u_d,n})\;
        $m_{d,n}\gets$ integer from (\ref{eq:m_d,n})\;
        
        \For{$0\leq m < m_{d,n}$}{
            $a_0+\cdots +a_{n_d}x^{2n_d}\gets\Phi(x^{2m}g_{d,n}f_n)$\commentR{$\triangleright\;$reduce with $\kappa$ a variable}
            append column $[a_3\cdots a_{n_d}]^T$ to $M_d$\commentR{$\triangleright\;$each $a_i$ is in $\bZ[\kappa]$}
        }
    }
    $\Delta\gets0$\commentR{$\triangleright\;$gcd of resultants}
    $\mathscr{D}\gets\emptyset$\commentR{$\triangleright\;$set of minor determinants}
    \For{\textup{maximal minors} $\hat{M}$ \textup{of} $M_d$}{
        $D\gets \det\hat{M}/h(\kappa;a,b_1,b_2,b_3,b_4)$ with $a$\commentR{$\triangleright\;$see Notation \ref{not:h}} 
        \nonl and each $b_j$ maximal such that $a$ is $2n_d$-\;
        \nonl smooth, $b_2\leq 2$, $b_3,b_4\leq 1$, and $D\in\bZ[\kappa]$\;
	 \For{$\tilde{D}$ \textup{in} $\mathscr{D}$}{
        $\Delta\gets\textup{gcd}(\Delta,\textup{resultant}(D,\tilde{D}))$\;
        \If(\commentF{$\triangleright\;$typically occurs at $|\mathscr{D}|=2$}){$\Delta = 1$}{\Return{$1$}}
	}
	$\mathscr{D}\gets\mathscr{D}\cup\{D\}$
    }
    \Return{$\Delta$}
    \caption{Verify the McCullough--Wanderley conjectures over $\bF_p$ for congruence classes of primes $p$.}\label{alg}
\end{algorithm}

Sage code for Algorithm \ref{alg} is available on the author's website: \href{https://dem6.people.clemson.edu/}{\nolinkurl{dem6.people.clemson.edu}}. Note that in order to execute line 7, the Sage implementation precomputes $\Phi(x^{2m}y^{2n})$ for sufficiently large $m$ and $n$. Then each $\Phi(x^{2m}g_{d,n}f_n)$ is a linear combination of the precomputed reductions. This is faster than computing each $\Phi(x^{2m}g_{d,n}f_n)$ directly because it avoids repeated work. (Though line 7 is not the bottleneck of the algorithm either way; line 14 is.) 

The author has executed this algorithm for all prime powers $d \leq 17$ with an output of $\Delta=1$ in each case. This proves Theorems \ref{thm:main1} and \ref{thm:main2}.

\printbibliography

\end{document}